%% file: LimitOrbHiddenSymms_arxiv_v2.tex
\documentclass[11pt]{amsart}
\usepackage[margin=1.25in]{geometry}
\usepackage{amsmath}
\DeclareMathOperator\arctanh{arctanh}\usepackage{amssymb}
\usepackage{graphicx}
\usepackage{tikz}
\usepackage{color}
\usepackage[all]{xy}
\usepackage{amsthm}
\usepackage{verbatim}
\usepackage{enumitem}
\usepackage[colorlinks,linkcolor=blue,citecolor=blue,pdfstartview=FitH]{hyperref}
\usepackage{cleveref}
\usepackage{enumitem}

\usepackage{subcaption}
\usetikzlibrary{arrows,shapes,positioning}
\usetikzlibrary{decorations.markings}
\tikzstyle arrowstyle=[scale=1]
\tikzstyle directed=[postaction={decorate,decoration={markings,
    mark=at position .55 with {\arrow[arrowstyle]{stealth}}}}]
\tikzstyle ddirected=[postaction={decorate,decoration={markings,
    mark=at position .45 with {\arrow[arrowstyle]{stealth}},
    mark=at position .55 with {\arrow[arrowstyle]{stealth}}}}]
\tikzstyle reverse directed=[postaction={decorate,decoration={markings,
    mark=at position .55 with {\arrowreversed[arrowstyle]{stealth};}}}]
\tikzstyle reverse ddirected=[postaction={decorate,decoration={markings,
    mark=at position .55 with {\arrowreversed[arrowstyle]{stealth};},
    mark=at position .65 with {\arrowreversed[arrowstyle]{stealth};}}}]

\theoremstyle{plain}

\newtheorem{para}{}[section]
\newtheorem{thm}[para]{Theorem}
\newtheorem{prop}[para]{Proposition}
\newtheorem{lem}[para]{Lemma}
\newtheorem{cor}[para]{Corollary}

\newtheorem{remark}[para]{Remark}
\newtheorem*{stremark}{Remark}

\newtheorem{fact}[para]{Fact}

\newtheorem{question}[para]{Question}

\theoremstyle{remark}

\theoremstyle{definition}

\newtheorem{dfn}[para]{Definition}
\newtheorem{example}[para]{Example}

\newtheorem{conjecture}[para]{Conjecture}

    \newtheoremstyle{TheoremNum}
        {5pt}{5pt}        %% space between body and thm
        {\itshape}                      %%% Thm body font
        {}                              %%% Indent amount (empty = no indent)
        {\bfseries}                     %%% Thm head font
        {.}                             %%% Punctuation after thm head
        { }                             %%% Space after thm head
        {\thmname{#1}\thmnote{ \bfseries #3}}%%% Thm head spec
    \theoremstyle{TheoremNum}
    \newtheorem{thmn}{Theorem}

\newcommand{\co}{\colon\thinspace}

\newcommand{\C}{\mathbb{C}}
\renewcommand{\H}{\mathbb{H}}
\newcommand{\Q}{\mathbb{Q}}
\newcommand{\R}{\mathbb{R}}
\newcommand{\Z}{\mathbb{Z}}

\newcommand{\Sth}{\mathbb{S}^3}

\newcommand{\orb}{\mathcal{O}}

\newcommand{\orbQ}{\mathcal{Q}}

\newcommand{\thick}{\geq \epsilon}

\newcommand{\thickDN}{\geq \epsilon/d_{N}}
\newcommand{\ZZ}{\mathbb{Z}}

\definecolor{bettergreen}{rgb}{0,0.6,0.4}
\definecolor{purple}{rgb}{0.4,0,0.6}

\newcommand{\nbd}{n}
\newcommand{\del}{\partial}
\newcommand{\gbar}{\bar{\gamma}}
\newcommand{\Qf}{Q_{k\gbar}}

\newcommand{\cone}{\mathrm{cone}}

\newcommand\restr[2]{{% we make the whole thing an ordinary symbol
  \left.\kern-\nulldelimiterspace % automatically resize the bar with \right
  #1 % the function
  \vphantom{\big|} % pretend it's a little taller at normal size
  \right|_{#2} % this is the delimiter
  }}

\newcommand\vol{\mathrm{vol}}

\author[E. Chesebro]{Eric Chesebro}
 \address{Department of Mathematical Sciences, University of Montana, Missoula, MT}
\email[]{Eric.Chesebro@mso.umt.edu}

\author[J. Deblois]{Jason Deblois}
 \address{Department of Mathematics, University of Pittsburgh, Pittsburgh, PA}
\email[]{jdeblois@pitt.edu}

\author[N. Hoffman]{Neil R Hoffman}
 \address{Department of Mathematics, Oklahoma State University, Stillwater, OK}
\email[]{neil.r.hoffman@okstate.edu}

\author[C. Millichap]{Christian Millichap}
 \address{Department of Mathematics, Furman University, Greenville, SC}
\email[]{christian.millichap@furman.edu}

\author[P. Mondal]{Priyadip Mondal}
 \address{Department of Mathematics, University of Pittsburgh, Pittsburgh, PA}
\email[]{prm50@pitt.edu}

\author[W. Worden]{William Worden}
 \address{Department of Mathematics, Rice University, Houston, TX}
\email[]{william.worden@rice.edu}

\begin{document}

\title[Dehn surgery and hyperbolic knot complements without hidden symms]{Dehn surgery and hyperbolic knot complements without hidden symmetries}

\maketitle

\begin{abstract}
Neumann and Reid conjecture that there are exactly three knot complements which admit hidden symmetries. This paper establishes several results that provide evidence for the conjecture. Our main technical tools
provide obstructions to having infinitely many fillings of a cusped manifold produce knot complements admitting hidden symmetries. 
Applying these tools, we show for any two-bridge link complement, at most finitely many fillings of one cusp can be covered by knot complements admitting hidden symmetries. 
We also show that the figure-eight knot complement is the unique knot complement with volume less than $6v_0 \approx 6.0896496$
that admits hidden symmetries. We then conclude with two independent proofs that among hyperbolic knot complements only the figure-eight knot complement can admit hidden symmetries and cover a filling of the two-bridge link complement $\Sth\setminus 6^2_2$. Each of these proofs shows that the technical tools established earlier can be made effective. 
\end{abstract}

\input{introduction_arxiv}

\input{geometric_isolation_arxiv_v2}

\input{cautionary_example_geometric_arxiv}

\input{effectivization_arxiv}

\input{sixtwotwo}

\input{deformation_arxiv}

%%%%%%%%%%%%%%%%
\bibliographystyle{plain}
%\bibliography{hidden_bib}   
\bibliography{LimitOrbHiddenSymms_arxiv_v2.bbl}
\end{document}

%% file: introduction_arxiv.tex
\section{Introduction}\label{sec:intro}

A  \textit{hidden symmetry} of a manifold (or orbifold) $M$ is a homeomorphism between finite-sheeted covers of $M$ that does not descend to $M$. Hidden symmetries play an important role in the study of locally symmetric spaces, where work of Margulis shows that arithmetic rank-one lattice quotients are characterized by the property of having infinitely many (up to a natural equivalence) \cite{Margulis1991}. This paper focuses on hidden symmetries of hyperbolic $3$-orbifolds, which are intimately related to the study of their covering maps and commensurability classes.  For instance, a hyperbolic knot complement in $\Sth$ with no hidden symmetries has at most two other such knot complements in its commensurability class \cite[Theorem 1.2]{BBoCWaGT}.

In fact, hidden symmetries seem to highlight some fundamental difference between complements of knots in $\Sth$ and other non-compact finite volume hyperbolic 3-orbifolds. (Moving forward, we reserve the term ``knot/link complement'' for the complements of knots/links \textit{in $\Sth$}.)  While it is not uncommon for members of the latter class to have hidden symmetries, only three hyperbolic knot complements are known to have them: that of the figure-eight, which is the only arithmetic knot complement \cite{ReidFig8}, and those of the two dodecahedral knots of Aitchison and Rubinstein \cite{AiRu}.  Moreover, there are (infinitely) many without hidden symmetries. For this and other reasons, in 1995 Neumann and Reid proposed:

\begin{conjecture}[Problem 3.64(A), \cite{KirbyList}]\label{NR conjecture} The only hyperbolic knot complements without hidden symmetries are the three already known.\end{conjecture}

The work of   Neumann and Reid  \cite[Proposition 9.1]{NeumReid}  provides one of the most striking pieces of evidence for this conjecture and the foundation for much subsequent work on it. This asserts that a hyperbolic knot complement admits hidden symmetries if and only if it covers a rigid-cusped orbifold (see \Cref{(non)rigid}). As an illustration of how exceptional such covers are, now-standard computational methods (SnapPy \cite{SnapPy} and Sage \cite{sagemath}) can be used to test the roughly 330,000 knot complements up to fifteen crossings, showing numerically that the figure-eight knot complement is the only one of these to have one \cite{Dunfield}. (The dodecahedral knots have crossing number $20$.)

 Applying this condition to infinite families of knot complements has proven more challenging, although there have been significant successes with, for instance, the two-bridge knots \cite{ReidWalsh} and other more specialized families \cite{Hoffman_3comm}, \cite{Hoffman_hidden},  \cite{MacMat}, \cite{Millichap}. More recently, two independent groups of the present authors turned attention to ruling out hidden symmetries among {certain large} families of knot complements, giving more broadly applicable criteria for doing so in \cite{CheDeMonda} and \cite{HMW19}. Motivated by these results, a subset of the authors in \cite{CheDeMonda} proposed the following ``generic'' (and weaker) version of \Cref{NR conjecture}:

\begin{conjecture}\label{main_conjecture}  For any $R>0$, at most finitely many hyperbolic knot complements have hidden symmetries and volume less than $R$.\end{conjecture}

As we wrote in \cite{CheDeMonda}, this conjecture can be reformulated in terms of Dehn surgery using the J\o rgensen--Thurston theory: it fails if and only if a sequence of hyperbolic knot complements with hidden symmetries converges to a link complement (cf.~\Cref{prop1}).  Here and in the rest of the paper, we say a sequence $\{M_n\}$ of hyperbolic manifolds or orbifolds \textit{converges} to a manifold $N$ if each $M_n$ is obtained from $N$ by hyperbolic Dehn filling, and all filling coefficients approach $\infty$ as $n\to\infty$. We call this ``convergence'' because it is equivalent to the sequence $\{M_n\}$ \textit{converging geometrically} to $N$, in the sense defined by Thurston \cite{Th_notes}.

Our aims in this paper are twofold. The first half is devoted to \Cref{main_conjecture}. Our main tool for analyzing this conjecture is \Cref{main_theorem}, which describes how a cover to a rigid-cusped orbifold relates to the limit of a convergent sequence of one-cusped manifolds that each have this property. We apply this result in \Cref{corollary2} to prove an important case of the conjecture, that a two-bridge link complement cannot be the limit of a counterexample sequence. On the other hand, \Cref{sec:CA} exhibits some ``cautionary examples'' which limit the conjecture's scope and hint at directions to search for counterexamples.

The paper's second half focuses on ``effectivization'': placing \textit{explicit} bounds on the number of members that may have hidden symmetries within a given family of knot complements. Such a bound rules out the given family from producing counterexamples to \Cref{main_conjecture}, but it also makes useful progress toward evaluating \Cref{NR conjecture} there. In \Cref{sec:effectivization} we rule out counterexamples to \Cref{NR conjecture} among \textit{all} low-volume knot complements, with \Cref{thm:volume_bound_gen}, and among all but an explicitly constrained list of fillings of one cusp of certain low-volume two-component link complements, in \Cref{earem}. Furthermore, via \Cref{cor:agol_bound}, we can place a universal upper bound of $84$ on the number of possible fillings of these two-component links that could be covered by a knot complement admitting hidden symmetries. In \Cref{sec:sixtwotwo}, we enumerate this list for the $6^2_2$-link complement and prove the complete classification \Cref{sixtwotwo} of those covered by a knot complement with hidden symmetries. We finish in \Cref{defvar} by giving a deformation variety-based proof of the same result.

In the rest of this introduction we state our main results and expand on their context. For the sake of readability we state a streamlined version of \Cref{main_theorem} below.

\begin{thm}\label{main_theorem_intro}
Suppose $N$ is a complete, multi-cusped hyperbolic $3$-manifold of finite volume and $\{M_n\}$ is a sequence of one-cusped manifolds converging to $N$. If, for each $n$, there is a covering map $\phi_n\co M_n\to\orbQ_n$ to a rigid-cusped orbifold $\orbQ_n$, then there is a multi-cusped orbifold $\orbQ$ with a single rigid cusp and a covering map $\phi\co N\to\orbQ$. Moreover, we may choose $\orbQ$ and pass to a subsequence so that the $\orbQ_n$ converge to $\orbQ$ and $\phi$ is a restriction of
$\phi_n$ for every $n$.
\end{thm}

The full statement of \Cref{main_theorem} expands  on the last sentence above. Specifically, for any $n$, $N$ can be taken as a submanifold of $M_n$ as the complement of the set of core geodesics of the filling tori, and similarly for $\orbQ$ in $\orbQ_n$. \Cref{main_theorem} is an extension of Proposition 2.6 of \cite{HMW19}, by three of the current authors, and is proven along similar lines. It is a powerful tool for showing that converging families of hyperbolic knot complements do not admit hidden symmetries. To start, we  can apply \Cref{main_theorem_intro}  to prove the following:

\begin{thm}\label{corollary1} Suppose that a sequence of knot complements with hidden symmetries converges to a hyperbolic 3-manifold $N$.  Then $N$ is a link complement with hidden symmetries.
\end{thm}  

Millichap--Worden showed that no hyperbolic two-bridge link complement has hidden symmetries, except the four arithmetic ones \cite{MilliWord}.  Combining their result with \Cref{corollary1} immediately implies that no non-arithmetic two-bridge link complement is the limit of a sequence of knot complements with hidden symmetries. This is significant because both components of any two-component two-bridge link are unknotted, so each such link \textit{is} the limit of a sequence of knot complements.

In fact, by working harder here we obtain the result below, a significantly stronger consequence of \Cref{main_theorem_intro} and \cite{MilliWord}. Note first that its hypotheses {apply to both} arithmetic {and non-arithmetic} two-bridge link complements. 

\begin{thm}\label{corollary2} If $N$ is the complement of a two-component hyperbolic two-bridge link, then at most finitely many orbifolds obtained by filling one cusp of $N$ are covered by knot complements with hidden symmetries.
\end{thm}

Here, for a filling coefficient $(p,q)$ such that $\gcd(p,q) > 1$, $(p,q)$-filling produces an orbifold with the core of the filling solid torus as its singular locus. Because there are many branched covers $\Sth\to\Sth$ that branch over an unknot, many orbifold fillings of one cusp of a hyperbolic two-bridge link complement are covered by a hyperbolic knot complement; see \Cref{prop:linking_number}. \Cref{corollary2} thus covers a much larger class of knots than only those produced by \textit{manifold} filling.

We now turn to our ``cautionary examples''  of \Cref{sec:CA}. We first consider \Cref{two_cusped_example}  which comes from \cite{NR_rigidity} and describes a $2$-cusped hyperbolic $3$-manifold  and a sequence of fillings of this manifold that satisfy the hypotheses and conclusion of \Cref{main_theorem}. Building off of this, in  \Cref{ex:three_cusped_example} we use a careful analysis of orbifolds covered by the Borromean rings complement $N$ to describe a sequence of one-cusped hyperbolic $3$-manifolds converging to $N$, each of which covers an orbifold with a rigid cusp. The $M_n$ are not knot complements, but they do satisfy the hypotheses and conclusion of \Cref{main_theorem}. They show that \Cref{main_conjecture} does not hold in the general setting of $1$-cusped hyperbolic $3$-manifolds. Both of these examples highlight interesting phenomenon in the context of geometric isolation described in \Cref{isooo}.

Our first effectivization result, from \Cref{sec:effectivization}, implies in particular that there is no counterexample to \Cref{NR conjecture} among manifolds with volume at most $6v_0$. Here and henceforth, $v_0$ denotes the volume of a regular ideal hyperbolic tetrahedron ($\approx 1.0149416$). The proof of  this results plays off results of Hoffman that give lower bounds on the degree of covers to rigid-cusped orbifolds against Adams' work classifying the lowest-volume (rigid-) cusped orbifolds. 

\newcommand\VolumeBoundGen{If a manifold $M$ has volume less than or equal to $6v_0$ and is covered by a knot complement admitting hidden symmetries, then $M$ is the figure-eight knot complement (and the cover is trivial).}

\theoremstyle{plain}
\newtheorem*{volboundthm}{Theorem \ref{thm:volume_bound_gen}}
\begin{volboundthm}\VolumeBoundGen\end{volboundthm}

This is even more general than what is needed for a lower bound on counterexamples to \Cref{NR conjecture}, in that it applies as well to knot complements \textit{covering} manifolds of low volume. Covers of this sort by knot complements with arbitrarily large volume may arise from lens space fillings on the cusp of a two-component link complement corresponding to an unknotted component, as we describe in \Cref{prop:linking_number}. Our next main theorem extends this line of inquiry to \textit{orbifold} fillings.

\newcommand\TheoremEarem{Suppose $L = K\sqcup K'$ is a two-component link in $\Sth$, with $K'$ unknotted, such that $N\doteq\Sth \setminus L$ is hyperbolic with volume at most $6v_0$. Giving the $K'$-cusp of $N$ the standard homological framing, if its $(p,q)$-filling is covered by a hyperbolic knot complement with hidden symmetries then:\begin{enumerate}[label=(\roman*)]
	\item the $(p,q)$-filling slope has normalized length {at most} $7.5832$;
	\item\label{thm:earem_part_ii} $p$ is relatively prime to the linking number $\ell$ between $K$ and $K'$; and
	\item $p$ and $q$ have greatest common divisor in $\{1,2,4\}$.\end{enumerate} }

\theoremstyle{plain}
\newtheorem*{earemtheorem}{Theorem \ref{earem}}
\begin{earemtheorem}\TheoremEarem\end{earemtheorem}

The ``standard homological framing'' used here is described in \Cref{whatsadiorama?}, and the ``normalized length'' is in the sense of Hodgson--Kerckhoff \cite[pp.~1036--37]{HK2008shape}.  The proof of \Cref{earem} plays off a careful analysis of covers to rigid-cusped orbifolds against results of \cite{HK2008shape} that constrain the shape of the filling torus after hyperbolic Dehn filling along long slopes.

For a link $L$ satisfying its hypotheses, \Cref{earem} leaves only a short list of slopes to check in order to completely determine which fillings of the $K'$-cusp of $\Sth\setminus L$ are covered by knot complements with hidden symmetries.   In \Cref{sec:sixtwotwo} we carry out the process of identifying and checking these slopes for the two-bridge link $L = 6^2_2$ (notation as in the Rolfsen tables). We obtain:

\newcommand\SixTwoTwo{The only hyperbolic knot complement with hidden symmetries that covers a filling of $N = \Sth\setminus 6_2^2$ is the figure-eight knot complement, which covers the $(2,0)$-filling of one cusp of $N$.}

\begin{thm}
\label{sixtwotwo}\SixTwoTwo
\end{thm}

This is  the ``simplest new example'' in the sense that there is only one two-component link with hyperbolic complement and fewer crossings, the Whitehead link $5^2_1$. $\Sth\setminus5^2_1$ has been known since Neumann-Reid's work \cite[Section 6]{NeumReid} to have no fillings covered by a knot complement with hidden symmetries except one: the figure-eight knot complement.
 
In \Cref{defvar} we give an alternative proof of \Cref{sixtwotwo} that echoes the analogous proof for $\Sth\setminus5^2_1$ in \cite{NeumReid}, using properties of the deformation variety of $\Sth\setminus6^2_2$. While this proof is very specific to $\Sth \setminus 6_{2}^{2}$, it has the advantage of providing additional algebraic information about the orbifolds produced by filling one of its cusps: for instance \Cref{arith} classifies the arithmetic ones, and \Cref{trfld} and \Cref{gareth} relate their cusp, trace and invariant trace fields. Computational methods can produce the former result (see \Cref{sub:covers_of_man_fillings}), but the latter two are not obviously accessible to them.

\subsection{Acknowledgements}
The authors wish to thank the American Institute of Mathematics, Oklahoma State University and Rice University for hosting different subsets of them over the course of this project. 
The third author was partially supported by grant from the
Simons Foundation (\#524123 to Neil R. Hoffman)

%% file: geometric_isolation_arxiv_v2.tex
\section{Geometric Convergence and Orbifold Covers} 
\label{sec:BandP}

\subsection{Orbifolds and Dehn filling}\label{whatsadiorama?} 

For a cusp $c$ of an orientable finite volume hyperbolic 3-manifold $N$, let $T$ be a torus cross-section of $c$.  A \textit{primitive slope} on $T$ is the unoriented isotopy class of a non-trivial simple closed curve in $T$. The set of primitive slopes on $T$ corresponds bijectively to the set of pairs $\{ \pm \alpha\}$ of primitive elements in $H_1(T;\Z)$. More generally, a \emph{slope} on $T$ corresponds to a pair $\pm\alpha=\pm k\beta$ of elements of $H_1(T;\Z)$, where $\beta$ is primitive, $k\in\Z_{>0}$. When $k=1$, so that $\alpha=\beta$ is primitive, the \emph{Dehn filling of $c$ along $\alpha$} is the manifold $N(\alpha)$ obtained by cutting $N$ along $T$, discarding the cusp $c$, and replacing it with a solid torus $J$ so that the meridian of $J$ is glued to a curve in $T$ isotopic to $\beta$ (see \Cref{fig:dehn_fill}).  If $k>1$ then we proceed similarly, except that in this case $J$ is instead an order-$k$ orbi-torus: a solid torus in which the core curve is marked by $k$-torsion, so that the meridian disk is a cone-disk with an order-$k$ cone point. In this case $N(\alpha)$ is an orbifold, with singular locus the core curve of $J$, and the image of $\beta$ in $N(\alpha)$ bounds an order-$k$ cone-disk.

We can identify the collection of slopes on $T$ with $\mathbb{Z}\oplus\mathbb{Z} - \{(0,0)\}/(p,q)\sim-(p,q)$ by fixing a \textit{framing}, ie.~ a basis for $H_1(T;\Z)$. For any framing, a slope $\alpha$ identified with the class of some $(p,q)$ is primitive if and only if $\gcd(p,q) = 1$. More generally, after Dehn filling along $\alpha$ the core of the filling orbi-torus is marked by $k$-torsion, where $k=\gcd(p,q)$. When a framing is understood, we will also refer to the Dehn filling along $\alpha$ as the \textit{$(p,q)$-filling of $N$ at $c$.}

When $N$ is marked (implicitly or explicitly) by a homeomorphism $h\co N\to\Sth\setminus L$ for a link $L$, we will call $c$ the \textit{$K$-cusp}, where $K$ is the component of $L$ such that $h(c)$ is contained in a regular neighborhood $\nbd(K)$ of $K$. In this case the \textit{standard homological framing} on $T$ is the basis $\{[\mu],[\lambda]\}$ for $H_1(T;\Z)$, where $\mu$ is a \textit{meridian} and $\lambda$ is a \textit{longitude}. These primitive slopes are characterized as follows: $h(\mu)$ bounds a disk in $\nbd(K)$, and $h_*([\lambda])$ lies in the kernel of the inclusion-induced map $H_1(h(T))\to H_1(\Sth \setminus  K)$.

More generally, we will need a notion of Dehn filling when $N$ is a cusped orientable hyperbolic $3$-orbifold. For such an $N$, the cusp cross-section is one of the five Euclidean 2-orbifolds: $T$, $S^2(2,2,2,2)$, $S^2(3,3,3)$, $S^2(2,3,6)$, and $S^2(2,4,4)$. The last four have underlying topological space $\mathbb{S}^2$ and cone points prescribed by the ordered tuples.

\begin{dfn}\label{(non)rigid} A cusp of a complete, orientable hyperbolic $3$-orbifold is \textit{rigid} if its cross-section is one of $S^2(3,3,3)$, $S^2(2,3,6)$, or $S^2(2,4,4)$, and \textit{non-rigid} otherwise.\end{dfn}

Rigid cusps are so-named because their Euclidean metrics are unique up to similarity. No solid torus-quotient is bounded by a rigid Euclidean orbifold, so we can't fill a rigid cusp.

For an orbifold's cusp with a torus cross-section $T$, Dehn filling is defined exactly as in the manifold case. When the cusp cross-section is $S=S^2(2,2,2,2)$, we can again identify a filling slope with an element $\alpha=(p,q) \in \Z\times \Z$. Namely, if $\alpha = k\beta$ is a slope on $T$, then under the involution $h:T\to S$ a simple closed curve representing $\beta$ (and disjoint from the fixed points of $h$) maps to a simple closed curve $\beta'$ on $S$. In this context, the Dehn filling of $c$ along $\alpha$ is the orbifold $N(\alpha)$ obtained by cutting along $S$, removing the cusp $c$, and gluing in an orbi-pillow $P$ (shown in \Cref{fig:orbi-fill}) in such a way that the curve on $\del P$ that bounds an order-$k$ cone-disk is identified with $\beta'\subset S$. Alternatively, one can think of this as gluing to $S$ an orbi-tangle of slope $\beta$, in which a closed curve isotopic to $\beta'$ bounds an order-$k$ cone-disk.

The \emph{norm} of a filling $\alpha=(p,q)$ is $|\alpha|=\sqrt{p^2+q^2}$.  We adopt the usual convention that $N(\infty)=N$ and $|\infty|=\infty$.

\begin{figure}[h]
	\begin{subfigure}{.49\textwidth}
 		\centering
 		\includegraphics[scale=1.4]{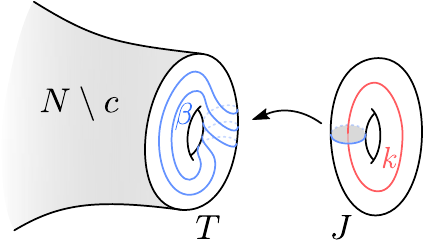}
 		\caption{}
   		\label{fig:dehn_fill}
	\end{subfigure}
	\begin{subfigure}{.49\textwidth}
 		\centering
 		\includegraphics[scale=1.4]{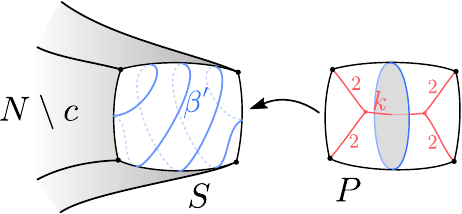}
 		\caption{}
   		\label{fig:orbi-fill}
	\end{subfigure}	
	\caption{Dehn filling torus and pillowcase cusps along non-primitive curves.}
	\label{fig:fillings_back}
\end{figure}
%*** add more ***  

Finally, if $N$ has cusps $\{ c_1, \ldots, c_k\}$, we use a \emph{multi-slope vector} $\bar{\alpha}$ where the $i^\text{th}$ cusp of $N$ is filled according to the $i^\text{th}$ coordinate of $\bar{\alpha}$.  For a sequence of  multi-slope vectors $\{ \bar{\alpha}_n\}$, we write $| \bar{\alpha}_n | \to \infty$ if, for each $i$, the norms of the $i^\text{th}$ coordinates go to infinity.  Here, we say that a sequence of hyperbolic orbifolds $\{ M_n \}$ \emph{converge} to an orbifold $N$ if there is a sequence of multi-slope vectors $\{\bar{\alpha}_n\}$ for which $M_n = N(\bar{\alpha}_n)$ and $| \bar{\alpha}_n | \to \infty$.  In this case, the geometric limit of the $M_n$'s is $N$.

We can now prove an equivalent version of \Cref{main_conjecture}, which was discussed in the introduction. 

\begin{prop}\label{prop1} 
\Cref{main_conjecture} is false if and only if there is a sequence of distinct hyperbolic knot complements, each of which is a cover of a rigid cusped orbifold, so that the sequence converges to the complement of a link in $S^3$.
\end{prop}

\begin{proof} If the conjecture holds then certainly no single hyperbolic link complement produces infinitely many hyperbolic knot complements with hidden symmetries by Dehn filling, since Dehn filling reduces volume \cite[Th.~6.5.6]{Th_notes}. 

Now suppose that the conjecture fails.  {Then there is some $R>0$ and a sequence $\{ M_n\}_1^\infty$ of distinct hyperbolic knot complements, each of which has hidden symmetries and volume less than $R$. }  By Proposition 9.1 of \cite{NeumReid}, each $M_n$ covers an orbifold with a rigid cusp. By \cite[Theorem 5.11.2]{Th_notes}, there is a link  {complement $N_0$ }
such that each $M_n$ is obtained by Dehn filling {on $N_0$}. 
{Since $N_0$ has only finitely many cusps, we may pass to a subsequence so that there is a fixed cusp of $N_0$ which remains unfilled in every $M_n$.  Beginning with this unfilled cusp, number the cusps of $N_0$ from $1$ to $m$ and fix a basis for the first homology of each cusp to obtain a sequence of multi-slope vectors $\{ \bar{\alpha}_n \}$ with $M_n = N(\bar{\alpha}_n)$.}  It is not necessarily true that $|\bar{\alpha}_n | \to\infty$, but by repeatedly subsequencing we may assume that, for each $i$, { the sequence of $i^\text{th}$ coordinates either goes to infinity or is constant.  Let $N$ be the manifold obtained by doing all constant fillings on $N_0$.  We have now arranged that $\{ M_n \}$ converge to $N$.} Since $N$ embeds in every $M_n$, with complement a disjoint union of circles, $N$ is itself a link complement.\end{proof}

\subsection{$(\epsilon,d_N)$-twisted fillings.}\label{sub:edn}

Suppose $\orbQ = \mathbb{H}^3/\Gamma$ is a complete hyperbolic 3-orbifold.  Following \cite[\S 2]{DM}, for $\epsilon >0$, we define the \textit{$\epsilon$-thin} part $\orbQ_{<\epsilon}$ of $\orbQ$ to be the set of $x\in \orbQ$ such that for some $\tilde{x}\in\mathbb{H}^3$ projecting to $x$ and some $\gamma\in\Gamma$ of infinite order, $d(\tilde{x},\gamma\tilde{x})<\epsilon$.  The $\epsilon$-thick part of $\orbQ$ is $\orbQ_{\ge \epsilon}=\orbQ\setminus \orbQ_{<\epsilon}$.  If $\epsilon$ is chosen to be smaller than the three-dimensional Margulis constant, Corollary 2.2 of \cite{DM} states that every non-compact component of $\orbQ_{<\epsilon}$ is isometric to the quotient of a horoball $B$ by the action of its stabilizer in $\Gamma$. Such a component is a cusp, and it is \textit{rigid} if its stabilizer subgroup is a Euclidean triangle group or its index-two orientation-preserving subgroup.

There are several ways to understand what it means for a filled manifold to ``geometrically approximate'' the corresponding unfilled manifold.  Here, we want a property which assures that if a filling $M$ of $N$ covers an orbifold $\orbQ$, then there is a natural restriction of that cover between the $\epsilon$-thick part of $M$ and the $\epsilon$-thick part of $\orbQ$. 

Let $N$ be a finite-volume, orientable, hyperbolic 3-manifold with cusps ${c_1, . . . , c_k}$ and take $\epsilon \in (0, 3.45)$.  Let $d_N$ be the maximum possible degree of a cover to a non-arithmetic orbifold, and observe that $d_N \leq 4\frac{vol(N)}{v_0}$; see \cite[Proposition 6.2]{HMW19}.  Recall the following definition from \cite{HMW19}.

\begin{dfn}\label{def:edHTN}
{A cusped hyperbolic Dehn-filling $M$ of a hyperbolic 3-manifold $N$ is an \emph{$(\epsilon,d_N)$-twisted filling} of $N$}, if the following spaces are homeomorphic
\[ N \setminus n\left(\bigcup c_i\right) \ \cong \ N_{\thick} \ \cong \ M_{\thick} \ \cong \ M_{\thickDN}. \]
{In this case, we often refer to $M$ as \emph{$(\epsilon,d_N)$-twisted} (relative to $N$).}

\end{dfn}

The assumption that $\epsilon < 3.45$ implies that $(\epsilon,d_N)$-twisted fillings are non-arithmetic (see \cite[Proposition 6.2]{HMW19}). This fact is important for \Cref{prop:edn}, which is a generalization of \cite[Proposition 2.6]{HMW19}. The proof of our version here is largely identical so we only point out the minor differences below. We appeal to the notation of \cite{HMW19}, letting $\gamma_{i}\subset M$ be the core geodesic of the solid torus introduced by filling the $i^{th}$ cusp $c_i$ of $N$, for $i=1, \ldots, k$.    
In this version of the proposition we allow the $i^{th}$ cusp to remain unfilled, in which case 
 $\gamma_{i}$ is an empty set of points.

\begin{prop}\label{prop:edn}

Let $N$ be a $k$-cusped hyperbolic manifold and let $\bar{\alpha}=(\alpha_1,\dots,\alpha_k)$ be a multislope. %, possibly with some $\alpha_i=\infty$.
\begin{enumerate}[label=(\roman*)]
\item For any sufficiently small $\epsilon>0$, there is a real number $C$ such that if all $\alpha_i$ satisfy $|\alpha_i|>C$, then $N(\bar{\alpha})$ is an $(\epsilon,d_N)$-twisted filling of $N$.
\item \label{prop:edn_part_1} If $N(\bar{\alpha})$  is an $(\epsilon,d_N)$-twisted filling of $N$
and   $\phi \colon   N(\bar{\alpha}) \rightarrow \orb$ is an orbifold cover, then there is an orbifold cover $\widehat{\phi} \co  N \rightarrow \orbQ$, where $\orbQ \cong \orb \setminus \sqcup_{i=1}^{k} \phi(\gamma_{i})$. 
\end{enumerate}
\end{prop}

\begin{proof}
	The main difference between the proposition above and that of \cite{HMW19} is that here $N$ can be any cusped hyperbolic 3-manifold. The assumption that $\epsilon<3.45$ from  \Cref{def:edHTN} assures that $(\epsilon,d_N)$ fillings are non-arithmetic, so the proof in \cite{HMW19} goes through without change by taking $C$ to be the maximum of the $k_i$'s in that proof. The above proposition is more general than that of \cite{HMW19}, because we require that any sufficiently small $\epsilon$ will work for (i).  However this change requires no alteration to the proof as given in \cite{HMW19}. Although stated there, we emphasize that the cusps of $\orbQ \cong  \orb \setminus \sqcup_{i=1}^{k} \phi(\gamma_{i})$ obtained from drilling are all non-rigid.
\end{proof}

\begin{remark}
Although we have defined the norm of $\alpha_i=(p_i,q_i)$ to be $\sqrt{p_i^2+q_i^2}$, the statement of \Cref{prop:edn} is also valid if we measure the length of $\alpha_i$ in a maximal horoball neighborhood, or if we use its normalized length. We appeal to the third notion in \Cref{sec:effectivization}.  
\end{remark}

{With \Cref{prop:edn} in hand, we can prove  \Cref{main_theorem}, which is our main tool for analyzing hidden symmetries of hyperbolic knot complements in the remainder of this section. \Cref{main_theorem}, while more technical, implies \Cref{main_theorem_intro} which was stated in the introduction. This gives a more geometric description of the restriction of the $\phi_n$ covering maps. In particular, \Cref{main_theorem} develops a correspondence between how  geodesics coming from high parameter fillings behave under the covering map $\phi_n$ on one hand, and how cusps behave under a restricted version of such covering maps  on the other. A feature of this result worth noting is that geometric data associated to the limiting manifold can force restrictions on the cover $\phi_n$ for sufficiently high parameter elements of the limiting sequences. Furthermore, \cite{FPS19} provides tools for making such arguments effective (see for example \cite[\S 6.1]{HMW19}). 
  As a result, this theorem gives us a powerful tool to start investigating \Cref{main_conjecture} in the context of \Cref{prop1}.}

\begin{thm} \label{main_theorem}
Suppose that $N$ is a $k$-cusped hyperbolic 3-manifold and $\{ \bar{\alpha}_n\}$ is a sequence of multi-slope vectors for $N$ such that $|\bar{\alpha}_n| \to \infty$ and every $N(\bar{\alpha}_n)$ is a hyperbolic orbifold with a single cusp.  Suppose further that, for each $n$, $\phi_n \co N(\bar{\alpha}_n) \to \orbQ_n$ is a covering map to an orbifold with a rigid cusp.  Then, after passing to a subsequence, we can guarantee the following:
\begin{itemize}
\item[(i)]	There is an orbifold cover $\phi \co N \to \orbQ$, where $\orbQ$ has a unique rigid cusp $c$ and $\phi^{-1}(c)$ is connected.  Moreover, each $\orbQ_n$ is obtained from $\orbQ$ by orbifold Dehn filling on all cusps of $\orbQ$ but $c$. %\neil{If $Q$ and $N$ both have $n$ cusps}, the cusp $p^{-1}(c)$ of $N$ is weakly geometrically isolated.
\item[(ii)] If $\epsilon>0$ is small enough, then the following diagram commutes for every $n$
$$
\xymatrix{
N \ar[d]_{\phi} \ar[r] & (N(\bar{\alpha}_n))_{>\epsilon} \ar[d]^{\bar{\phi}_{n}}\\
\orbQ \ar[r] & (\orbQ_{n})_{>\epsilon}\\
}
$$
where the horizontal maps are homeomorphisms and $\bar{\phi}_n$ is obtained by restricting $\phi_n$ and then post-composing with a homeomorphism $(\orbQ_{n})_{>\epsilon/deg(\phi)} \to (\orbQ_{n})_{>\epsilon}$.
\end{itemize}

\end{thm}

The proof of \Cref{main_theorem} will follow from \Cref{prop:edn} with a bit of additional work. 

\begin{proof}
By \Cref{prop:edn}(i), we may pass to a subsequence and select an $\epsilon >0$ such that every $N(\bar{\alpha}_n)$ is an $(\epsilon, d_N)$-twisted filling of $N$.  Also, as in the proof of \Cref{prop1}, we may pass to a further subsequence and index the cusps of $N$ so that the first cusp of $N$ remains unfilled in every $N(\bar{\alpha}_n)$.

Let $\gamma_{i,n}$ be the core geodesic in $N(\bar{\alpha}_n)$ of the filling solid torus for the $i^{th}$ cusp of $N$. By \Cref{prop:edn}(ii), there are orbifold covers $\widehat{\phi}_n:N\to  \orbQ_{n}\setminus \sqcup_{i=2}^k \phi(\gamma_{i,n})$. By construction, these covers are the restrictions of $\phi_{n}$ to $N(\bar{\alpha}_n)\setminus \sqcup_{i=2}^k \gamma_{i,n}$, post-composed with a homeomorphism (see proof of \cite[Prop. 2.6]{HMW19}).  {The cover $\bar{\phi}_n$ described in $(ii)$ is then obtained from $\widehat{\phi}_n$ by composing with the homeomorphisms $N\xrightarrow{\cong} (N(\bar{\alpha}_n))_{>\epsilon}$ and $(\orbQ_{n}\setminus \sqcup_{i=2}^k \phi(\gamma_{i,n})) \xrightarrow{\cong} (\orbQ_n)_{>\epsilon}$. To establish existence of the cover $\phi$, we observe that there are only finitely many minimal orbifolds covered by $N$ (see for example \cite[Theorems 6.5 and 6.6]{DM} or \cite[$\S$ 6.6]{Th_notes}), so by again passing to a subsequence, we may assume that all the orbifolds $\orbQ_{n}\setminus \sqcup_{i=2}^k \phi(\gamma_{i,n})$ are homeomorphic to a fixed orbifold $\orbQ$.  Because there are only finitely many covers $N \to \orbQ$, we may subsequence again so that every cover $\widehat{\phi}_n$ induces the same cover $\phi:N\to \orbQ$.  That the diagram in $(ii)$ commutes follows from the construction of $\phi$ and $\bar{\phi}_n$, as does the description of $\bar{\phi}_n$ as a restriction of $\phi_n$ composed with a homeomorphism.}

The orbifolds $\orbQ_{n}$ converge geometrically to $\orbQ$, so $\orbQ$ has a single rigid cusp. Because the cover $\phi$ is (up to homeomorphism) constructed by restricting the covers $\phi_{n}$, the rigid cusp of $\orbQ$ can only be covered by one cusp of $N$, so \Cref{main_theorem}(i) is established. 
\end{proof}

We are now ready to prove \Cref{corollary1}. 
Before we proceed, it is worth remarking that a link complement with hidden symmetries need not cover an orbifold with a rigid cusp. In fact, any arithmetic link complement with invariant trace field other than $\Q(\sqrt{-3})$ and $\Q(i)$ furnishes such an example (see the `arithmetic zoo' of \cite[Chapter 13]{MR03}).  However, as we see in the proof below, this is the case under the hypotheses of \Cref{corollary1}.

\begin{proof}[Proof of \Cref{corollary1}] Suppose that $N$ is a hyperbolic 3-manifold and $\{ \bar{\alpha}_n \}$ is a sequence of multi-slope vectors such that $|\bar{\alpha}_n| \to \infty$ and, for every $n$,  $N(\bar{\alpha}_n)$ is a knot complement with hidden symmetries.  As in the proof of \Cref{prop1}, $N$ is a link complement. Also, for each $n$, we have a cover $N(\bar{\alpha}_n) \to \orbQ_n$ to an orbifold with a rigid cusp so we may choose an $\epsilon >0$ and assume that conditions (i) and (ii) of \Cref{main_theorem} hold. 

In particular we have a cover $\phi \co N \to \orbQ$ to an orbifold with a rigid cusp.  We also have a commutative diagram of induced homomorphisms on orbifold fundamental groups
$$
\xymatrix{
\pi_1(N) \ar[d] \ar[r] & \pi_1(N(\bar{\alpha}_n))\ar[d]\\
\pi_1^{orb}(\orbQ) \ar[r] & \pi_1^{orb}(\orbQ_{n})\\
}
$$
where the vertical maps are injective and the horizontal maps are surjective.

Now fix any $n$. If $\pi_1(N)$ is normal in $\pi_1^{orb}(\orbQ)$, then its image is normal in $\pi_1^{orb}(\orbQ_{n})$.  But, by \cite[Prop.~9.1]{NeumReid} and since $N(\bar{\alpha}_n) \to \orbQ_n$ is a cover from a knot complement to an orbifold with a rigid cusp, this is not the case.  Therefore, the cover $\phi \co N\to \orbQ$ is not normal and the elements of $\pi_1^{orb}(\orbQ) $ that do not normalize $\pi_1(N) $ induce hidden symmetries of $N$.  \end{proof}

The last result we prove in this setting deals with fillings of two-bridge link complements.  In particular, \Cref{corollary2} asserts that for any hyperbolic two component 2-bridge link complement, there are at most finitely many ways to fill one of the components so that the resulting orbifold is covered by a knot complement admitting hidden symmetries. 
{As noted in the introduction, many orbifold fillings of these link complements produce orbifolds covered by knot complements. We refer the reader to \Cref{subsec:pre} for more details on this.} As part of the argument we build a cover $\phi \colon \Sth \setminus L\to\orbQ$ which preserves the number of cusps (i.e., no cusps of $\Sth \setminus L$ are identified). 

\begin{proof}[Proof of \Cref{corollary2}] By way of contradiction, suppose that for some hyperbolic two-bridge link $L$ there is an infinite sequence of orbifolds $M_n$ obtained by Dehn filling one cusp of $N=\Sth \setminus L$, and that each $M_n$ is covered by a knot complements with hidden symmetries.

After possibly excluding at most a finite number of orbifolds in the sequence, we may assume that each $M_n$ is an $(\epsilon,d_N)$-twisted filling of $N$. Hence, $M_n$ is non-arithmetic by \cite[Proposition 6.2]{HMW19}.
 Since the knots covering the $M_n$ have hidden symmetries they cover orbifolds with rigid cusps \cite[Prop.~9.1]{NeumReid}, and it follows that the $M_n$ do as well. Therefore $N=\Sth \setminus L$ and the $M_n$ satisfy the hypotheses of \Cref{main_theorem}.

Suppose first that $N$ is a non-arithmetic two bridge link complement. The minimal orbifold $\orbQ_0$ in the commensurability class of $N$ has a single cusp, since there is an involution of $\Sth$ exchanging the components of $L$. This cusp is rigid by \Cref{main_theorem}, since $N$ covers an orbifold $\orbQ$ with a rigid cusp, and $\orbQ$ in turn covers $\orbQ_0$. But results of Millichap--Worden \cite{MilliWord} show that $N$ cannot cover an orbifold with a rigid cusp, a contradiction. In particular, \cite[Th.~3.6]{MilliWord} shows that the ``lifted cusp triangulation'' $\widetilde{T}$ of $N$ has no combinatorial rotational symmetry of order greater than two. But for any orbifold $\orbQ$ covered by $N$ the deck transformations of $\orbQ$ must preserve the canonical triangulation of $M$; and if $\orbQ$ had a rigid cusp then one such deck transformation would be an order-$3$ or -$4$ rotational symmetry of $\widetilde{T}$.

If $N=\Sth \setminus L$ is arithmetic then by \cite{GMacMart}, $L$ is one of the Whitehead link $5^2_1$, $6^2_2$, or $6^2_3$. The complement of $6^2_3$ is commensurable with $\mathrm{PSL}_2(\mathcal{O}_7)$, so Corollary 1.3(1) of \cite{CheDeMonda} implies the desired conclusion (cf.~Fact 3.1 there). It follows for the complements of $5^2_1$ and $6^2_2$ by condition (2) of  \cite[Cor.~1.3]{CheDeMonda}, applying results of \cite{NeumReid} (cf.~\cite[\S 1.3]{CheDeMonda}) and \cite[Prop.~2.3]{CheDeMonda}, respectively.\end{proof}

%% file: cautionary_example_geometric_arxiv.tex
%%%%%%%%%%%%%%%
\section{Two cautionary examples}\label{sec:CA}

\Cref{main_theorem} concerns a sequence of Dehn fillings $N(\bar{\alpha}_n)$ for which each $N(\bar{\alpha}_n)$ has a single cusp and $| \bar{\alpha}_n | \to \infty$.  It asserts that, if each $N(\bar{\alpha}_n)$ covers an orbifold with a single cusp, then so does the unfilled manifold $N$.  Furthermore, upon passing to a subsequence, each cover from $N(\bar{\alpha}_n)$ is an extension of the given cover from $N$.  \Cref{prop1} shows that there exist infinitely many hyperbolic knot complements with hidden symmetries and bounded volume if and only if this phenomenon occurs where $N$ is a hyperbolic link complement and the $N(\bar{\alpha}_n)$ are knot complements.

In \Cref{ex:three_cusped_example}, we show that the scenario addressed by  \Cref{main_theorem} \textit{does}  occur when we take $N$ to be the complement of the Borromean rings.  Here, the fillings are not knot complements -- such an example would yield a counterexample to  \Cref{main_conjecture} and, \textit{a fortiori}, Neumann--Reid's conjecture. 
Before we explain this example, we recall a related example introduced by Neumann and Reid in \cite{NR_rigidity}  where the unfilled manifold has two cusps but is not a link complement  (as shown in \cite{NR_rigidity} , its first homology has a $\mathbb{Z}/2\mathbb{Z}$ summand).   

But first we provide a proof of the lemma below, which pertains to both examples, giving a sufficient condition for a cover between cusped orbifolds to extend to their fillings. While we were unable to find a proof of this fact in the literature, it is likely well-known to experts.

\begin{lem}\label{lem:extend}
Suppose $p \co N \to Q$ is a cover from a manifold $N$ with torus boundary components to an orbifold $Q$.  Let $\gamma=(\gamma_1,\dots,\gamma_l)$ be a multi-slope for a subset of the boundary components of $N$ and $C$ a component of $\partial Q$, either a torus or $S^2(2,2,2,2)$, such that for each $i$ the component of $\partial N$ containing $\gamma_i$ maps to $C$. For a simple closed curve $\gbar$ on $C$ and $k\in\mathbb{N}$, the cover $p$ extends to a cover $p_\gamma$ from the Dehn filling $N_\gamma$ of $N$ along $\gamma$ to the Dehn filling $Q_{k\gbar}$ of $Q$ along a multiple $k\gbar$ of $\gbar$ if the following conditions hold:
\begin{enumerate}
	\item every curve in $p^{-1}(\gbar)$ has a multiple freely homotopic in $\partial N$ to some $\gamma_i$, and 
	\item for each $i$, $p(\gamma_i)$ is freely homotopic to $k\gbar$.
\end{enumerate}

\end{lem}

\begin{stremark}The above result is easily extended to the more general setting, where $\gbar$ is a multi-curve $\gbar=(\gbar_1,\dots,\gbar_m)$, by extending along $p^{-1}(\gbar_j)$ for each $\gbar_j$, one at a time.
\end{stremark}

Above for a simple closed curve $\gbar$ on $C$, regarded as an embedding $\mathbb{S}^1\to C$, a \textit{multiple $k\gbar$ of $\gbar$} is the composition of $\gbar$ with a degree-$k$ covering map $\mathbb{S}^1\to\mathbb{S}^1$. In the below proof, for a topological space $X$, we will denote by $\cone(X)$ the \textit{cone} on $X$,  ie.~the quotient space $\cone(X)=(X\times I)/(X\times \{1\})$. For $(x,t)\in X\times I$, we'll denote by $[x,t]$ the equivalence class of $(x,t)$, so that $[x,1]$ is the cone point (for any choice of $x$).

\begin{proof}

We will assume that the boundary component $C$ of $Q$ that is filled is an $S^2(2,2,2,2)$ orbifold---the case where $C$ is a torus is similar. Let $P$ be the filling orbi-pillow for $C$ in $\Qf$. Then $\gbar$ bounds an orbi-disk $\bar{D}$ with a single order-$k$ singularity $\bar{c}$, properly embedded in $P$ such that $\bar{D}$ cuts $P$ into 2 orbi-balls $\bar{H}$ and $\bar{H}'$ with (2,2,k) dihedral isotropy graphs. Refer to \Cref{fig:cones}.

Fix a curve $\alpha$ in $p^{-1}(\gbar)$ and let $V_i\subset N_{\gamma}$ be the filling orbi-solid torus for the component $T_i$ of $\partial N$ containing $\alpha$. $V_i$ has a cone angle of $2\pi/l_i$ around its core, where the filling curve $\gamma_i$ contained in $T_i$ is a multiple $l_i\gbar_i$ of a simple closed curve $\gbar_i$. Because $\alpha$ is a simple closed curve with a power freely homotopic to $\gamma_i$, it is isotopic to $\gbar_i$ and hence bounds an orbi-disk $D_\alpha$ with a single order-$l_i$ singularity $\bar{c}_i$ that is properly embedded in $V_i$.

We can identify the underlying disk of $D_\alpha$ with $\cone(\alpha)$, and of $\bar{D}$ with $\cone(\gbar)$, taking each singular point to the corresponding cone point. Having done so, we define $p_\gamma$ on $D_\alpha$ by coning the restriction of $p$ to $\alpha$: for $x\in \alpha$ and $t\in I$, define $p_\gamma([x,t])=[p(x),t]$. Taking $\restr{p_\gamma}{N}=p$ thus gives a well-defined extension $p_{\gamma}$ of $p$ to $N\cup D_{\alpha}$. Moreover, $p_{\gamma}$ takes the singular point $[x,1]$ of $D_{\alpha}$ to that of $\bar{D}$, and the fact that $\restr{p}{\alpha}$ is a covering map to $\gbar$ implies that $\restr{p_\gamma}{D_{\alpha}}$ is an orbifold cover to $\bar{D}$ (of degree $j_i$, where $k = j_il_i$).

We may assume that the disks $D_{\alpha}$ bounded by components $\alpha$ of $p^{-1}(\gbar)$ were chosen so that for any distinct  such components $\alpha$ and $\alpha'$, $D_\alpha\cap D_{\alpha'}=\emptyset$. Then for any fixed component $T_i$ of $\partial N$ that intersects $p^{-1}(\gbar)$, the union of disks $\bigcup D_\alpha$, taken over the components $\alpha$ of $p^{-1}(\gamma)\cap T_i$, cut the filling orbi-solid torus $V_i$ into ``orbi-2-handles''. Such a handle $H$ has underlying space homeomorphic to $\mathbb{D}^2\times I$, and a cone angle of $2\pi/l_i$ about its core $\{\mathbf{0}\}\times I$. We have $\del H=D_\alpha\cup A\cup D_{\alpha'}$, where $\alpha$ and $\alpha'$ are components of $p^{-1}(\gbar)$ and $A\subset T_i$ is an annulus. Thus $p_\gamma$ is already defined on $\del H$, mapping it to the boundary of one of the orbi-balls $\bar{H}$ and $\bar{H}'$ comprising $P$. We now show how to extend $p_{\gamma}$ over $H$.

Suppose without loss of generality that $p_{\gamma}(\del H) = \del \bar{H}$. As an orbifold, $\bar{H}$ is isomorphic to the quotient of the unit ball $\mathbb{D}^3\subset\mathbb{R}^3$ by the action of the dihedral subgroup $D_{2k}$ of $O(3)$ generated by rotations about the $x$- and $z$-axes, of order $2$ and $k$, respectively. Recalling from above that $k = j_il_i$, we see that this quotient factors through that of $\mathbb{D}^3$ by $\langle \rho^{l_i}\rangle$, where $\rho$ is the order-$k$ rotation about the $z$-axis. Moreover, map lifting yields an isomorphism $\iota\co\del H\to \del(\mathbb{D}^3/\langle\rho^{l_i}\rangle)$ such that the following diagram commutes:
\[ \xymatrix{ \del H \ar[r]^{\iota}\ar[d]^{p_{\gamma}} & \mathbb{D}^3/\langle\rho^{l_i}\rangle \ar[d] \\
	 \bar{H} \ar[r] & \mathbb{D}^3/D_{2k} } \]
We may use this to extend $p_{\gamma}$ across $H$ by just extending $\iota$ to an isomorphism $H\to\mathbb{D}^3/\langle\rho^{l_i}\rangle$.  We may accomplish this by coning again, first identifying the underlying ball of $H$ with $\cone(\del H)$ so that the singular locus is the cone of the two boundary singular points. Then map $\cone(\del H)$ to $\mathbb{D}^3/\langle\rho^{l_i}\rangle$ taking the cone point to the origin and $[x,t]$ to $(1-t)\iota(x)$ for each $x\in\del H$ and $t\in I$. This map determines an extension of $p_{\gamma}$ to an orbifold cover $H\to\bar{H}$ whose branching locus is the cone of the branching locus of $\restr{p_{\gamma}}{\del H}$.

After thus extending $p_\gamma$ on all of the 2-handles, we get a map $p_\gamma:N_\gamma\to Q_{k\gbar}$, which is an orbifold cover by construction.
\end{proof}

\begin{figure}%[h] why not let it float?
	\begin{subfigure}{.49\textwidth}
 		\centering
   		\includegraphics[scale=.9]{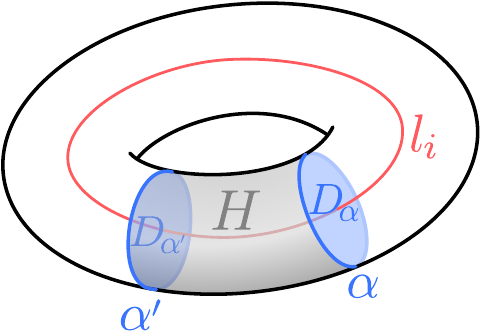}
   		\caption{}
   		\label{fig:solid_torus}
	\end{subfigure}
	\begin{subfigure}{.49\textwidth}
 		\centering
   		\includegraphics[scale=.9]{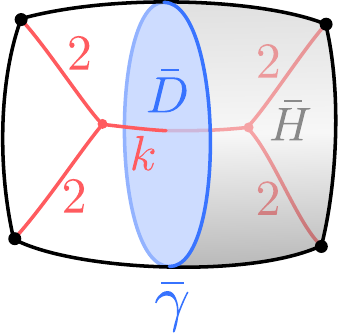}
   		\caption{}
   		\label{fig:pillow}
	\end{subfigure}	
	\caption{Left: A 2-handle $H$ in a filling solid torus of $N_\gamma$. Right: A filling pillow for $Q_{k\gbar}$ decomposes along $\bar{D}$ into two cones $\bar{H}$ (right half) and $\bar{H}'$ (left half, not labelled). We define $p_\gamma$ on $D_\alpha$ and $D_{\alpha'}$ by coning from $p|_\alpha$ and $p|_{\alpha'}$, then we extend it to $H=\cone(\del H)$ by coning from $p_\gamma|_{\del H}$. }
	\label{fig:cones}
\end{figure}

\begin{example}\label{two_cusped_example} 
The proof of Theorem 2 of \cite{NR_rigidity} describes a finite-volume, two-cusped hyperbolic $3$-manifold $N=\mathbb{H}^3/G_4$ (now known as census manifold \texttt{t12046}, or \texttt{ooct02\_00020}) that covers a two-cusped orbifold $\orbQ=\mathbb{H}^3/G_1$ which has underlying space a ball and isotropy graph pictured in  \Cref{Fig6prime}.  The empty circles in the figure represent the cusps of $\orbQ$,  each of which has cross section a Euclidean orbifold with underlying space $\mathbb{S}^2$. With four edges of the isotropy graph incident on it, each labeled $2$, the lower circle represents a pillowcase cusp. Analogously, the upper circle represents a rigid $S^2(2,4,4)$ cusp.  (Note that $N$ is called ``$M$'' in \cite{NR_rigidity}.)  

\begin{figure}
\begin{tikzpicture}

\draw (0,0) circle [radius=1];
\fill [color=white] (0,1) circle [radius=0.2];
\draw (0,1) circle [radius=0.2];

\draw (-1,0) -- (0.5,0.3) -- (1,0);
\fill [color=white] (0,0.2) circle [radius=0.1];
\draw (0,0.8) -- (0,-0.8);
\draw (0.5,0.3) -- (0,-1);
\fill [color=white] (0,-1) circle [radius=0.2];
\draw (0,-1) circle [radius=0.2];

\node [above right] at (0.75,0.45) {$4$};
\node [above left] at (-0.75,0.45) {$4$};
\node [below left] at (-0.75,-0.45) {$2$};
\node [below right] at (0.75,-0.45) {$2$};
\node at (-0.15,-0.15) {$2$};
\node at (-0.5,0.3) {$2$};
\node at (0.45,-0.3) {$2$};
\node at (0.7,-0.05) {$2$};

\end{tikzpicture}
\caption{$\orbQ \doteq\mathbb{H}^3/G_1$} 
\label{Fig6prime}
\end{figure}
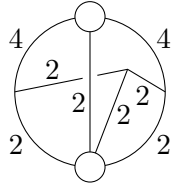

Our main goal in this example is to highlight the following assertion about $N$ and $\orbQ$ from \cite[p.~230]{NR_rigidity}, quoted here with our notation substituted: 

\begin{quote}One of the two cusps of $N$ covers the (non-rigid) pillow[case] cusp of $\orbQ$, while the other covers the (rigid) $(2,4,4)$-cusp. Hyperbolic Dehn filling on the former cusp of $N$ lifts from hyperbolic Dehn filling of the non-rigid cusp of $\orbQ$, so under such Dehn fillings, the other cusp of $N$ is still the cover of the $(2,4,4)$-cusp of $\orbQ$, and hence geometrically invariant.\end{quote}

Any term in an infinite sequence of such fillings of $\orbQ$ has a rigid cusp, so the sequence of corresponding fillings of $N$ will satisfy the hypotheses and  conclusion of  \Cref{main_theorem}.

The same sequence of fillings of $N$ also satisfies the hypotheses and conclusions of \cite[Corollary 1.8]{CheDeMonda}, in particular its conclusion (2) pertaining to geometric isolation (the subject of \cite[Theorem 2]{NR_rigidity}). We will discuss geometric isolation further in \Cref{isooo}.
\end{example}

We now exhibit a second example where the hypotheses and conclusion of Proposition \ref{main_theorem} hold for a sequence of one-cusped manifolds 
which converge to the complement of the Borromean rings.  This link complement has three cusps and covers $\orbQ = \mathbb{H}^3/G_1$ from  \Cref{two_cusped_example}.  Our discussion below includes some additional comparisons between this pair of examples.

\begin{example}\label{ex:three_cusped_example}  
For this example let $N$ denote the Borromean rings complement.   \Cref{fig:BorromeanRingsCover} pictures an eight-to-one cover $\phi \co N \to \orbQ$, where $\orbQ$ is the same orbifold $\orbQ$ as in  \Cref{Fig6prime} and  \Cref{two_cusped_example}.  This cover factors as a  sequence $N\to \orbQ_1 \to \orbQ_0 \to \orbQ$ of intermediate covers, each a two-to-one involution quotient. Inspecting \Cref{fig:BorromeanRingsCover} reveals that each cusp of $N$ covers a distinct cusp of $\orbQ_0$, and two distinct cusps of $\orbQ_0$ cover the pillowcase cusp $c$ of $\orbQ$. Therefore the rigid $S^2(2,4,4)$ cusp of $\orbQ$ has connected preimage in $N$, and the other two cusps of $N$ cover $c$.

%%%%%%%%%%%%%%%%%%%%%%%%%%%%%%%%%%%%
%%%%%%%%%%%%%%%%%%%%%%%%%%%%%%%%%%%%

\begin{figure}
\includegraphics[width=6in]{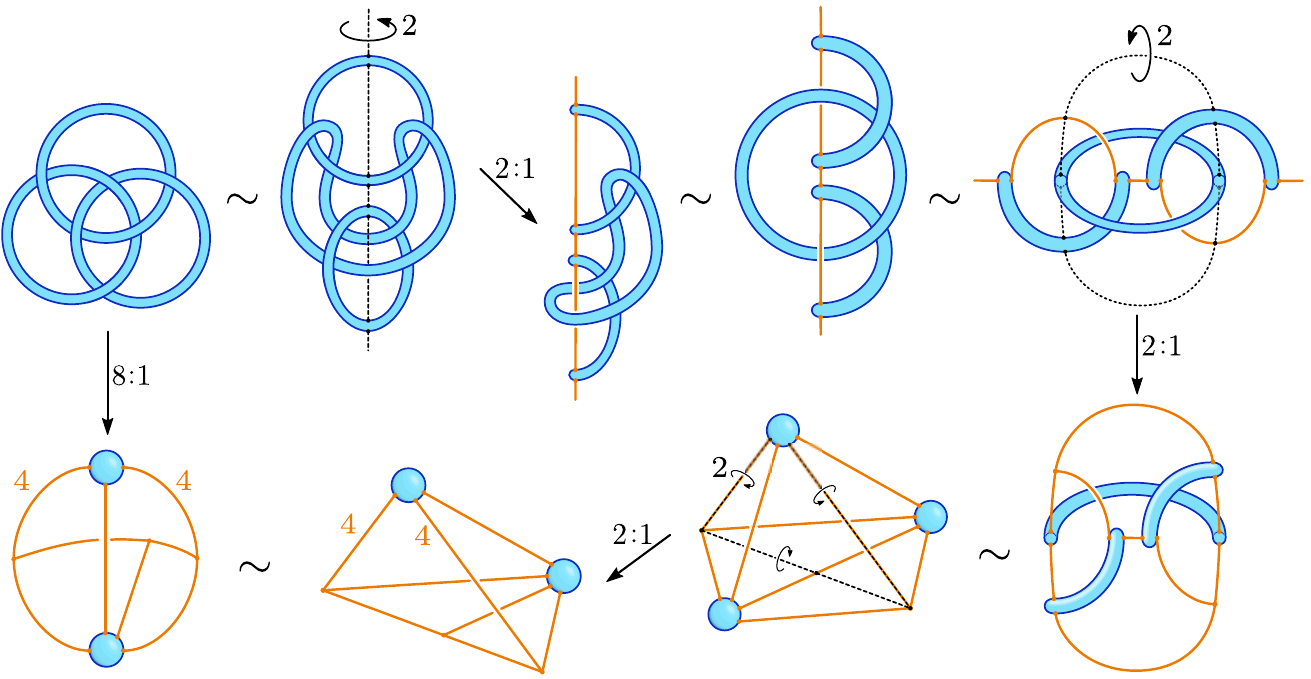}
\caption{\label{fig:BorromeanRingsCover} The cover of the orbifold in \Cref{Fig6prime} by the Borromean rings. All edges are order 2 unless labelled otherwise.}
\end{figure}

Let $c_1$ and $c_2$ be the cusps of $N$ that cover $c$. Unlike in Example \ref{two_cusped_example}, many fillings of $c_1$ and $c_2$ do not produce manifolds that cover a filling of $c$. In particular, certain double twist knot complements are produced by fillings of $c_1$ and $c_2$, as pictured in \Cref{fig:doubleTwistKnots}. These are 2-bridge knot complements, which are known not to have $\Q(i)$ as a subfield of their invariant trace field by \cite[Proposition 2.5]{ReidWalsh}. Therefore none of these covers a filling of the pillowcase cusp of $\orbQ$: the $S^2(2,4,4)$ cusp of such a filling forces every cover to have cusp field $\mathbb{Q}(i)$.

\begin{figure}[h]
	\begin{subfigure}{.49\textwidth}
 		\centering
   		\includegraphics[scale=1]{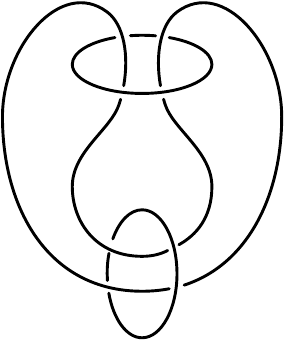}
   		\label{fig:borromean}
   		\caption{}
	\end{subfigure}
	\begin{subfigure}{.49\textwidth}
 		\centering
   		\includegraphics[scale=1]{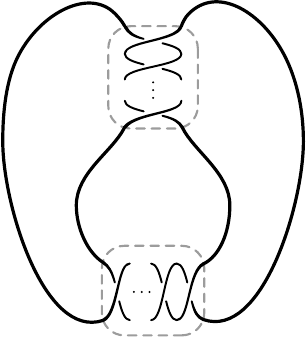}
   		\label{fig:double_twist}
   		\caption{}
	\end{subfigure}	
	\caption{Performing $(1,q_1)$ and $(1,q_2)$ Dehn surgery on the circular components of the Borromean rings, as drawn in (a), results in the double twist knot shown in (b). A double twist knot obtained in this way will have $2q_1$ half twists in one twist region, and $2q_2$ in the other.}
	\label{fig:doubleTwistKnots}
\end{figure}
 
On the other hand, many fillings of $c_1$ and $c_2$ do cover fillings of $c$. Take $\Pi$ to be the maximal $\mathbb{Z}\times\mathbb{Z}$ subgroup of the orbifold fundamental group of $c$, and let $\{\alpha_n\}\subset\Pi$ be an infinite sequence whose members all lie in the intersection of the finite-index sub-lattices $\pi_1 c_1$ and $\pi_1 c_2$ of $\Pi$. Then all but finitely many $\alpha_n$ determine hyperbolic fillings of $c_0$ by the hyperbolic Dehn surgery theorem (see \cite[\S 5.3]{DM} for the orbifold version), and by \Cref{lem:extend} 
each of these is covered by fillings of $c_1$ and $c_2$. Such a sequence $\{\alpha_n\}$ therefore determines a sequence $\{M_n\}$ of one-cusped manifolds converging to $N$ and satisfying the hypotheses and conclusions of  \Cref{main_theorem}.
 We point out, though, that no such filling is a knot complement by recent work of the third author \cite[Theorem 1.1]{Hoffman_rigid_cusps}. 
\end{example}

{ \subsection{Geometric Isolation}\label{isooo}
 Here we analyze \Cref{two_cusped_example},  \Cref{ex:three_cusped_example}, and  \Cref{main_conjecture} in the context of \textit{geometric isolation}. Given an n-cusped hyperbolic $3$-orbifold with some choice of labeling on the cusps, we say that cusps $1, \ldots, k$ are \textit{geometrically isolated} from cusps $k+1, \ldots, n$ if the cusp shapes of the first $k$ cusps only change under finitely many Dehn fillings along cusps $k+1, \ldots, n$. (There seems to be some variations of this definition in the literature.) For instance, \Cref{two_cusped_example} which comes from \cite[Theorem 1]{NR_rigidity}, shows that there are infinitely many $2$-cusped hyperbolic $3$-manifolds whose pairs of cusps are geometrically isolated from one another. In the context of \Cref{prop1}, finding  n-component hyperbolic link complements with one cusp $c$ geometrically isolated from the other $n-1$ cusps is a reasonable starting point to explore examples that could show that \Cref{main_conjecture} is false.  As a result, it's natural for us to consider when geometric isolation occurs among link complements and then analyze their fillings and covers, as done in \Cref{two_cusped_example}.
 
This geometric isolation phenomenon is further analyzed in  \cite{CheDeMonda}. Proposition 1.6 there describes a rational function $\tau_c$ on the deformation variety of a complete hyperbolic manifold $N$ associated to a cusp $c$ of $N$, which measures the ``cusp parameter'' of $c$ on the (complex)  codimension-one subvariety $V_0$ consisting of hyperbolic structures where $N$ is complete. In the proof of Corollary 1.8 there it is observed that when $N$ has exactly two cusps, $\tau_c$ is constant on $V_0$ if an infinite sequence $\{M_n\}$ of fillings of the other cusp yield manifolds covering orbifolds with rigid cusps. In particular, the cusp $c$ of $N$ is geometrically isolated from the other cusp of $N$; see Corollary 1.3 from \cite{CheDeMonda}. 

 \Cref{ex:three_cusped_example} shows that when $N$ has three cusps, there may exist a sequence of one-cusped fillings $\{M_n\}$ that cover orbifolds with rigid cusps even while $\tau_c$ is non-constant on $V_0$. {As described above, taking $N$ to be the Borromean rings complement, one of the $M_n$ can be constructed by choosing a filling slope of $\alpha_i$ for one cusp of $N$ and then filling along $\alpha_j$ the other cusp along a slope $\alpha_j$ such that under the cover in \Cref{fig:BorromeanRingsCover} the images of $\alpha_i$ and $\alpha_j$ are identical. However, if we perform $(1,q_1)$ and $(1, q_2)$ fillings along the cusps of the Borromean rings complement $N$ we obtain the double twist knot complements, which never have cusp parameter in $\mathbb{Q}(i)$ by \cite[Proposition 2.5]{ReidWalsh}.}
 This is not a surprise since $V_0$ has dimension two in this case and we expect the locus where $\tau_c$ takes the same value as at the complete structure to have codimension one in $V_0$. What seems exceptional about the case of  \Cref{ex:three_cusped_example} is that the points of $V_0$ corresponding to the $M_n$ lie in this locus.}

%% file: effectivization_arxiv.tex
%%%%%%%%%%%%%%%%%%%%%%%%%%
\section{Effectivization}\label{sec:effectivization}

Our main results give us tools for showing that no infinite sequence of knot complements with hidden symmetries converges to a given link complement $N=\Sth\setminus L$. It is natural to ask for an \textit{effective} version of these results, where the number of such knot complements that can be produced by filling cusps of $N$ is bounded above in terms of some properties of $N$. The machinery of $(\epsilon,d_N)$-fillings described in Section \ref{sub:edn} can be used with recent work of Futer--Purcell--Schleimer \cite{FPS19} to produce such bounds, provided the systole of $N$ is known (see \cite[Section 6]{HMW19}). Unfortunately, the numbers it yields are not small enough to render the collection of remaining fillings computationally approachable at present. 

In this section we introduce new methods to prove two effectivization results that lack this shortcoming, with the trade-off that they apply only to knot complements that cover orbifolds of low volume. For the first, recall that ``$v_0$'' denotes the volume of a regular ideal tetrahedron, roughly $1.01$.

\begin{thm}\label{thm:volume_bound_gen}\VolumeBoundGen
\end{thm}

\Cref{thm:volume_bound_gen} implies in particular that the only hyperbolic knot complement with volume at most $6v_0$ and hidden symmetries is that of the figure-eight, but it also applies to hyperbolic knot complements with unbounded volume. For instance, suppose $L = K\sqcup K'$ is a hyperbolic two-component link such that $N\doteq \Sth \setminus L$ has volume at most $6v_0$, $K'$ is unknotted, and $K$ has non-zero linking number with $K'$. Then there are knot complements covering lens space fillings of the $K'$-cusp of $N$ with arbitrarily large degree. See \Cref{prop:linking_number} below, and \Cref{eminem} for an infinite family of link complements as above.

The second main result of this section gives a short list of explicit criteria that \textit{any} filling slope on the $K'$-cusp of a link complement $N$ as in the paragraph above must satisfy if it is to produce a manifold or orbifold covered by a knot complement with hidden symmetries.

\begin{thm}\label{earem}\TheoremEarem
\end{thm}

Here following Hodgson--Kerckhoff, we define the \textit{normalized length} of a peripheral curve to be the length of the curve's geodesic representative in the Euclidean metric on a horospherical cusp cross-section $T$, divided by $\sqrt{\mathrm{area}(T)}$. The set of slopes satisfying the criteria above is small enough in individual cases to be practically checked using rigorous computational methods. We illustrate this in \Cref{sec:sixtwotwo} by using \Cref{earem} to prove \Cref{sixtwotwo}, which asserts that a cusp of the $6^2_2$-link complement has exactly one filling covered by a hyperbolic knot complement with hidden symmetries.

Both main results of this section use the low-volume hypothesis to play off results of the third author, which give lower bounds on the degree of covers to rigid-cusped orbifolds, against results of Adams classifying the lowest-volume (rigid-)cusped orbifolds. \Cref{thm:volume_bound_gen} follows quickly from this strategy, in \Cref{sub:mani}. For \Cref{earem} we also leverage work of Hodgson--Kerckoff, which shows that if the filling slope is long then the geodesic introduced by Dehn filling is short with a large tubular neighborhood. A careful study of orbifold covers in the situation in question rules this out, and also leads to further restrictions on $p$ and $q$.

Finally, in \Cref{sub:orbifold_fillings} we give some related questions and directions for future work.

%%%%%%%%%%%%%%%%%%%%
\subsection{Linking number and Dehn fillings} \label{subsec:pre}

Before proceeding further, here we clarify exactly which fillings of one cusp of a two-cusped manifold $N$ are covered by knot complements in $\Sth$, in the current case of interest: $N = \Sth \setminus L$ for a link $L = K\sqcup K'$, where $K$ and $K'$ are tame knots with $K'$ unknotted. The main result \Cref{prop:linking_number} of this sub-section shows that this property is completely determined by the linking number between $K$ and $K'$. Then in \Cref{eminem} we describe an infinite family of two-component links with unknotted components and low-volume complements.

Recall that the \textit{linking number} between $K$ and $K'$ can be defined (up to a factor of $\pm 1$) as the oriented intersection number between $K$ and a \textit{Seifert surface} for $K'$, an embedded orientable surface $S\subset\Sth$ with a single boundary component such that $\partial S = K'$. 
Letting $n(K')$ be a closed regular neighborhood of $K'$, chosen so that $K\cap n(K') =\emptyset$ and $S\cap n(K')$ is a collar of $K'=\partial S$ in $S$, it follows that $S$ intersects the boundary $T$ of $n(K')$ in the longitude $\lambda$ of the standard homological framing. Compare \cite[5D]{Rolfsen}.

$K'$ is unknotted if and only if it has a Seifert surface $D$ homeomorphic to a disk. In this case, the Dehn twist of $T$ about $\lambda$ extends to a self-homeomorphism of $\Sth \setminus \mathrm{int}(n(K'))$ supported in a neighborhood of $D$. Taking powers of this Dehn twist, for any $q\in\mathbb{Z}$ we thus have a self-homeomorphism of $\Sth \setminus \mathrm{int}(n(K'))$ fixing $\lambda$ and mapping the meridian $\mu$ to a curve with homology class $[\mu]+q[\lambda]$ (where coordinates are given in the standard homological framing of \Cref{whatsadiorama?}). This map extends to a homeomorphism from $\Sth$ to the $(1,q)$-surgery along $K'$, and we obtain the well-known fact:

\begin{fact}\label{HoobieDoobie} For a two-component link $L = K\sqcup K'\subset\Sth$ with $K'$ unknotted, and any $q\in\mathbb{Z}-\{0\}$, the $\pm(1,q)$-filling on the $K'$-cusp of $\Sth \setminus L$ (given the standard homological framing) yields a knot complement in $\Sth$.\end{fact}

{More generally, for any relatively prime $p,q\in\mathbb{N}$, $(p,q)$-filling on the $K'$-cusp yields a knot complement in the lens space $L(p,q)$. The below proposition uses this fact to describe fillings of the $K'$-cusp that are \textit{covered} by knot complements. The \textit{orbi-lens space} mentioned below is as defined in \cite[\S 3]{BBoCWaGT}. As defined in that reference, the homology class of a curve $[K]$ is \emph{primitive} in an orbi-lens space $\orb$ if $[K]$ is a generator of $H_1(\orb,\Z).$ A knot $K$ is \emph{primitive} in $\orb$ if $[K]$ is primitive. } 
 Given Dehn fillings {of both cusps of} $\Sth \setminus L$ yielding an orbifold $\orb$, we have that $\orb \cong \Sth \setminus n(L) \cup( (D^2 \times S^1) \cup  (D^2 \times S^1))$. Removing the cores of these two solid tori, results in a link complement in  $\orb$, $\orb \setminus \bar{L}$ which is homeomorphic to $\Sth \setminus L$. Here we say $L$ and $\bar{L}$ are \emph{surgery dual}. %or just \emph{dual}.

{
\begin{prop}\label{prop:linking_number}
Let $L = K\sqcup K'$ be a two component link in $\Sth$ such that $K'$ is unknotted, {and give both components the standard homological framing}. Let $\ell$ be the linking number between $K$ and $K'$. 
%Furthermore, assume that $K'$ carries the standard homological framing, for any $(p,q)\in\mathbb{Z}^2$ (including where $p$ and $q$ are not relatively prime). 
Denote by $\orbQ$ the orbifold obtained by $(p,q)$ filling along $K'$ and $\orb$ the orbi-lens space obtained by $(1,0)$ filling along $K$ and $(p,q)$ filling along $K'$. Then the following are equivalent:
\begin{enumerate}[label=(\roman{*}), ref=(\roman{*})]
\item \label{prop:linking_number_i} 
%When considered as an embedded knot in $\orb$, $K$ is primitive,
{The surgery dual $\bar{K}$ of $K$ in $\orb$ is primitive}.  
\item \label{prop:linking_number_ii} 
$gcd(p,\ell)=1$.
\item \label{prop:linking_number_iii}
 $\orbQ$ has a $p$-fold cover by a knot complement in $S^3$. 
\end{enumerate}
\end{prop}

}

{
\begin{proof}
For this proof, denote by $\bar{K}$ and $\bar{K}'$ the dual knots to $K$ and $K'$ in $\orb\cong L(p,q)$. {Since it is unknotted}, $K'$ bounds a disk $D$ in $\Sth$. We say $D_\ell$ is the punctured disk corresponding to $D$ in $\Sth \setminus L$, and $\bar{D}_\ell$ is the corresponding disk in $\orb \setminus \bar{K} \cup \bar{K'}$.

{\bf (i) $\Rightarrow$ (ii)} We observe that $H_1(\orb,\Z) = \Z/p\Z$.  {Since the linking number of $K$ with $K'$ is $\ell$, the signed intersection number of $K$ with the disk bounded by $K'$ is $\ell$.}
Since $[\bar{K}]$ is a generator of $H_1(\orb,\Z)$, we must have $gcd(p,\ell)=1$.

{\bf (ii) $\Rightarrow$ (iii)} {Since $gcd(p,l)=1$}, $\bar{K}$ is homotopic to an $\ell$ strand braid in $\orb \setminus \bar{K}'$. If we pick a base point on $\bar{K}$ and follow the intersections between $n(\bar{K})$ and $\bar{D}_\ell$ in $\orb \setminus \bar{K} \cup \bar{K}'$,  the braid gives a permutation on the curves $n(\bar{K}) \cap \bar{D}_\ell$. Since $K$ and $\bar{K}$ are connected, this is an $\ell$-cycle $\sigma$. 
{In the $p$-fold cyclic cover of $\orb\setminus \bar{K}'$, this permutation lifts to a permutation $\sigma^p$, which is again an $\ell$-cycle since $gcd(p,\ell)=1$. It follows that $\bar{K}$ lifts to a knot in the $p$-fold cyclic cover of $\orb\setminus \bar{K}'$, and hence the cusp corresponding to $\bar{K}$ in $\orbQ$ lifts to a single cusp in its $p$-fold cyclic cover.}
Finally, we point out the that local homotopy relations lift to an invariant homotopy in the cover, in both cases, the connectivity of $K$ and $\bar{K}$ remains unchanged. 

 {\bf (iii) $\Rightarrow$ (i)} In this case, we assume $\orbQ$ has a $p$-fold cover by $\Sth \setminus K_p$ for some $K_p$. Since $\orbQ$ has a torus cusp and $\Sth \setminus K_p$ is a knot complement, this cover is both regular and cyclic (see proof of \cite[Proposition 9.1]{NeumReid}).
{Suppose that $\bar{K}$ is not primitive. Then by a similar argument to that given above, if the order of $[\bar{K}]$ is $p/d$, then $\bar{K}$ will lift to $d$ disconnected components in the $p$-fold cyclic cover of $\orb\setminus \bar{K}'$ {(corresponding to the cover $\Sth \rightarrow \orb$)}, and hence the $p$-fold cyclic cover of $\orbQ$ will be a link complement with $d>1$ components.}
\end{proof}
}

{
\begin{remark}
We now make two further observations about the previous proposition.
First, if $gcd(p,\ell)=1$ and $\orbQ$ admits a second non-trivial orbi-lens space filling, then this filling also 
corresponds to a knot complement covering $\orbQ$ by \cite[Proposition 4.13]{BBoCWaGT}. Second, in the case that 
$K$ is not primitive, we have that $\orbQ$ is covered by an $\Sth$ link complement. The number of 
components of the link is precisely $gcd(p,\ell)$. 
\end{remark}
}

In the second and third case, if the relevant linking number $\ell$ is not relatively prime to $p$, the preimage of $K$ under the universal cover is not connected, so the relevant filling is a (multi-cusped) link complement as opposed to a knot complement. In the special setting that $\ell = 0$ we make the following observation.

\begin{remark}\label{rem:zero_linking}
If $\orbQ$ is the result of $(p,q)$-filling on one unknotted component of a two component link $L=K \cup K'$ such that the linking number $\ell$ of $K$ and $K'$ is $0$, then $\orbQ$ is covered by a knot complement in $\Sth$ if and only if $p=\pm 1$ (and the cover is the identity map).
\end{remark}

The below example demonstrates that there are links $L=K\sqcup K'$ satisfying the conditions of \Cref{earem}, for all but finitely many choices of linking number $l$.

\begin{example}\label{eminem} The ``magic manifold'' is the complement in $\Sth$ of the minimally twisted three-component chain link $L_3=6^3_1$, shown in \Cref{fig:magic}. It has the minimal known volume among all complete, three-cusped hyperbolic $3$-manifolds, roughly $5.33$. (Note that this is less than $6v_0$.) 

Each component of $L_3$ is unknotted, so for any $q\in\mathbb{Z}$, performing $(1,q)$ surgery along one of them (given the standard homological framing) yields a two-component link in $\Sth$ as in \Cref{rem:zero_linking}. The resulting link has two unknotted components with linking number $q+1$; in fact, $q=2$ yields $6^2_2$ (see \Cref{fig:magicfillings}.)  By the hyperbolic Dehn surgery theorem, all but finitely many of these has a hyperbolic complement, and each such has volume less than $6v_0$.\end{example}

\begin{figure}[h]
	\begin{subfigure}{.49\textwidth}
 		\centering
   		\includegraphics[scale=1]{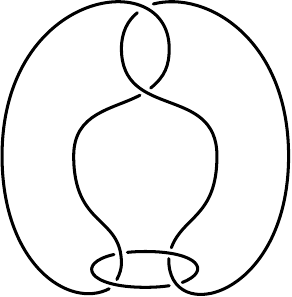}
   		\caption{}
   		\label{fig:magic}
	\end{subfigure}
	\begin{subfigure}{.49\textwidth}
 		\centering
   		\includegraphics[scale=1]{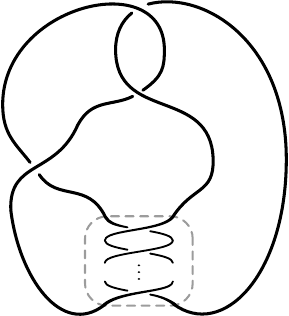}
   		\caption{}
   		\label{fig:magic_twisted}
	\end{subfigure}	
	\caption{Performing $(1,q)$ Dehn surgery on the bottom component of the `magic manifold', the framed link complement drawn in (a), results in the two component link with one clasp and one twist region consisting of $2q$ crossings (b).}
	\label{fig:magicfillings}
\end{figure}

%%%%%%%%%%%%%%%%
\subsection{Knot complements covering manifolds}\label{sub:mani}

In this subsection we prove \Cref{thm:volume_bound_gen}, eliminating (with one exception) the possibility that a knot complement with hidden symmetries can cover a manifold with volume at most $6v_0$. The following lemma is a key tool. It establishes an upper bound on volume for rigid cusped orbifolds not covered by the figure-eight knot complement. 

\begin{lem}\label{lem:ANR_vol_bound}[Adams, Neumann--Reid]
Let $\orbQ$ be a non-arithmetic $3$-orbifold with a rigid cusp. %$C$,

\begin{enumerate}
\item If $\orbQ$ has a $S^{2}(2,3,6)$ cusp, $vol(\orbQ) > \frac{v_0}{4}$,
\item If $\orbQ$ has a $S^{2}(3,3,3)$ cusp, $vol(\orbQ) > \frac{v_0}{2}$, and
\item If $\orbQ$ has a $S^{2}(2,4,4)$ cusp, $vol(\orbQ) \geq \frac{v_0}{2\sqrt{3}}$.
\end{enumerate}
\end{lem}

\Cref{lem:ANR_vol_bound} follows from results of Adams \cite{Adams_small_vol} and of Neumann--Reid \cite{NR92b}. Specifically, \cite[Theorem 3.3]{Adams_small_vol} gives the 3 possible volumes of hyperbolic $3$-orbifolds that have a $S^{2}(2,3,6)$ cusp and volume less than $\frac{v_0}{4}$. The work preceding the statement of that theorem describes the orbifolds with these properties, which are shown to be arithmetic in \cite{NR92b}. For hyperbolic $3$-orbifolds that admit an $S^{2}(2,3,6)$ cusp and have volume exactly $\frac{v_0}{4}$, the discussion on page 5 of \cite{Adams_small_vol} shows that any such orbifold must admit the maximally dense horoball packing, and so, must be arithmetic. A similar analysis using Adams \cite{Adams_small_vol} and  Neumann--Reid \cite{NR92b} can be given for the $S^{2}(3,3,3)$ and $S^{2}(2,4,4)$ cases.

By combining the volume bounds in the above lemma with the minimum degree covering bounds from \cite{Hoffman_hidden} and \cite{Hoffman_rigid_cusps}, we are now able to prove \Cref{thm:volume_bound_gen}.

\begin{proof}[Proof of \Cref{thm:volume_bound_gen}]
For the proof of this theorem, we will consider possible coverings $\phi \colon \Sth \setminus K \to \orbQ$, where $\orbQ$ is a rigid cusped hyperbolic orbifold. To start, if $\Sth \setminus K$ is arithmetic, then it must be the figure-eight knot complement \cite{ReidFig8}, which admits hidden symmetries. Moving forward, we will assume that $\Sth \setminus K$ (and therefore $\orbQ$) is non-arithmetic. 

 We now consider the cusp type of the rigid cusped quotient and the minimum degree of a manifold cover for each type. 
Let $M$ be a manifold covered by $\Sth \setminus K$ (possibly trivially) and $\phi_2 \colon M \to \orbQ$ as an orbifold cover. By \cite[Lemma 5.5]{Hoffman_hidden}, $deg(\phi_2)$ is a multiple of $12$ if $\orbQ$ has an $S^2(3,3,3)$ cusp, and $deg(\phi_2)\ge 24$ if $\orbQ$ has a $S^2(2,4,4)$ cusp. By \cite[Theorem 1.2]{Hoffman_rigid_cusps} $deg(\phi_2)$ is a multiple of $24$ if $\orbQ$ has an $S^2(2,3,6)$ cusp. Appealing to  \Cref{lem:ANR_vol_bound} we conclude that $vol(\Sth\setminus K) \geq vol(M) > 6v_0$ for any non-arithmetic $\Sth \setminus K$. 
\end{proof}

%%%%%%%%%%%%%%%%
\subsection{Knot complements covering orbifolds produced by Dehn filling}\label{sub:orbi}

The ultimate goal of this subsection is to prove \Cref{earem}, restricting which Dehn fillings of one cusp of a two-component link complement can be covered by a knot complement with hidden symmetries. Given \Cref{thm:volume_bound_gen}, we are left to address the case of an \textit{orbifold} filling, where the filling slope does not represent a primitive element of the peripheral subgroup. We begin with some background on orbifolds, focusing on those of particular interest here: orientable, rigid-cusped orbifolds covered by knot complements.

By \cite[Corollary 4.11]{BoiBoCWa2} (see also \cite[Proposition 2.3]{Hoffman_hidden}) a rigid cusped orbifold $\orb$ that is covered by a knot complement has underlying space an open ball, and therefore is determined by its singular set, which is a properly embedded trivalent graph $\Lambda_{\orb}$ called the \emph{isotropy graph}. Each edge of $\Lambda_{\orb}$ is labelled by the order of the  torsion on the edge, which is the order of the cyclic \emph{isotropy group} at any (interior) point on that edge. At each vertex of $\Lambda_{\orb}$ there is an associated isotropy group which is either dihedral, or is one of $A_4$, $S_4$, or $A_5$. If an edge incident on such a vertex is labelled by $k$-torsion, then the isotropy group of the vertex must have a cyclic subgroup of order $k$. These restrictions on possible isotropy groups will be useful going forward.

The \textit{cusp} of such an orbifold $\orb$ is a neighborhood homeomorphic to $\mathbb{S}^2\times (0,1)$, complementary to a compact sub-ball, that intersects $\Lambda_{\orb}$ in a disjoint union of three intervals. These intervals may have torsion labels $(2,4,4)$, $(2,3,6)$, or $(3,3,3)$, reflecting that a horospherical cross-section is the quotient of $\mathbb{R}^2$ by the orientation-preserving subgroup of a Euclidean triangle group. We accordingly say that $\orb$ has an $S^2(2,4,4)$, $S^2(2,3,6)$, or $S^2(3,3,3)$ cusp, respectively.

$\orb$ is the quotient of $\mathbb{H}^3$ by the isometric action of its \emph{orbifold fundamental group}, which can be computed from the fundamental group $\pi_1(\orb \setminus \Lambda_{\orb})$ of the complement of $\Lambda_{\orb}$ in $\orb$. In particular, for an edge $e$ of $\Lambda_{\orb}$ let $k_e$ be the order of the isotropy group at $e$, and let $\mu_e\in \pi_1(\orb \setminus \Lambda_{\orb})$ be a meridian around $e$. Let $N$ be the subgroup of $\pi_1(\orb \setminus \Lambda_{\orb})$ normally generated by all elements $\mu_e^{k_e}$, where $e$ ranges over the edges of $\Lambda_{\orb}$. Then the orbifold fundamental group $\pi_1^{orb}(\orb)$ is the quotient of $\pi_1(\orb \setminus \Lambda_{\orb})$ by $N$. Of course when $\orb$ is hyperbolic, $\pi_1^{orb}(\orb)$ is isomorphic to the Kleinian group $\Gamma_{\orb}$ satisfying $\orb\cong \H^3/\Gamma_{\orb}$. We refer the reader to \cite{BMP05} for a more thorough introduction to the above topics.

We now describe a particularly useful homomorphism of $\pi_1^{orb}(\orb)$, called the \emph{cusp-killing homomorphism}. For each edge of $\Lambda_{\orb}$ which terminates on a cusp of $\orb$, let $P_e$ be the isotropy group for the edge, and let $P$ be the subgroup of $\pi_1^{orb}(\orb)$ normally generated by elements of these peripheral isotropy groups $P_e$. Define $f:\pi_1^{orb}(\orb)\to \pi_1^{orb}(\orb)/P$ to be the canonical homomorphism. This homomorphism can be understood as follows: starting with the isotropy graph $\Lambda_{\orb}$, first erase all peripheral edges (those terminating on a cusp). If two edges $e_1$ and $e_2$ come together at a vertex of $\Lambda_{\orb}$ at which the third edge has been erased, then join $e_1$ and $e_2$ into a single edge $e$, labelled with $k_e=\gcd(k_{e_1},k_{e_2})$-torsion, where $k_{e_i}$ is the torsion on $e_i$. If $k_e=1$, then erase the edge $e$ (since it is no longer singular). Continue until either all vertices are trivalent or the isotropy graph is empty. See \cite[Figure 3]{Hoffman_hidden} for a visual interpretation of this process. The usefulness of this homomorphism is suggested by \cite[Proposition 2.3]{Hoffman_hidden}, which states that for a rigid cusped orbifold $\orb$ covering a knot complement, the cusp-killing homomorphism is trivial.

In what follows, we will consider an orbifold $\orbQ$ that is both covered by a knot complement and is obtained as a $(p,q)$-filling of one component of $\Sth\setminus L$, where $L=K\sqcup K'$ is as in the statement of \Cref{earem}.  Here, $(p,q)$ do not need to be relatively prime, and so such fillings could produce an orbifold that is not a manifold. We are interested in determining when such orbifolds have hidden symmetries,  or equivalently, when such orbifolds cover a rigid cusped orbifold.  The recent work of Hoffman \cite{Hoffman_rigid_cusps} shows that we only need to consider when $\orbQ$ covers an orbifold with an $S^{2}(3,3,3)$ cusp. The following lemma starts the analysis of what possible $(p,q)$ fillings and covering maps could result in an orbifold $\orbQ$ with such properties.

\begin{lem}\label{lem:cusp_shapes_6_2_2_fillings} Suppose $L = K\sqcup K'$ is a two-component link in $\Sth$ with $K'$ unknotted, such that $N = \Sth\setminus L$ is hyperbolic with volume at most $6v_0$. For any $(p,q)\in\mathbb{Z}^2$ with $\gcd(p,q)>1$ such that the orbifold $\orbQ$ obtained from $(p,q)$ filling on the $K'$-cusp of $N$ is covered by a knot complement and has a cover $\phi \colon\orbQ \to \orb$ such that $\orb$ has an $S^{2}(3,3,3)$ cusp: \begin{itemize}
	\item $\phi$ has degree 6;
	\item $\gcd(p,q)\in\{2,4\}$; and 
	\item there is a point $\tilde{v}$ on the singular set $\gamma$ of $\orbQ$ such that the isotropy group at $\phi(\tilde{v})$ is $A_4$ or $S_4$, $\phi$ has local degree six at $\tilde{v}$, and $\phi$ is unbranched near $\tilde{v}$ on $\gamma- \tilde{v}$.
\end{itemize}\end{lem}

\begin{proof} If $\orbQ$ is arithmetic, then it is the unique torus-cusped orbifold non-trivially covered by the figure-eight knot complement  (appealing to the fact that the figure-eight knot complement is the only arithmetic knot complement \cite{ReidFig8}, and a computation of its symmetry group). In this very specific case, $gcd(p,q)=2$, $\orbQ$ is the quotient of the full group of isometries acting freely on the cusp, and $deg(\phi)=6$.

We will assume for the rest of proof that $\orbQ$ is non-arithmetic, so $\orb$ is as well. Therefore by \Cref{lem:ANR_vol_bound}, $\orb$ has volume at least $v_0/2$. On the other hand, $\orbQ$ has volume less than $6v_0$, being obtained from $N$, with $\vol(N)\leq 6v_0$, by hyperbolic Dehn filling. It follows that $\phi$ has degree strictly less than $12$. Since $\orbQ_3$ has 3-torsion on the cusp, the $deg(\phi)=3n$ for some $n \geq 1, n\in\ZZ$. Hence $\deg(\phi)\in \{3,6,9\}$.

Since $\orb$ has an $S^2(3,3,3)$ cusp, its isotopy graph has three ends of edges labeled $3$ running out of this cusp. We will call an edge \textit{peripheral} if it has such an end. The isotropy graph has at least one vertex that belongs to a peripheral edge, with the possible isotropy groups at such a vertex $v$ being dihedral of order six, $A_4$, $S_4$, or $A_5$. If the isotropy group of $v$ is dihedral then there are two endpoints of non-peripheral edges, each labeled $2$, at $v$. Under the cusp-killing homomorphism, the peripheral edge ending at $v$ would be removed and the remaining two would be joined to form a single edge, still labeled $2$. It follows that if all endpoints of peripheral edges had dihedral isotropy groups then the cusp killing homomorphism would have non-trivial image, contradicting that $\orb$ is covered by a knot complement. Thus at least one of these isotropy groups is one of $A_4$, $S_4$, or $A_5$. 

The degree of the orbifold cover $\phi\co \orbQ\to\orb$ is the sum, for any $v\in\orb$, of the local degrees at points $\phi^{-1}(v)$. In turn, the local degree at any $\tilde{v}\in\phi^{-1}(v)$ equals the index $[G:H]$, where $G$ is the isotropy group of $v$ and $H$ is the isotropy group of $\tilde{v}$, naturally a subgroup of $G$. Here for a point not on the isotropy graph, we take the isotropy group to be trivial. For any vertex $v$ with isotropy group equal to $A_4$, $S_4$ or $A_5$, this implies that every $\tilde{v}\in\phi^{-1}(v)$ lies in the isotropy graph of $\orbQ$, since the isotropy group of $v$ has order greater than the maximum possible degree $9$ of $\phi$.

The isotropy graph of $\orbQ$ is a closed geodesic, and the isotropy group at every point of it is cyclic of order $d = \gcd(p,q)$. As every cyclic subgroup of $A_5$ has index at least $12$, greater than any possible degree of $\phi$, it follows that $A_5$ is not an isotropy group of $\orb$. Taking $v$ to be a vertex of $\orb$ with isotropy group $A_4$ or $S_4$, and $\tilde{v}\in\phi^{-1}(v)$, necessarily in the isotropy graph of $\orbQ$, we further find that $d\in\{2,3,4\}$: this is the set of orders of cyclic subgroups of $A_4$ and $S_4$.

Finishing off the proof entails eliminating the possibilities that $\phi$ has degree $3$ or $9$ or that $d =3$, as well as verifying the final assertion about the behavior of $\phi$ near $\tilde{v}\in\phi^{-1}(v)$, for $v$ with isotropy group $G$ isomorphic to $A_4$ or $S_4$. These all follow by carefully considering the individual possibilities for $G$. The cyclic subgroups of $S_4$ have indices $6$, $8$ and $12$ only, so if $G \cong S_4$, the only possibility compatible with existing restrictions on the degree and local degree of $\phi$ is that $d = 4$, $\deg(\phi)=6$, and that this is also its local degree at the unique $\tilde{v}\in\phi^{-1}(v)$. Since $d=4$ the subgroup $H$ of $G$ corresponding to the isotropy group of $\tilde{v}$ is maximal cyclic, and it follows that $\phi$ is unbranched at  $\gamma-\tilde{v}$ is unbranched near $\tilde{v}$. Similarly, if $G\cong A_4$ then we must have $d=2$, $\deg(\phi)=6$, a unique $\tilde{v}\in\phi^{-1}(v)$, and no branching around points of $\gamma$ near but not at $\tilde{v}$.\end{proof}

We now proceed to the proof of \Cref{earem}, whose statement we first recall.

\begin{earemtheorem}\TheoremEarem\end{earemtheorem}

\begin{proof} The Theorem's second assertion follows directly from \Cref{prop:linking_number}, and its third from \Cref{lem:cusp_shapes_6_2_2_fillings}. (For the latter, recall that manifold fillings are those with $\gcd(p,q)=1$, and the Lemma itself pertains to orbifold fillings.) It remains to prove the first assertion, so suppose the $(p,q)$ filling slope has normalized length greater than $7.5832$, and this filling is covered by a hyperbolic knot complement with hidden symmetries.

We now appeal to Hodgson and Kerckhoff's detailed description of geodesic tubes (\cite{HK2008shape}, see also \cite[Theorem 2.7 and \S3]{haraway2019practical} for an application of Hodgson-Kerckhoff similar to the one here). Let $\gamma$ be the core of the surgery solid torus. By \cite[Corollary 5.13]{HK2008shape}, $\gamma$ is isotopic to a geodesic of length $\ell$  with $\ell \leq 0.156012$. We will abuse notation and consider $\gamma$ to be this geodesic. Theorem 5.7 of \cite{HK2008shape} further gives us that the (maximal) tube around $\gamma$ has radius $R_0$ with $R_0 > \arctanh(\frac{1}{\sqrt{3}}) \approx   0.6585$. 

First suppose that $\gcd(p,q)=1$, ie.~that the result of $(p,q)$-Dehn filling is a manifold $M$. $M$ has volume less than $6v_0$, by hypothesis and the hyperbolic Dehn filling theorem, so by \Cref{thm:volume_bound_gen} it must be the figure-eight knot complement. In particular, it is arithmetic, so by \cite[Corollary 4.7]{NeumReid} its shortest geodesic has length at least $.43127...$. But this contradicts the length bound above for $\gamma$.

Now suppose that $\gcd(p,q)=d >1$, so the result of $(p,q)$-Dehn filling is a one-cusped orbifold $\orbQ$ with a cone angle $2\pi/d$ around $\gamma$. If $\orbQ$ is covered by a knot complement with hidden symmetries then appealing to \cite[Theorems 1.1 and 1.2]{Hoffman_rigid_cusps}, it has an orbifold cover $\phi\co\orbQ\to \orb$ where $\orb$ has an $S^2(3,3,3)$ cusp. We now apply  \Cref{lem:cusp_shapes_6_2_2_fillings}, which implies that $d \in \{ 2,4\}$. %In particular 
It further gives  
a point $\tilde{v} \in\gamma$ such that $v=\phi(\tilde{v})$ is a vertex of the singular set of $\orb$ with isotropy group $A_4$ or $S_4$, $\phi$ has local degree $6$ at $\tilde{v}$, and $\phi$ is unbranched near $\tilde{v}$ on $\gamma-\tilde{v}$.

Because $\phi(\gamma)$ contains an edge of the isotropy graph of $\orb$ that terminates at $v$, $\phi$ is two-to-one on $\gamma$ near $\tilde{v}$. The image of $\gamma$ thus has length at most $\ell/2$. As $\phi$ has local degree six at $\tilde{v}$ but is unbranched on $\gamma-\tilde{v}$ near $\tilde{v}$, $\phi^{-1}(\phi(\gamma))$ contains geodesics other than $\gamma$ that meet at $\tilde{v}$. The total length of  $\phi^{-1}(\phi(\gamma))$ is at most $6(\ell/2)<.468036$. Given the lower bound on the tube radius $R_0$, it follows that none of the geodesics in $\phi^{-1}(\phi(\gamma))$ that meet at $\tilde{v}$ exit the tube around $\gamma$. But $\gamma$ is the unique geodesic entirely contained in this tube, contradicting that $\phi^{-1}(\phi(\gamma))$ contains a distinct geodesic meeting $\gamma$ at $\tilde{v}$. 
\end{proof}

An immediate consequence of \Cref{earem} is that for a fixed manifold there is a bounded number of slopes with those properties. Following Hodgson and Kerckhoff's universal bounds \cite[Corollary 1.4]{HK2005universal} which relies on a clever argument of Agol \cite[Lemma 8.2]{agol2000bounds}, we can give a concrete upper bound after observing that $7.5832^2 < 57.51 < 59$, so there are at most $60$ primitive slopes with normalized length less than the desired bound. Also, if $gcd(p,q)=2$, we can count slopes of the form $(p/2,q/2)$ which would be primitive and have length at most $7.5832/2$. Since  $(7.5832/2)^2 < 14.38 < 17$, there at at most $18$ of these slopes. Finally, if $gcd(p,q)=4$, there are at most $6$ slopes with normalized length less than the bound as $(7.5832/4)^2 < 3.6 < 5$. All told, we have at most $84$ possible slopes to check. We summarize this observation via the following:

\begin{cor}\label{cor:agol_bound}
Suppose $L = K \cup K'$ is a two-component link in $\Sth$, with $K'$ unknotted,
such that $N \cong \Sth \setminus L$ is hyperbolic with volume at most $6v_0$. Then there are at most %60+18+6
$84$ fillings that could result in an orbifold covered by a knot complement admitting hidden symmetries.  
\end{cor}

Of course, the bound of $84$ can be significantly reduced if the linking number between $K$ and $K'$ is small.

\begin{remark}\label{sub:covers_of_man_fillings}%bad label!
It is worth noting that comparing Corollary 5.13 of \cite{HK2008shape} with Corollary 4.7 of \cite{NeumReid} as in the proof above implies that the filling along any slope that has normalized length at least $7.5832$ is non-arithmetic. This reduces the problem of checking which fillings are arithmetic to a finite one that can be approached using rigorous computational methods.\end{remark}

%%%%%%%%%%%%%%%%%%
\subsection{Questions about orbifold fillings and concrete volume bounds.}\label{sub:orbifold_fillings}

Although \Cref{thm:volume_bound_gen} places a lower bound on the volume of non-arithmetic knot complements which admit hidden symmetries, it would be nice to place a lower bound on the volume of any torus cusped orbifold covered by a knot complement that can cover a rigid cusp. Making such a statement truly effective involves enumerating a (minimal) finite list of manifolds with the property that they are a complete set of ancestors for all hyperbolic orbifolds with volume less  $6v_0$. This problem is open; we refer the reader to \cite[Problem 10.16, Remark 10.17]{GMM2010mom} for more background (the problem list appears in the arxiv version). However, we ask the same questions for orbifolds of lower volume, which seems like a natural first step. In fact, $(2,0)$ surgery on the magic manifold $\Sth \setminus 6^3_1$ appears to be one of the smallest volume orbifolds with two torus cusps and singular set an embedded knot. This surgery has volume $\frac{vol(\Sth \setminus 6^3_1)}{2} \approx 2.66674478345.$

\begin{question}
Determine a complete (infinite) list of orbifolds $\{\orbQ_i\}$ with the following properties: \begin{enumerate}[label=(\roman*)]
\item  $\orbQ_i$ has a single torus cusp,
\item the singular set of $\orbQ_i$ is a set of embedded circles,
\item  $vol(\orbQ_i) \leq \frac{vol(\Sth \setminus 6^3_1)}{2} \approx 2.66674478345.$
\end{enumerate} 
\end{question}

Finally, by \cite[Corollary 1.2]{GMM2009minimum} (and a verified volume computation in sage), we find that we can focus on the case where the singular set of these orbifolds is non-empty, since the set of 1-cusped orientable manifolds with volume less than or equal to  $\frac{vol(\Sth \setminus 6^3_1)}{2}$ is exactly \{m003, m004, m006, m007, m009, m010\}. Although the verified computation involves interval arithmetic, one can show that $vol(m009)=vol(m010)= \frac{vol(\Sth \setminus 6^3_1)}{2}$ by finding common covers of m009, m010 and the Magic Manifold.

A classification of low volume orbifolds with a rigid cusp and non-rigid cusp similar to the classification of low volume rigid cusps orbifolds in \cite{Adams_small_vol} or limit orbifolds as in \cite{Adams_limit} could be leveraged with  \Cref{main_theorem} to improve the bounds of \Cref{thm:volume_bound_gen}. Along these lines, it seems the following question is open:

\begin{question}
For each type of rigid cusp, what is the smallest volume orbifold with that cusp and at least one non-rigid cusp?  Alternatively, what is is the small volume limit orbifold for an infinite sequence of distinct manifolds $\{\orbQ_n\}$ such that each $\orbQ_n$ has a rigid cusp? 
\end{question}

A candidate for the minimizer is given in \Cref{Fig236mins}(left), which Damian Heard's orb \cite{heard2005orb} gives (experimental) volume $\frac{5}{6}v_0 \approx 0.845784672$. However, we observe that any filling of this orbifold will have non-peripheral $\Z/2\Z$ homology and hence be non-trivial under the cusp killing homomorphism (a similar argument is used in \cite[Lemma 3.8]{HMW19}). 
Therefore, a natural modification of this question is to find the minimal volume example of an orbifold which admits fillings  that are trivial under the cusp killing map. A candidate for such a minimal volume example is given on the right side of \Cref{Fig236mins}. According to orb, this orbifold has volume $\approx 1.15111$.  In order to fill with tangles or orbi-tangles (as defined in \Cref{whatsadiorama?}) such that the resulting orbifold is trivial under the cusp killing homomorphism, we must apply the following rules for filling the (2,2,2,2) cusp. For tangles, the northwest corner is not connected to the southeast corner. For orbi-tangles, the central torsion is odd and no two torsion connects the northwest corner with the southeast corner.

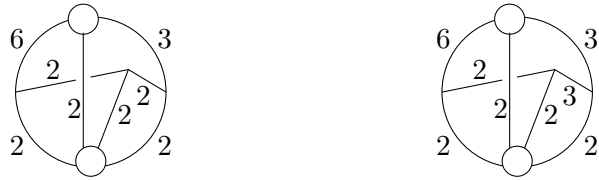
\begin{figure}
\begin{tikzpicture}
\draw (0,0) circle [radius=1];
\fill [color=white] (-.1,.97) circle [radius=0.2];
\draw (-.1,.97) circle [radius=0.2];
\draw (-1,0) -- (0.5,0.3) -- (1,0);
\fill [color=white] (-0.1,0.18) circle [radius=0.1];
\draw (-0.1,0.79) -- (-0.1,-0.8);
\draw (0.5,0.3) -- (0,-1);
\fill [color=white] (0,-.92) circle [radius=0.2];
\draw (0,-.92) circle [radius=0.2];
\node [above right] at (0.75,0.45) {$3$};
\node [above left] at (-0.75,0.45) {$6$};
\node [below left] at (-0.75,-0.45) {$2$};
\node [below right] at (0.75,-0.45) {$2$};
\node at (-0.22,-0.22) {$2$};
\node at (-0.5,0.3) {$2$};
\node at (0.45,-0.3) {$2$};
\node at (0.7,-0.05) {$2$};
\end{tikzpicture}
\hspace{3cm}
\begin{tikzpicture}
\draw (0,0) circle [radius=1];
\fill [color=white] (-.1,.97) circle [radius=0.2];
\draw (-.1,.97) circle [radius=0.2];
\draw (-1,0) -- (0.5,0.3) -- (1,0);
\fill [color=white] (-0.1,0.18) circle [radius=0.1];
\draw (-0.1,0.79) -- (-0.1,-0.8);
\draw (0.5,0.3) -- (0,-1);
\fill [color=white] (0,-.92) circle [radius=0.2];
\draw (0,-.92) circle [radius=0.2];
\node [above right] at (0.75,0.45) {$3$};
\node [above left] at (-0.75,0.45) {$6$};
\node [below left] at (-0.75,-0.45) {$2$};
\node [below right] at (0.75,-0.45) {$2$};
\node at (-0.22,-0.22) {$2$};
\node at (-0.5,0.3) {$2$};
\node at (0.45,-0.3) {$2$};
\node at (0.7,-0.05) {$3$};
\end{tikzpicture}
\caption{(left) A candidate for the minimum volume orbifold with a rigid and non-rigid cusp. (right) A candidate for the minimum volume orbifold with a rigid and non-rigid cusp which has fillings that are trivial under the cusp killing map.} 
\label{Fig236mins}
\end{figure}

%% file: sixtwotwo.tex
\section{A nice example}\label{sec:sixtwotwo}

Given a two-component link $L = K\sqcup K'$, with $K'$ unknotted, that satisfies the hypotheses of \Cref{earem}, the Theorem's conclusion leaves one with two tasks to execute in order to completely classify the fillings of the $K'$-cusp of $\Sth\setminus L$ that are covered by a knot complement with hidden symmetries. First, enumerate the filling slopes on this cusp that satisfy \Cref{earem}'s conclusions, in particular those with length less than $7.5832$. For each slope $\alpha$ on the resulting list, the second task is to individually check the $\alpha$-filling of the $K'$-cusp for the existence of the cover in question.

In this section we will completely execute the tasks laid out above for the two-bridge link named $6^2_2$ in the Rolfsen tables, proving:

\begin{thmn}[\ref{sixtwotwo}]\SixTwoTwo\end{thmn}

This is  the ``simplest new example'' in the sense that the Whitehead link $5^2_1$ is the only hyperbolic link with fewer crossings than $6^2_2$. The Whitehead link complement has been known since Neumann-Reid's work \cite[Section 6]{NeumReid} to have no fillings covered by a knot complement with hidden symmetries except one: the figure-eight knot complement. 

We will classify the relevant filling slopes for $\Sth\setminus 6^2_2$ using an explicit geometric triangulation of its cusp cross sections. We describe this triangulation, and the tetrahedral decomposition of the link complement that gives rise to it, in \Cref{triang} below. In giving the account there we also have an eye toward \Cref{defvar}, where we use the deformation variety of $\Sth\setminus 6^2_2$ to give an alternate proof of \Cref{sixtwotwo}. We give the current proof in \Cref{fillin} by applying rigorous computer-aided methods to the filling along each of the short list of slopes that satisfy all of \Cref{earem}'s criteria.

%%%%%%%%%%%%%%%%%%%%%%%%%%%%%%%%%
\subsection{Two triangulations of $\Sth\setminus 6_2^2$}\label{triang}

It is well known that $\Sth\setminus 6_2^2$ 

has a complete hyperbolic structure admitting a triangulation $\mathcal{T}_4$ by four regular ideal tetrahedra.  We obtain this triangulation from the six-tetrahedron triangulation $\mathcal{T}_6$ used in \cite{CheDeMonda}.  The tetrahedra of $\mathcal{T}_6$ are all isometric, with dihedral angles equal to $2\pi/3$ and $\pi/6$. 
That $\Sth\setminus 6_2^2$ has both triangulations is related to the fact that it is arithmetic and covers two different minimal orbifolds. These, called $Q_0$ and $Q_3$ in \cite{NR92b}, come from tilings of $\mathbb{H}^3$ by copies of the tetrahedra of $\mathcal{T}_4$ and $\mathcal{T}_6$, respectively. See \cite{NR92b} for further details.

A pair of 3-2 \textit{Pachner moves}, each of which

 replaces a union of three distinct tetrahedra sharing a degree 3 edge with a union of two glued along a face, transforms $\mathcal{T}_6$ to $\mathcal{T}_4$. We perform these along each of the two edges of $\mathcal{T}_6$, both with degree $3$, that are incident at both ends to the cusp $c_1$ of $\Sth\setminus 6_2^2$.

The triangulations of the cusps of $\Sth\setminus 6_2^2$ induced by $\mathcal{T}_4$ and $\mathcal{T}_6$ are depicted in \Cref{fig:cusp_triangs_1,fig:cusp_triangs_2}, respectively. On the left are those of $c_1$ (compare \Cref{fig:cusp_triangs_2} with \cite[Figure 5]{CheDeMonda}), and on the right are those of $c_2$.
For the cusp triangulations in \Cref{fig:cusp_triangs_2}, a prefered edge for each tetrahedron is marked by a dot, and the corresponding edge parameters are labelled following \cite[Section 3]{CheDeMonda}. The edge labelings of the right-hand cusp are obtained by applying 
a triangulation-preserving, cusp-exchanging involution, as described in \cite{CheDeMonda}, to the labels of the left-hand cusp.

\begin{figure}%[ht]

\includegraphics[scale=2.2]{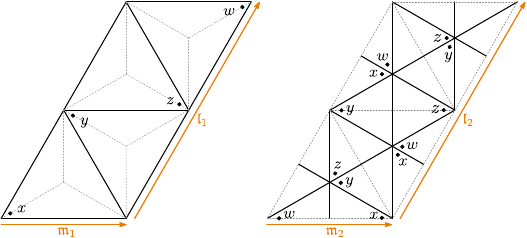}
\caption{ The triangulations of the two cusps of $S^3\setminus 6^2_2$ induced by $\mathcal{T}_4$ are indicated in solid black lines. In the discussion to follow, only the cusp on the right will be filled. The dots correspond to the edge with the preferred tetrahedral parameter. The dotted lines are edges of the cusp triangulations induced by $\mathcal{T}_6$ which are not edges of the $\mathcal{T}_4$ cusp triangulations. Compare to  \Cref{fig:cusp_triangs_2}.
}
\label{fig:cusp_triangs_1}
\end{figure}

\begin{figure}%[ht]
\includegraphics[scale=2.2]{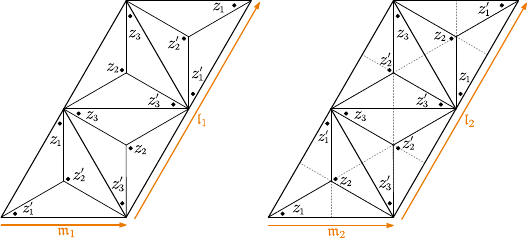}
\caption{ The triangulations of the two cusps of $S^3\setminus 6^2_2$ induced by  $\mathcal{T}_6$ are indicated in solid black lines. On the right, the dotted lines indicate edges of the cusp triangulation induced by $\mathcal{T}_4$ that are not edges of the triangulation induced by $\mathcal{T}_6$. Just as above in  \Cref{fig:cusp_triangs_1}, only the cusp on the right will be filled in the discussion to follow.
}
\label{fig:cusp_triangs_2}
\end{figure}

Horospherical cusp cross-sections for $\Sth\setminus 6_2^2$ are produced by identifying opposite sides of the parallelograms in the figure. In particular, the horizontal sides are identified to closed curves we will call $\mathfrak{m}_1$ on the left and $\mathfrak{m}_2$ on the right, and the diagonal sides are identified to curves $\mathfrak{l}_1$ and $\mathfrak{l}_2$. We use the same notation to refer to their homology classes.  

For a cusp $c_i$, $\mathfrak{m}_i$ and $\mathfrak{m}_i^{-1}\mathfrak{l}_i$ form a basis for its homology and we refer to these classes as meridians and longitudes respectively. It is clear from \Cref{fig:cusp_triangs_1} that $\{\mathfrak{m}_i,\mathfrak{m}_i^{-1}\mathfrak{l}_i\}$ is a \textit{geometric basis} for $H_1(c_i)$, consisting of the two shortest curves.

%%%%%%%%%%%%%%%%%%%%%%%%%%%
\subsection{The analysis of individual fillings}\label{fillin}

\begin{figure}[h]
	\begin{subfigure}{.32\textwidth}
 		\centering
   		\includegraphics[scale=.5]{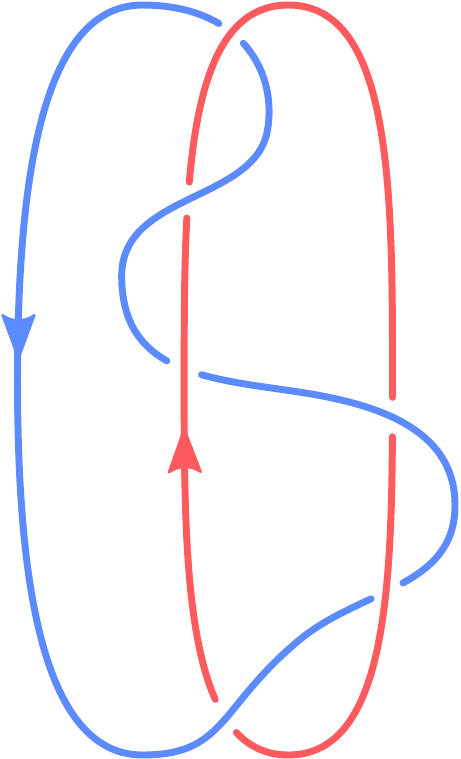}
   		\caption{}
   		\label{fig:sixtwotwo}
	\end{subfigure}
	\begin{subfigure}{.66\textwidth}
 		\centering
   		\includegraphics[scale=.5]{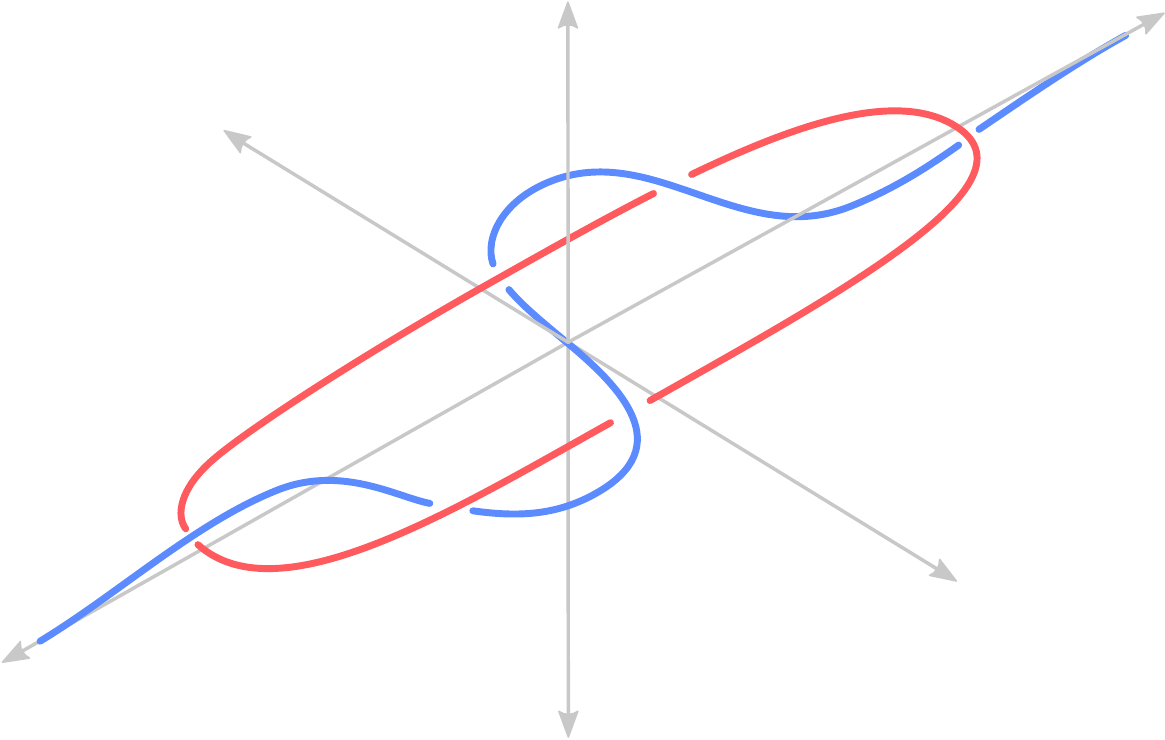}
   		\caption{}
   		\label{fig:sixtwotwo_isom}
	\end{subfigure}	
	\caption{Left: With the orientation shown, the linking number between components of the $6_2^2$ link complement is 3 (each crossing contributes $+1/2$). Right: If we identify $\Sth$ with $\R^3\cup\{\infty\}$, and embed the $6_2^2$ link as shown, with one component passing through $\infty$, then the involution $i:v\mapsto -v$ of $\R^3\cup \{\infty\}$ fixes the meridian of the blue component, and reverses the orientation of its longitude.}
	\label{fig:sixtwotwo}
\end{figure}

This subsection is devoted to the proof of \Cref{sixtwotwo}. For this we will use the ``meridian-longitude'' pair $\{\mathfrak{m}_i,\mathfrak{m}_i^{-1}\mathfrak{l}_i\}$ identified in \Cref{triang} to enumerate slopes satisfying the criteria of \Cref{earem}, but first we pause to resolve a possible dissonance. The coordinates $(p,q)$ for filling slopes are given in \Cref{earem} in terms of the standard homological framing, whereas our basis for $H_1(c_i)$ is geometric. But we claim that it does give a homological framing.

The triangulation $\mathcal{T}_6$ is the well-known Sakuma--Weeks triangulation, described in \cite{SW}. It follows from the construction of this triangulation that for each $i\in\{1,2\}$, $\mathfrak{m}_i$ is the meridian of the component $L_i$ of the $6_2^2$ link $L=L_1\cup L_2$ shown in \Cref{fig:sixtwotwo}.
 The longitude for $L_i$, i.e., the curve that bounds a disk in $\Sth\setminus L_i$, is $\mathfrak{m}_i^{-1}\mathfrak{l}_i$. To see this we use a trick of Thurston (see \cite[Section 4.6]{Th_notes}), observing that the there is an isometry of $\Sth\setminus 6_2^2$ that fixes $\mathfrak{m}_i$ and maps $\mathfrak{l}_i$ to $-\mathfrak{l}_i$, shown in \Cref{fig:sixtwotwo_isom}. Thus the meridian and longitude must be orthogonal, and referring to \Cref{fig:cusp_triangs_1} we conclude that the longitude is $\mathfrak{m}_i^{-1}\mathfrak{l}_i$.

\begin{remark}In fact, \Cref{earem}'s first and third criteria are basis-invariant. The first, on normalized length, is a geometric property of the slope.  And the third criterion regards the greatest common divisor of the coordinates $p$ and $q$, which is simply the natural number $d$ such that the slope is the $d$th power of a primitive element.
\end{remark}

Recall that the \textit{normalized length} of a closed geodesic $\gamma$ on a cross-section $T$ of $c_i$ is the length of $\gamma$ divided by the square root of the area of $T$ (cf.~\cite{HK2008shape}). Our geometric description thus gives:

\begin{fact}\label{smack't} For $i=1,2$ and any $(p,q)\in\mathbb{Z}^2-\{(0,0)\}$, $\mathfrak{m}_i^p(\mathfrak{m}_i^{-1}\mathfrak{l}_i)^q$ has normalized length:
\[ \frac{\sqrt{p^2+3q^2}}{3^{1/4}}. \]\end{fact}

With this in hand we proceed to the proof.

\begin{proof}[Proof of \Cref{sixtwotwo}] Suppose the $(p,q)$ filling of the cusp $c_1$ of $\Sth-6^2_2$ is covered by a knot complement with hidden symmetries (where the slope's coordinates are given in the meridian-longitude basis of \Cref{smack't}). Since $6^2_2$ has two unknotted components and $\Sth-6^2_2$ has volume $4v_0$, \Cref{earem} applies. By its first criterion and \Cref{smack't} we have $\sqrt{p^2+3q^2} < 7.5833\cdot3^{1/4}$. Its second criterion gives $\gcd(p,3) = 1$, since the components of $6^2_2$ have linking number $3$, and its third gives $\gcd(p,q)\in\{1,2,4\}$.

We address the case that $\gcd(p,q)=1$, ie.~of manifold fillings, first. Here by \Cref{thm:volume_bound_gen}, we must only check for each relevant $(p,q)$ is whether the $(p,q)$ filling of $c_1$ is isometric to the figure-eight knot complement. Before recording the list of relevant $(p,q)$, we point out that $S^3 \setminus 6^2_2$ admits an involution which sends $(p,q)$ to $(p,-q)$ (see \Cref{fig:sixtwotwo_isom}). Hence, we need only consider positive $p$ and $q$. The fillings that satisfy $\gcd(p,3)=1$, $\gcd(p,q)=1$, and having normalized length smaller than $7.5832...$ are as follows:
$$
\{(1,0),( 1,1),( 1,2),( 1,3),( 2,1),( 2,3),( 4,1),( 4,3),( 5,1),\\
( 5,2),( 7,1),( 7,2),(8,1)\}.
$$

Using SnapPy, we rigorously check that none of these fillings yields a manifold homeomorphic to the figure-eight knot complement. The fillings are listed in \Cref{tab:fillings} (code to compute this data is provided in the ancillary files of the arxiv post). The non-hyperbolic fillings can be distinguished by their fundamental groups (also compare to \cite[Table 7]{MartPet}, but note that the framing used there is non-standard).  
The hyperbolic fillings are all identifiable in SnapPy. They can be distinguished from the figure-eight knot complement by their invariant trace fields, which can be computed using the exact arithmetic functionality of SnapPy in Sage (see also the data files associated to arxiv version of this paper).

\begin{table}
\centering
\begin{tabular}{|c|c|c|c|c|c|}
\hline
\multicolumn{2}{|c|}{Non-hyperbolic fillings}
& \multicolumn{4}{|c|}{Hyperbolic fillings}\\
\hline
\hline
(1,0) & the solid torus& (1,2) & m016 & (5,1) &  m078  \\
\hline
(1,1) &$\quad$ (5,2) torus knot $\quad$& (1,3) & $\quad$m082 $\quad$& (5,2) & m115 \\
\hline
(2,1) & toroidal & (2,3) & m100 & (7,1) &  v0160\\
\hline
&&(4,1) & m034 & (7,2) & s081 \\
\hline
&&(4,3) & m121  & (8,1) &$\quad$ t00320 $\quad$\\
\hline
\end{tabular}
\caption{Hyperbolic and non-hyperbolic manifold fillings of the $6_2^2$ link complement.}
\label{tab:fillings}
\end{table}

We now consider orbifold fillings, ie.~where $\gcd(p,q)\in\{2,4\}$ by the above. The set of slopes with normalized length at most $7.5832$, with $gcd(p,q)\in\{2,4\}$ and $gcd(p,3)=1$, and as above, with $p, q\geq 0$, is as follows: 
\[ S = \{ (2, 0), (2, 2), (2, 4), (4, 0), (4, 2), (4, 4), (8, 2)\}. \]

The $(2,0)$ filling of $c_1$ is double covered by the figure-eight knot complement. The others can be eliminated individually using SnapPy's verified computations in Sage, or via deformation variety arguments (cf.~\Cref{embark}). For the former, note that SnapPy's current implementation  does not support exact arithmetic computations for orbifolds. To overcome this limitation, we construct the relevant 2-fold cyclic cover and fill along (p/2,q) for:  \{(2, 2), (2, 4), (4, 2), (8, 2)\} and a 4-fold cyclic cover filled along $(p/4,q)$ for \{(4, 0),(4,4)\}. This computation is included as an ancillary file for the arXiv posting.
\end{proof}

In some sense the $\Sth \setminus 6^2_2$ requires the most delicate analysis among the arithmetic two-component two-bridge links. As discussed earlier the Whitehead link complement cannot be filled to give a  knot complement other than the figure-eight knot complement which admits hidden symmetries. 

\begin{cor}\label{cor:all_arithmetic_tbl}
Let $N$ be the complement of an arithmetic two-bridge link. If $N(\alpha)$ is not the figure-eight knot or its unique torus cusped two-fold (orbifold) quotient, then $N(\alpha)$ is not covered by a knot complement admitting hidden symmetries.  
\end{cor}

\begin{proof}
As mentioned earlier, the components of the Whitehead link have linking number $0$ so only $(1,q)$ surgery can be covered by a knot complement. However, all such fillings are twist knot complements, and the only twist knot complement that admits hidden symmetries is the figure-eight knot complement \cite[Theorem 1.1]{ReidWalsh}).

Fillings of $\Sth\setminus 6^2_2$ are addressed by \Cref{sixtwotwo}.

It remains to show that no filling of $N=\Sth \setminus 6^2_3$ is covered by a knot admitting hidden symmetries.
The $vol( \Sth \setminus 6^2_3)\approx 5.3334895669< 6v_0$ and the linking number between the components is $2$. Applying \Cref{earem} in this case, shows we only need to consider fillings along the fifteen slopes: 

$$\{ (1, 0), (1, 1), (3, 1), (5, 1), (7, 1), (9, 1), (11, 1), (1, 2), (3, 2), (5, 2), (7, 2),$$
$$(9, 2), (1, 3), (5, 3), (1, 4) \}$$

$N(-,(1,0))$ is the solid torus, otherwise the fillings are hyperbolic. It can be observed that each of the fillings has a shape field admitting real embeddings (see ancillary files of the arXiv posting for example). Hence, none of these fillings can have quadratic imaginary subfields, which implies that none cover a rigid cusped orbifold (see \cite[Corollary 2.2]{ReidWalsh}).  
\end{proof}

%% file: deformation_arxiv.tex
%%%%%%%%%%
%%%%%%%%%%
\section{The deformation variety approach} \label{defvar}

In this section we will give an alternative proof of \Cref{sixtwotwo} using the \textit{deformation variety} of $\Sth\setminus 6^2_2$, an algebraic set parametrizing all hyperbolic structures on the link complement. Our approach in this section echoes that of Neumann--Reid in \cite[\S 6]{NeumReid}, which was applied there to the Whitehead link complement. One reason for pursuing this approach is that it yields extensive algebraic information about fillings of  $\Sth\setminus 6_2^2$. For instance \Cref{gary} here classifies all fillings of one cusp that produce an orbifold with imaginary quadratic cusp field, which comprise the arithmetic such fillings by \Cref{arith}. \Cref{trfld} and \Cref{gareth} in fact show that the cusp, trace, and invariant trace fields coincide among almost all orbifolds produced by filling a single cusp of $\Sth\setminus 6^2_2$.

\Cref{sixtwotwo} and the related results from this section concern the $(p,q)$ orbifold fillings of the cusp $c_2$ where $(p,q) \in \Z \times \Z -\{ (0,0)\}$.  As in the previous section's %Section \ref{sec:sixtwotwo}'s 
proof of  \Cref{sixtwotwo}, the fillings given by $(p,q)$, $(-p,q)$, $(p,-q)$, and $(-p,-q)$ are all homeomorphic to each other.  So, from now on, we will assume that $p$ and $q$ are both non-negative.

The deformation variety used here is determined by the four-tetrahedron triangulation $\mathcal{T}_4$ of $\Sth\setminus 6^2_2$ described in \Cref{triang}. In \Cref{defection} we leverage computations already carried out in \cite[$\S$3]{CheDeMonda}, which used the six-tetrahedron triangulation $\mathcal{T}_6$ of \Cref{triang}, to give the relevant deformation variety equations. The changeover from $\mathcal{T}_6$ to $\mathcal{T}_4$ yields some useful simplifications, which we exploit in \Cref{election} to give our second proof of \Cref{sixtwotwo}.

%%%%%%%%%%%%%%%%%%
\subsection{The deformation variety and its Dehn filling points}\label{defection} We are ready to describe the deformation variety $V$ determined by $\mathcal{T}_4$ as well as its subvariety $V_0$ of geometric structures for which the cusp $c_1$ remains complete. Here we further describe the function $\tau_{c_1}$ on $V_0$ that measures the cusp parameter of $c_1$ and the equations determining points of $V_0$ corresponding to the fillings of $c_2$. See equations %(\ref{cuts out}), 
(\ref{cusp param}) and (\ref{surgery eqn}) below.

To record the deformation variety equations, we choose a preferred edge of each tetrahedron as indicated by the dots in \Cref{fig:cusp_triangs_1,fig:cusp_triangs_2}.  For $\mathcal{T}_4$ we denote the associated edge parameters $x$, $y$, $z$ and $w$, as labelled in \Cref{fig:cusp_triangs_1}.   As in \cite[\S 3]{CheDeMonda}, we take the edge parameters for $\mathcal{T}_6$ to be $z_i$ and $z_i'$ for $i \in \{ 1,2,3\}$.  Define $\zeta_1$ to be the rational function $\zeta_1(z) = \frac{1}{1-z}$ and $\zeta_2=\zeta_1^2$.  With this definition, the cusp parameters of the three edges of $\mathcal{T}_4$ (or $\mathcal{T}_6$) corresponding to three vertices of a triangle in the cusp triangulation form a cyclically ordered triple $(z, \zeta_1(z), \zeta_2(z))$ when recorded counter-clockwise around the triangle.

 We have the following relationships between the edge parameters of $\mathcal{T}_4$ and the edge parameters of $\mathcal{T}_6$, which can be deduced from the labelling in \Cref{fig:cusp_triangs_1,fig:cusp_triangs_2}:
 \begin{align*}
 & x = z_1'\zeta_1(z_1) && y = z_3\zeta_2(z_3') && z = z_3'\zeta_2(z_3) && w = z_1\zeta_1(z_1') \end{align*}
It is possible to solve these for $z_1$, $z_1'$, $z_3$ and $z_3'$ in terms of $x$, $y$, $z$ and $w$, yielding:\begin{align*}
 & z_1 = \frac{w(1-x)}{1-xw} && z_1' = \frac{x(1-w)}{1-xw} && z_3 = \frac{1-zy}{1-z} && z_3' = \frac{1-zy}{1-y} \end{align*}
With this the deformation variety equations associated to $\mathcal{T}_4$ can be obtained from those associated to $\mathcal{T}_6$ by substitution. The equations labeled (1/2) in \cite[\S 3]{CheDeMonda} yield $z_2 = z_2' = xw$; those labeled (1/4) yield $z_2 = z_2' = yz$, from which we conclude that $xw = yz$. After substitution and simplification, those labeled (0:1) and (1/3:2/7) both give
\begin{align}\label{big variety} (1-x)(1-y)(1-z)(1-w) = xw = yz \end{align}
This yields an irreducible degree-five polynomial in $x$, $y$ and $z$, say, after substituting $w=yz/x$ above, which determines a deformation variety $V$.

The subvariety $V_0$ of the deformation variety where $c_1$ is complete is determined by the equation $\mu(\mathfrak{m}_1) = 1$, where $\mu$ is the derivative of the holonomy, and $\mathfrak{m}_1$ the simple closed curve on $c_1$ shown in \Cref{fig:cusp_triangs_1}. For those unfamiliar, $\mu$ is explicitly defined by Neumann--Zagier in \cite{NZ}, and re-recorded as equation (3) in \cite[\S 1.2]{CheDeMonda}. It can be computed directly from the definition, or by substituting in the formula for $\mu(\mathfrak{m}_1)$ in \cite[\S 3]{CheDeMonda} ($\mu(\mathfrak{m})$ there), that 
\begin{align*} \mu(\mathfrak{m}_1) = - \frac{w}{1-w} \frac{1}{1-z}. \end{align*}
This means that $z = 1/(1-w)$ on $V_0$ which, together with the fact that $zy=xw$, implies that $y = xw(1-w)$.  When these substitutions are made in \Cref{big variety}, it becomes $P=0$ where
\begin{align*}%\label{cuts out} 
P&=1 - x(1-x)w(1-w). \end{align*}
This realizes $V_0$ as an affine algebraic curve in $\C^2$ with coordinates $w$ and $x$.

Recall that $\Sth\setminus 6_2^2$ admits an involution which exchanges $w$ and $x$.  Under the complete structure, this involution is represented by an isometry; hence $w=x$ for any point of $V$ which represents this structure.  The $w=x$ subset of $V$ is the four point set with $w \in \left\{ \frac12 (1 \pm i\sqrt{3}), \frac12 (1 \pm \sqrt{5}) \right\}$.   Therefore, the positively oriented complete structure on $\Sth \setminus 6_2^2$ corresponds exactly to the point $w=x=\frac12 (1 + i \sqrt{3})$ and each tetrahedron is regular.  The points $w=x=\frac12 (1 \pm \sqrt{5})$ lie in $\R^2$ and do not represent hyperbolic fillings.

%Note also that $V_0$ has involutions given by exchanging $x$ and $w$ (which corresponds to an involution of $6^2_2$), and by exchanging $x$ with $1-x$ and/or $w$ with $1-w$, as well as by exchanging $x$ and $w$ with their complex conjugates.

We now identify the cusp parameter function of the cusp $c_1$. Equation (4) in the proof of \cite[Prop.~1.6]{CheDeMonda} defines, for each $\mathfrak{g}\in H_1(T)$, a function $\tau(\mathfrak{g}):V_0\to\C$. Letting the reference edge $f$ coincide with the curve $\mathfrak{m}_1$, we get the constant map $\tau(\mathfrak{m}_1) = 1$. Taking the longitude $\mathfrak{l}_1$ to be as shown in \Cref{fig:cusp_triangs_1,fig:cusp_triangs_2}, we obtain;
\[ \tau(\mathfrak{l}_1) = x - x\zeta_2(x) y \zeta_2(z) = x- (x-1)xw(1-w)\,w. \]
The latter equality holds on $V_0$ where we have substituted $z = 1/(1-w)$ and $y = xw(1-w)$.  It follows that $\tau(\mathfrak{l}_1) = x+w$ in the coordinate ring $\C[V_0]$ for $V_0$.  Hence, using the rational function given in \cite[Prop.~1.6(3), Def. 1.7]{CheDeMonda}, we find that the cusp parameter function $\tau_{c_1} \in \C[V_0]$ for $c_1$ is
\begin{align}\label{cusp param} \tau_{c_1} = \frac{\tau(\mathfrak{l}_1)}{\tau(\mathfrak{m}_1)} = x+w. \end{align}

We now record $\mu(\mathfrak{m}_2)$ and $\mu(\mathfrak{l}_2)$ for curves $\mathfrak{m}_2$ and $\mathfrak{l}_2$ that form a homology basis for the second cusp $c_2$, pictured on the right in \Cref{fig:cusp_triangs_1,fig:cusp_triangs_2}. We have:
\begin{align*} 
   \mu(\mathfrak{m}_2) & = -\zeta_1(z_2')z_3\zeta_1(z_2)\zeta_1(z_1')z_1'\zeta_1(z_1) = \frac{x(1-w)}{w(1-x)} = x^2(1-w)^2 = \frac{1}{w^2(1-x)^2} \\
   \mu(\mathfrak{l}_2) & = z_1\zeta_1(z_1')z_1'\zeta_2(z_2)\zeta_2(z_3')z_3\zeta_2(z_3')\zeta_2(z_2')\zeta_1(z_2')z_3\zeta_1(z_2)\zeta_1(z_1') \\
   & = \frac{xw(1-w)^3}{1-x} = \frac{(1-w)^2}{(1-x)^2} \end{align*}
The final equalities for both $\mu(\mathfrak{m}_2)$ and $\mu(\mathfrak{l}_2)$ above use that $P=0$ and are valid only in the coordinate ring for $V_0$. It will be convenient below to replace $\mathfrak{l}_2$ with the alternative basis element $\mathfrak{m}_2^{-1}\mathfrak{l}_2$ for $H_1(c_2)$, for which we have:\begin{align*}
   \mu(\mathfrak{m}_2^{-1}\mathfrak{l}_2) & = \frac{w(1-w)}{x(1-x)} = w^2(1-w)^2 = \frac{1}{x^2(1-x)^2} \end{align*}
   
Using the basis $\{\mathfrak{m}_2,\mathfrak{m}_2^{-1}\mathfrak{l}_2\}$ for $H_2(c_2)$, at $(p,q)$-surgery on $c_2$ we have
\begin{align}\label{surgery eqn} \left[\frac{x(1-w)}{w(1-x)}\right]^p\left[\frac{w(1-w)}{x(1-x)}\right]^q = 1. \end{align}

%This collection of equations has some symmetries. First, note that $(x,w)$ satisfies \Cref{surgery eqn} for a particular $(p,q)$ if and only it satisfies it for $(-p,-q)$, reflecting that Dehn filling is indifferent to slope orientation. For any fixed $(p,q)$, \Cref{surgery eqn} is also simultaneously satisfied by $(x,w)$ and $(\bar{x},\bar{w})$ in $V_0$.
%
%For the final observation, recall from below \Cref{cuts out} that $(x,w)\in V_0$ if and only if $(1-x,1-w)\in V_0$. Substituting the latter for the former above, we obtain:
%
%\begin{fact}\label{act} A point $(x,w)\in V_0$ satisfies the Dehn surgery \Cref{surgery eqn} for a particular $(p,q)\in\mathbb{Z}^2$ if and only if $(x',w') = (1-\bar{x},1-\bar{w})\in V_0$ satisfies it for $(p',q')$, where $p'=-p$ and $q'=q$.\end{fact}
%
%We have used complex conjugates here since for all positively oriented hyperbolic structures on $S^3-6_2^2$, all tetrahedral parameters lie in the upper half-plane.

%\eric{***I edited the statement of the next lemma slightly and changed  the exposition of the proof.  The ideas remain the same as before.  I also moved this lemma forward from the following subsection.  I made enough edits to the remainder of Section 6, that it seemed silly to make it all {\em Cheesey} colored, but still I stayed true to Jason's brilliant argument.  The old verson is commented out in the tex file, so we can recover any of its parts that are missed.***}

We note now that the polynomial $P$ defining $V_0$ and the cusp parameter function $\tau_{c_1}\in\C[V_0]$ are each symmetric in $x$ and $w$. While the Dehn surgery equation (\ref{surgery eqn}) above is not symmetric in $x$ and $w$, our final observation here is that it has a symmetric factor satisfied by any point of $V_0$ corresponding to a hyperbolic Dehn filling of the cusp $c_2$.

\begin{lem}\label{cable} 
Suppose that $(p,q)$ is a pair of non-negative integers which are not both zero.  There is a symmetric polynomial $g_{(p,q)} \in \Z[w,x]$ such that $g_{(p,q)}(w,x)=0$ whenever $(w,x) \in V_0$ corresponds to a $(p,q)$-hyperbolic Dehn filling of $c_2$.
\end{lem}

\begin{proof}
This proof utilizes the binomial theorem to express \Cref{surgery eqn} 
in terms of the homogeneous symmetric polynomials.  

For a non-negative integer $d$, let $s_d$ be the degree-$d$ homogeneous symmetric polynomial: 
\[ s_d = \sum_{j=0}^d w^j x^{d-j}.\]
Define also $s_{-1}=0$.  After expanding and simplifying in the $d\geq 0$ case, we find that
\begin{align}\label{sym identity}
(w-x) s_d &= w^{d+1}-x^{d+1}.
\end{align}
This formula also holds in the $d=-1$ case.  From the binomial theorem,
\begin{align} \label{bin}
(1-x)^d&= \sum_{j=0}^d  {d\choose j}  (-x)^j  = 1+ \sum_{j=0}^{d-1} {d \choose j+1} (-x)^{j+1}. 
\end{align}
For $(p,q)$ as in the statement of the lemma, $(p,q) \in R_1 \cup R_2$ where
\[ R_1 = \left\{ (p,q) \in \Z \times \Z \, \big| \, 0 \leq p \leq q \right\} \quad \text{and} \quad R_2 = \left\{ (p,q) \in \Z \times \Z \, \big| \, 0 \leq q < p \right\}. \]

First, suppose that $(p,q) \in R_1 - \{ (0,0) \}$.  Then, using (\ref{sym identity}) and (\ref{bin}),  \Cref{surgery eqn} is equivalent to 
\begin{align*}
0&= w^{q-p} (1-w)^{p+q} - x^{q-p} (1-x)^{p+q} \\
&= \sum_{j=0}^{p+q} {p+q \choose j} (-1)^j \left( w^{q-p+j} - x^{q-p+j} \right) \\
&= (w-x) g_{(p,q)}
\end{align*}
where
\begin{align} \label{R1} g_{(p,q)} &= \sum_{j=0}^{p+q} {p+q \choose j} (-1)^j s_{q-p+j-1}. \end{align}

Now, suppose that $(p,q) \in R_2$ and let
\begin{align} \label{R2}  g_{(p,q)} &=  \sum_{j=0}^{p-q-1} {p+q \choose j} (-wx)^j  s_{p-q-j-1} - \sum_{j=p-q+1}^{p+q} {p+q \choose j} (-1)^j (wx)^{p-q} s_{j+q-p-1}.\end{align}
(If $q=0$, the latter sum is empty.) Using (\ref{sym identity}) and (\ref{bin}) again,  \Cref{surgery eqn} is equivalent to 
\begin{align*}
0&= w^{p-q} (1-x)^{p+q} - x^{p-q} (1-w)^{p+q} \\
&= \sum_{j=0}^{p+q} {p+q \choose j} (-1)^j \left( w^{p-q}x^j - x^{p-q}w^j \right) \\
&= (w-x) g_{p,q)}.
\end{align*}

As discussed above the Lemma, the only points of $V_0$ with $w=x$ correspond to the complete structure or lie in $\mathbb{R}^2$. Therefore if $(w,x)$ corresponds to a $(p,q)$-hyperbolic Dehn filling of $c_2$, hence satisfying \Cref{surgery eqn}, then it satisfies $g_{(p,q)}(w,x)=0$.\end{proof}

\begin{remark}\label{embark} The fillings in the exceptional set $S$ from the previous sub-section's proof of \Cref{sixtwotwo} can be analyzed using Groebner bases.  For a given $s \in S$, take $P$ and $g_s$ as above and let $C$ be the polynomial $w+x-t$.  Next, compute a Groebner basis for the ideal $(P, C, g_s) \subset \Z[t,w,x]$ to obtain a polynomial $T_s \in \Z[t]$ in which the variable $t$ represents the cusp parameter $\tau_{c_1}$ for the corresponding filling of $\Sth\setminus 6_2^2$. Except for $s=(2,0)$, each $T_s$ is irreducible and of degree at least four, showing that the cusp parameter is not quadratic imaginary in these cases.
\end{remark}

\subsection{The proof of \Cref{sixtwotwo} via deformation varieties}\label{election} 
We now turn to a direct proof of all cases of \Cref{sixtwotwo}. The methods applied here rely on further study of the deformation variety of $\mathcal{T}_4$. In contrast to the previous proof, they are agnostic to manifold vs orbifold fillings and we need not worry if $p$ and $3$ are relatively prime.  Instead of considering geodesics arising from fillings, we will exploit the fact that all important data above is carried by polynomials that are symmetric in $x$ and $w$.  

Using coordinates $(\sigma_1, \sigma_2)$ on the image, let $\phi \co \C^2 \to \C^2$ be the regular map $(w,x) \mapsto (w+x, wx)$.  This induces a ring homomorphism $\phi_\ast \co \Z[\sigma_1, \sigma_2] \to \Z[w,x]$ which substitutes $w+x$ and $wx$ for $\sigma_1$ and $\sigma_2$.  Although $\phi_\ast$ is injective, it is not surjective since every element in its image must be symmetric  in $w$ and $x$.  However, $P \in \text{Im}(\phi_\ast)$.  In particular, if we take
\begin{align*}%\label{symm cuts out}
	f = 1 + \sigma_1\sigma_2 - \sigma_2 -\sigma_2^2
\end{align*}
then $\phi_\ast f = P$.  Therefore, if we take $U_0$ to be the affine curve determined by $f=0$ then $\phi$ restricts to a regular map $V_0 \to U_0$ and $\phi_\ast$ descends to the respective coordinate rings. Notice that $\phi_\ast \sigma_1 = \tau_{c_1}$.

\begin{remark} \label{cusp field} If $\phi(w,x) = (\sigma_1, \sigma_2) \in U_0$, where $(w,x)\in V_0$ corresponds to $(p, q)$-hyperbolic Dehn filling of $c_2$, then the cusp field of the unfilled cusp $c_1$ is $\mathbb{Q}(\sigma_1)$. \end{remark}

The key to proving \Cref{sixtwotwo} via deformation variety arguments is now reduced to computing which Dehn filling parameters yield quadratic imaginary fields for $\mathbb{Q}(\sigma_1)$.  After analyzing the equations that define the corresponding points in $U_0$, we will see that $\mathbb{Q}(\sigma_1)$ is only quadratic imaginary for the parameters $(2,0)$ and $(3,0)$. The remainder of this section is dedicated to establishing this.

Because the symmetric polynomials $s_d$ from the proof of \Cref{cable} lie in the image of $\phi_\ast$ (see \cite[$\S$14.6]{dummitfoote2004}), the next result follows from the descriptions of the $g_{(p,q)}$ there.

\begin{cor}\label{barry} 
For a pair $(p,q)$ of non-negative integers which are not both zero, there is a polynomial $h_{(p,q)} \in \Z[\sigma_1, \sigma_2]$ with $\phi_\ast h_{(p,q)} = g_{(p,q)}$. Thus if $\phi(w,x) = (\sigma_1, \sigma_2) \in U_0$, where $(w,x)\in V_0$ corresponds to $(p, q)$-hyperbolic Dehn filling of $c_2$, then $h_{(p,q)}(\sigma_1,\sigma_2) = 0$.
\end{cor}

We will say that a point $(\sigma_1,\sigma_2)\in U_0$ \textit{corresponds to $(p,q)$-hyperbolic Dehn filling of $c_2$} if it is the image of such a Dehn filling point $(w,x)\in V_0$, as in \Cref{cusp field} and \Cref{barry}. To begin analyzing the polynomials $h_{(p,q)}$ that these points satisfy, we first reinterpret the symmetric polynomials $s_d$ in terms of $\sigma_1$ and $\sigma_2$.

\begin{lem}\label{able} 
The polynomials $t_d \in \Z[\sigma_1, \sigma_2]$ such that $\phi_\ast t_d = s_d$ have the form $t_{-1} = 0$, $t_0 = 1$, $t_1 = \sigma_1$, and for $d>1$:
\[ t_d = \sigma_1^d + n_1 \sigma_1^{d-2}\sigma_2 + n_2\sigma_1^{d-4}\sigma_2^2 + \hdots + n_{\lfloor d/2\rfloor}\sigma_1^{\delta}\sigma_2^{\lfloor d/2\rfloor}, \]
where $n_j \in \Z$ and  $\delta = 0$ if $d$ is even and $\delta = 1$ if $d$ is odd.  Also,  $n_{1} = -d+1$.
\end{lem}

\begin{proof} Since $s_d$ is homogeneous of degree $d$ and $\sigma_1$ and $\sigma_2$ have degree $1$ and $2$ as functions in $w$ and $x$, the powers of $\sigma_1$ that occur all have parity matching that of $d$. 

There is a monomial term $\sigma_1^d$ of $t_d$ that accounts for the terms $w^d$ and $x^d$ of $s_d$.  The terms of $\phi_\ast \sigma_1^d$  that are linear in one variable are $d\,x^{d-1}w$ and $d\,xw^{d-1}$, so to obtain the corresponding terms of $s_d$ it must be true that $c_{1,d} = -d+1$.\end{proof}

Recall that $U_0$ is the algebraic curve on which $f=0$.  Since this equation can be solved for $\sigma_1$ as $\sigma_1 = \sigma_2+1-1/\sigma_2$, $U_0$ is the graph of the function $\sigma_2 \mapsto \sigma_2+1-1/\sigma_2$.   So, the rational function $\psi \co \C \to \C^2$ given by $\sigma_2 \mapsto (\sigma_2+1-1/\sigma_2, \sigma_2)$ is birational onto $U_0$.   Moreover, the induced map between function fields $\psi_\ast \co \C(U_0) \to \C(\sigma_2)$ is an isomorphism.

Lemma \ref{able} allows us to understand the rational functions $\psi_\ast h_{(p,q)}$ of one variable %polynomials 
well enough to prove the next proposition, a key pillar of our argument.

\begin{prop}\label{aiu} 
Suppose $(p,q)$ is a pair of non-negative integers with $q \neq 0$ and $p \neq 2q$. If $(\sigma_1,\sigma_2)$ corresponds to $(p,q)$-hyperbolic Dehn filling of $c_2$ then $\sigma_2$ is a unit algebraic integer.
\end{prop}

\begin{proof}%[Proof of \Cref{aiu}] 
Take $(p,q)$ and $(\sigma_1,\sigma_2)$ as in the statement of the proposition.  The strategy is to factor and simplify $\psi_\ast h_{(p,q)}$ letting $h'_{(p,q)} \in \Z[\sigma_2]$ represent the numerator of the result.   Since $\sigma_2$ must satisfy $h'_{(p,q)}$, it suffices to show that the leading and constant coefficients of $h'_{(p,q)}$ are both $\pm 1$.

First, if $k$ and $l$ are non-negative integers then, as a Laurent polynomial in $\sigma_2$, the term of highest degree in $\sigma_2^k ( \sigma_2 +1-\sigma_2^{-1})^l$ is $\pm \sigma_2^{k+l}$ and its term of lowest degree is $\pm \sigma_2^{k-l}$.  So, using Lemma \ref{able}, the term of highest degree in $\psi_\ast t_d$ is $\pm \sigma_2^d$ and its term of lowest degree is $\pm \sigma_2^{-d}$.

Now suppose that $(p,q) \in R_1$, for $R_1$ as defined in the proof of Lemma \ref{cable}.  From (\ref{R1}), 
\begin{align} \label{hR1} h_{(p,q)} = \sum_{j=0}^{p+q} {p+q \choose j} (-1)^j t_{q-p+j-1}.\end{align}
So,  $\psi_\ast h_{(p,q)} = \sigma_2^{-2q-1} h'_{(p,q)}$, where $h'_{(p,q)}$ is a polynomial in $\Z[\sigma_2]$ with leading term $\pm \sigma_2^{4q-2}$ and constant term $\pm 1$.  This established the lemma except in the case that $(p,q) \in R_2$.

Suppose $(p,q) \in R_2$, ie.~that $0\leq q<p$.  From (\ref{R2}),
\begin{align} \label{hR2}
h_{(p,q)} =  \sum_{j=0}^{p-q-1} {p+q \choose j} (-\sigma_2)^j  t_{p-q-1-j} - \sum_{j=p-q+1}^{p+q} {p+q \choose j} (-1)^j \sigma_2^{p-q} t_{j-(p-q-1)}.
\end{align}
The term of maximal degree in the Laurent polynomial $\sigma_2^j \psi_\ast t_{p-q-j-1}$ is $\pm \sigma_2^{p-q-1}$ and the term of minimal degree is $\pm \sigma_2^{2j+q-p+1}$.  With an eye on the second sum in (\ref{hR2}), notice that the term of maximal degree in the Laurent polynomial $\sigma_2^{p-q} \psi_\ast t_{j+q-p-1}$ is $\pm \sigma_2^{j-1}$ and the term of minimal degree is $\pm \sigma_2^{2p-2q+1-j}$.  

Provided that $q\neq 0$, the term of maximal degree in the Laurent polynomial $\psi_\ast h_{(p,q)}$ is $\pm \sigma_2^{p+1-q}$.  If $q=0$ then the term of maximal degree in this Laurent polynomial is 
\[\sigma_2^{p-1} \sum_0^{p-1} { p \choose j} (-1)^j = (-1)^{p-1}.\]
So, regardless of $q$, the leading coefficient of the polynomial numerator of $\psi_\ast h_{(p,q)}$ is $\pm 1$.

If we apply $\psi_\ast$ to the first sum in (\ref{hR2}), the term of minimal degree in the resulting Laurent polynomial is the $j=0$ term, $\pm \sigma_2^{q-p+1}$.  For the second sum, the term of minimal degree is the $j=p+q$ term, $\pm \sigma_2^{p-3q+1}$.  The degrees of these two terms are distinct unless $p=2q$.   
\end{proof}

\begin{remark} In the excluded case $p=2q$,  both of the sums in (\ref{hR2}) have a term which contributes to the constant term of $h'_{(p,q)}$.  These terms may have the same sign, whence $\sigma_2$ is an algebraic integer divisor of two, or they may be opposite, in which case analysis becomes more difficult.\end{remark}

\begin{remark} \label{no}
In the case that $(p,q)=(1,1)$, the formula in the above proof gives $h_{(1,1)}=\sigma_1-2$.  So, by Remark \ref{cusp field} and the fact that the cusp field for a hyperbolic orbifold cannot be real, the $(1,1)$-filling of $c_2$ is not hyperbolic.
\end{remark}

Since hidden symmetries give information about the cusp parameter $\sigma_1$, we need a robust connection between values of $\sigma_1$ and $\sigma_2$. The next lemma gives a tool for finding one.

\begin{lem}\label{fable} 
Suppose $h \in \Q[\sigma_1, \sigma_2]$ is not represented in $\C[U_0]=\Q[\sigma_1, \sigma_2]/(f)$ by a polynomial in $\Q[\sigma_1]$.  If $(\sigma_1, \sigma_2) \in U_0$ satisfies $h$ then $\Q(\sigma_1)= \Q(\sigma_2)$.
\end{lem}

\proof
We've already seen that the equation $f(\sigma_1,\sigma_2)=0$ can be used to write $\sigma_1$ as a rational function in $\sigma_2$. Since $U_0$ is the zero locus of $f$, it follows for any $(\sigma_1,\sigma_2)\in U_0$ that $\mathbb{Q}(\sigma_1)\subseteq\mathbb{Q}(\sigma_2)$.

Recall that $f$ has $\sigma_2$-degree two.  For $h$ as above, use multivariable polynomial division under
the lexicographic ordering with $\sigma_2>\sigma_1$ to divide $h$ by $f$. (For background on this, see Section 3 in Chapter 2 of \cite{CLO}.) The polynomial remainder $r$ represents $h$ in $\C[U_0]$ and has $\sigma_2$-degree at most one.  By assumption, this degree cannot be zero.  Therefore, $r=0$ can be solved for $\sigma_2$ expressing $\sigma_2$ as a rational function in $\sigma_1$.    
\endproof

Using \Cref{fable} we establish the second key pillar of our argument.

\begin{prop}\label{carrie} 
Suppose $(p,q)$ is a pair of non-negative integers with $(p,q) \neq (1,1)$ and $q \neq 0$.  If $(\sigma_1, \sigma_2) \in U_0$ satisfies $h_{(p,q)}$ then $\mathbb{Q}(\sigma_1)=\mathbb{Q}(\sigma_2)$.\end{prop}

To prove Proposition \ref{carrie}, it will helpful to have the following degree criterion for when Lemma \ref{fable} applies.

\begin{lem} \label{crit}
Take $f=1+\sigma_1 \sigma_2 -\sigma_2-\sigma_2^2$ and suppose $h,g \in \Q[\sigma_1, \sigma_2]$ with $h-gf \in \Q[\sigma_1]$.  Write
\[ g=\sum_{j=0}^{m-1} \sigma_1^j k_j \quad \text{and} \quad h=\sum_{j=0}^{n} \sigma_1^j \ell_j \]
where $k_j, \ell_j\in\mathbb{Q}[\sigma_2]$ for all relevant $j$, and $k_{m-1}$ and $\ell_n$ are non-zero.  Then $n \geq m$, and
\begin{enumerate}[label=(\roman*)]
\item the leading terms of $\ell_0$ and $\left(1-\sigma_2-\sigma_2^2\right) k_0$ are equal,
 \item the leading terms of $\sigma_2 k_{j-1}$ and $\ell_j - (1-\sigma_2-\sigma_2^2)k_j$ are equal whenever  $1 \leq j \leq m-1$, 
\item the leading terms of $\ell_m$ and $\sigma_2 k_{m-1}$ are equal, and
\item $\deg(\ell_j) =0$ for $j \geq m+1$.
\end{enumerate}
In particular, the polynomials compared in each of conditions (i)-(iii) have equal degrees.
\end{lem}

%\jason{***I strengthened the statement a bit for use in the argument at the end of Prop. 6.10.***}

\proof
We have
\[ gf = (1-\sigma_2-\sigma_2^2) k_0 + \sum_{j=1}^{m-1} \sigma_1^j \left[(1-\sigma_2-\sigma_2^2) k_j+\sigma_2k_{j-1}\right] +\sigma_1^{m} \sigma_2k_{m-1}.\]
Since $h-gf \in \Q[\sigma_1]$, we must have that $n \geq m$ and for each $j\leq n$, $\ell_j$ differs from the coefficient of $\sigma_1^j$ in $gf$ by a constant. Conditions (i) through (iv) follow.
\endproof

Lemma \ref{crit} has a couple of useful consequences.  

\begin{cor}\label{consequences}
For $g$ and $h$ as given in  Lemma \ref{crit} and a fixed $a<m$, suppose that for every $j\in \{0,1,\hdots, a-1\}$, $\deg(\ell_{a-j}) < 2 + j+\deg(k_a)$.   Then
\begin{enumerate}
\item For each $j\in\{0, 1,\hdots, a\}$, $\deg(k_{a-j}) = j+\deg(k_a)$ and
\item $\deg(\ell_0) = 2+a+\deg(k_a)$.
\end{enumerate}
\end{cor}

%\jason{***I switched the roles of $j$ and $a-j$ above, which seems to help in both the proof and the application.***}

\proof We establish conclusion (1) by induction on $j$, using condition (ii) of Lemma \ref{crit}. In the base case $j=0$ for instance, since $\deg(\ell_a) < \deg(k_a)+2$ the top-degree terms of $\sigma_2^2k_a$ and $\sigma_2k_{a-1}$ must cancel. This implies in particular that $\deg(k_{a-1}) = \deg(k_a)+1$. The induction step is similar, and conclusion (2) here then follows from (1) using condition (i) from Lemma \ref{crit}. \endproof

%\jason{***Changes below either fill in detail or make the strategy more explicit.***}

\begin{proof}[Proof of Prop. \ref{carrie}]
First suppose that $(p,q) \in R_1$, ie.~that $0\leq p\leq q$, and $q\neq 0$.  Put $h=h_{(p,q)}$ as given in Equation (\ref{hR1}).  The case $p=q=1$ was excluded, and in the case $p=0$, $q=1$ a quick calculation shows for $(\sigma_1, \sigma_2)$ satisfying $h$ and $f$ that $\sigma_1=1$ and $\sigma_2=\pm 1$.  So the proposition holds for this (non-hyperbolic) surgery, and we may assume $q \geq 2$.  Now, as in Lemma \ref{crit}, write $h=\sum_0^n \sigma_1^j \ell_j$.  Using Lemma \ref{able} and the definition of $h$, we find that
\begin{enumerate}[label=(\roman*)]
\item $n=2q-1$,
\item $\deg(\ell_{2q-1})=\deg(\ell_{2q-2}) = 0$,
\item $\deg(\ell_{2q-3})=\deg(\ell_{2q-4})=1$, and %\jason{not sure why the latter assertion here was omitted, but I \textit{think} it's correct }\eric{I think I had removed it because I didn't think it was used in arguments that follow.  I think it is correct too, so please keep it if you think we need it.  Wow, that's a lot of thinks!}
% \jason{I think you are right, but I am still sort of inclined to keep it because of the pattern common to (i) (ii) and (iii) }and
\item for $0\leq j \leq q-2$, both $\deg(\ell_{2q-3-2j})$ and $\deg(\ell_{2q-4-2j})$ are at most $j+1$.
\end{enumerate}
If there were $g \in \Q[\sigma_1, \sigma_2]$ with $h-gf \in \Q[\sigma_1]$, then writing $g=\sum_0^{m-1} \sigma_1^j k_j$, by Lemma \ref{crit}(iii) and (iii) above we would have $m=2q-3$ and $\deg(k_{m-1})=0$. In this case Corollary \ref{consequences} would apply with $a=m-1$, by (iii) and (iv) above, and its conclusion (2) would give $\deg(l_0) = 2q-2$. But $\deg(\ell_0) \leq q-1$ by (iv) above, so we would have $q-1\geq  2+a = 2q-2$, contradicting that $q\geq 2$.  This completes the proof for $(p,q) \in R_1$.

In the case that $(p,q) \in R_2$, ie.~that $0<q<p$, we will exploit a symmetry of \Cref{surgery eqn} to trade out $h_{(p,q)}$ from \Cref{hR2} for a more convenient polynomial. If $(w,x)\in V_0$ satisfies \Cref{surgery eqn} then for $w' = 1-w$ and $x' = 1-x$, $(w',x')$ satisfies the same equation but with $p$ replaced by $-p$. Proceeding as in the proof of Lemma \ref{cable} yields the polynomial
\[ \sum_{j=0}^{p-q} { p-q \choose j} (-wx)^j  \, s_{p+q-1-j}. \]
with $(w',x')$ as a root. Thus for $t_d$ as in \Cref{able}, $(\sigma_1',\sigma_2') = \phi(w',x')$ must satisfy:
\[ h = \sum_{j=0}^{p-q} { p-q \choose j} (-\sigma_2)^j  \, t_{p+q-1-j}.\] 
Here $\sigma'_1=2-\sigma_1$ and $\sigma'_2=1-\sigma_1+\sigma_2 = \sigma_2^{-1}$, so $\Q(\sigma_1) = \Q(\sigma_2)$ if and only if $\Q(\sigma'_1) = \Q(\sigma'_2)$. We may therefore ignore the discrepancy and assume that $(\sigma_1, \sigma_2)$ satisfies $h$.

Writing $h = \sum_{j=0}^{p+q-1}\sigma_1^j\ell_j$, for $0 \leq j \leq p-q$ the leading term of $\ell_{p+q-1-j}$ is $(-1)^{j} {p-q \choose j } \sigma_2^j$. On the other hand, for $a=p+q-1 - (p-q+1) = 2q-2$ and any $j'\in \{0,1,\hdots,a\}$, $\ell_{a-j'}$ has degree at most $p-q + \lceil j'/2\rceil$. (The highest-degree terms of these $\ell_{a-j'}$ all come from either the $t_{2q-1}$ or the $t_{2q}$ term of $h$.) We intend to apply \Cref{consequences} with $a=2q-2$. For this we claim that if there exists a polynomial $g$ such that $h-gf\in\mathbb{Q}(\sigma_1)$, then writing $g = \sum_{j=0}^m \sigma_1^j k_j$ with $k_j\in\mathbb{Q}[\sigma_2]$ for all $j$, $k_a$ has degree $p-q-1$.

%Let $a=2q-2$.  By (\ref{hR2}) and Lemma \ref{able}, 
%\begin{enumerate}[label=(\roman*)]
%\item If $0 \leq j \leq a+1$ then $\deg(\ell_j) \leq p-1-\lfloor j/2 \rfloor$
%\end{enumerate}
For such a polynomial $g$, condition (iii) of Lemma \ref{crit} implies that $m = p+q-2$, and since $\ell_{p+q-2}-\sigma_2 k_{p+q-3} $ is constant, that $k_{p+q-3} = {p+q\choose 1}$. Now using \Cref{crit}(ii) and the fact that the leading term of $\ell_{p+q-1-j}$ is $(-1)^{j} {p-q \choose j } \sigma_2^j$ for $1 \leq j \leq p-q$, induction (with $j=1$ as the base case) shows that the leading term of $k_{p+q-2-j}$ is 
\[ \sigma_2^{j-1} \, \sum_{i=1}^{j} (-1)^i {p-q \choose i}.\]
Taking $j=p-q$ and recalling that an alternating sum of binomial coefficients is zero, we find that the leading term of $k_{a}$ is $-\sigma_2^{p-q-1}$.  Thus its degree is $p-q-1$ as claimed, and it follows that $\deg (\ell_j) < 2+a-j+\deg(k_a)$ whenever $1 \leq j \leq a$.  Therefore by Corollary \ref{consequences}, \[p-1 \geq \deg(\ell_0) = 2+a+\deg(k_a) = p+q-1.\]
This is not possible because $q$ is positive.
\end{proof}

We now directly address the cases $q=0$ and $p=2q$ excluded in \Cref{aiu}.

\begin{lem}\label{gable} The $(1,0)$-filling of $c_2$ is non-hyperbolic; $\sigma_1 = 1+\sqrt{-3}$ at the $(2,0)$ filling; and $\sigma_1 = 1+2i$ at the $(3,0)$-filling. For each $p> 3$, $\mathbb{Q}(\sigma_1)$ is not quadratic imaginary at the $(p,0)$ filling of $c_2$.

The $(2,1)$-filling of $c_2$ is not hyperbolic. Otherwise, if $q>1$, $\mathbb{Q}(\sigma_1)$ is not quadratic imaginary at the $(2q,q)$-filling of $c_2$.\end{lem}

\begin{proof} 
Assume that $(w,x) \in V_0$ corresponds to $(p,q)$ filling of $c_2$.  To start, suppose that $p>0$ and $q=0$.  From the expressions for $\mu(\mathfrak{m}_2)$ above \Cref{surgery eqn}, \Cref{surgery eqn}
is equivalent to  $ \mu(\mathfrak{m}_2)^p =1$.
Hence, $\mu(\mathfrak{m}_2)  = \cos(2\pi/p)\pm i\sin(2\pi/p)$ is a primitive $p$th root of unity $\zeta_p$.  Using the alternative descriptions $\mu(\mathfrak{m}_2) = x^2(1-w)^2 = 1/(w^2(1-x)^2)$ given above  \Cref{surgery eqn} and taking a square root yields:
\[ x(1-w)+ w(1-x) = \sigma_1 - 2\sigma_2 = \zeta_{2p} + \zeta_{2p}^{-1} = 2\cos(\pi/p). \]
Since $\sigma_1 = \sigma_2 + 1 - 1/\sigma_2$, this shows that $\cos(\pi/p)\subset\mathbb{Q}(\sigma_2)$.  If we apply $\psi_\ast$ to this formula and clear denominators, we obtain
\[ 0 = \sigma_2^2 + (2\cos(\pi/p) - 1)\sigma_2 + 1. \]
So $\mathbb{Q}(\sigma_2)$ has degree at most two over $\mathbb{Q}(\cos(\pi/p))$. At any hyperbolic Dehn filling of $c_2$ the degree is exactly two, since in this case $\sigma_1\subset\mathbb{Q}(\sigma_2)$ is non-real. This rules out $p=1$, where the polynomial above has two real roots. At $p=2$ we have $\sigma_2 = \frac{1}{2}(1\pm\sqrt{-3})$ and $\sigma_1 = 2\sigma_2$. And at $p = 3$ we have $\sigma_2 = \pm i$ and hence $\sigma_1 = 1\pm 2i$. These correspond to positively and negatively oriented structures on the corresponding surgeries.

In the cases $p=4,5,6$ where $\mathbb{Q}(\cos(\pi/p))$ has degree two over $\mathbb{Q}$, the quadratic formula yields an explicit formula for $\sigma_2$, hence an explicit formula for $\sigma_1 = 2(1+\cos(\pi/p))$.  This can be used to show that $\mathbb{Q}(\sigma_1)=\mathbb{Q}(\sigma_2)$ with degree four over $\mathbb{Q}$. 

For $p>6$, we observe that $\varphi(2p)>4$, where $\varphi$ is the Euler $\varphi$-function. Appealing to standard properties of cyclotomic fields, it then follows that $[\Q(\cos(\pi/p)):\Q]>2$. We may assume that the degree of $\Q(\sigma_2)$ over $\Q(\cos(\pi/p))$ is two, since otherwise $\Q(\sigma_2)$ is real and so $\Q(\sigma_1)\subset \Q(\sigma_2)$ cannot be quadratic imaginary. But then it follows that $[\Q(\sigma_2):\Q]=[\Q(\sigma_2):\Q(\cos(\pi/p))][\Q(\cos(\pi/p)):\Q]>4$, so $[\Q(\sigma_1):\Q]>2$ since the degree of $\Q(\sigma_2)$ over $\Q(\sigma_1)$ is at most two. Thus $\Q(\sigma_1)$ cannot be quadratic imaginary.

It remains to consider the surgeries $(2q,q)$.   By Proposition \ref{carrie}, $\Q(\sigma_1)=\Q(\sigma_2)$, so we need only show that $\Q(\sigma_2)$ is not quadratic imaginary.  
First, $\psi_\ast h_{(2,1)} = (2-\sigma_2)\sigma_2$ has only rational solutions, so the $(2,1)$ filling is not hyperbolic.  Therefore, we may assume that $q\geq 2$.  
Since the surgery coefficients are taken with respect to the basis $\{\mathfrak{m}_2,\mathfrak{m}_2^{-1}\mathfrak{l}_2\}$, \Cref{surgery eqn} is equivalent to $\mu(\mathfrak{m}_2 \mathfrak{l}_2)^q=1$ in this case.  In the function field $\C[w,x]/(P)$ for $V_0$, 
\[ \mu(\mathfrak{m}_2 \mathfrak{l}_2) = \left( \frac{1-w}{w(1-x)^2} \right)^2 = \left( \frac{x(1-w)^2}{1-x} \right)^2.\]
So as in the previous case,
\[ 2 \cos(\pi/q) = \frac{x(1-w)^2}{1-x} + \frac{w(1-x)^2}{1-w}.\]
Using $\phi_\ast$ and $\psi_\ast$ as above, we find that the polynomial
\begin{align} \label{cosine} 2 \cos(\pi/q)  +\sigma_2^4-3\sigma_2^3+\sigma_2^2+2\end{align}
must be zero.  Therefore, $\Q(\sigma_2)$ is an extension of $\Q(\cos(\pi/q))$ of degree at most 4.

In the cases $q=2,3$ in which $\cos(\pi/q)\in\mathbb{Q}$, it is straightforward to check that the polynomial (\ref{cosine}) is irreducible over $\mathbb{Q}(\cos(\pi/q))$, and it follows that $\Q(\sigma_2)$ is not quadratic imaginary.  In the cases $q = 4,5,6$, $\mathbb{Q}(\cos(\pi/q))$ is quadratic over $\mathbb{Q}$. If $\Q(\sigma_2)$ is quadratic, then $\Q(\sigma_2)=\Q(\cos(\pi/q))\subset \R$, so it is not imaginary. Thus $\Q(\sigma_2)$ cannot be quadratic imaginary. When $q>6$ we once again have $\phi(2q)>4$, so that $[\Q(\cos(\pi/q)):\Q]>2$ and consequently $\Q(\sigma_2)$ cannot be quadratic since it contains $\Q(\cos(\pi/q))$.

\end{proof}

\begin{cor}\label{gary} The only fillings of $c_2$ for which $\mathbb{Q}(\sigma_1)$ is quadratic imaginary are the $(2,0)$ filling, where $\sigma_1 = 1+\sqrt{-3}$, and the $(3,0)$ filling, where $\sigma_1 = 2i+1$.\end{cor}

\begin{proof} By \Cref{gable}, $\mathbb{Q}(\sigma_1)$ is not quadratic imaginary at any $(2q,q)$-fillings of $c_2$, and the only $(p,0)$-fillings for which it is are $p=2,3$, as described above where $\sigma_1 = 1+\sqrt{-3}$. In all remaining cases, $\mathbb{Q}(\sigma_1) = \mathbb{Q}(\sigma_2)$ by \Cref{carrie}, and $\sigma_2$ is a unit algebraic integer by \Cref{aiu}. We thus need only consider the cases that $\sigma_1$ lies in $\mathbb{Q}(i)$ or $\mathbb{Q}(\sqrt{-3})$, as these are the only quadratic imaginary fields with unit algebraic integers not equal to $\pm 1$.

Because $\sigma_1$ is not real at any hyperbolic Dehn filling, we can always trade $(w,x)$ for $(\bar{w}, \bar{x})$ to assume the imaginary part of $\sigma_1$ is positive. %Moreover, given the observation below \Cref{act}, we need only consider $\sigma_2$ for which $\sigma_1$ lies in the upper half-plane, since we may always trade $(x,w)$ for $(\bar{x},\bar{w})$ to make the imaginary part of $\sigma_1 = x+w$ non-negative, and $\sigma_1$ is not real at any hyperbolic Dehn filling. 
This leaves three cases to consider: $\sigma_2 = \zeta_6$ or $\zeta_6^2$, which each yield $\sigma_1 = \sigma_2+1-1/\sigma_2 = 1+\sqrt{-3}$; and $\sigma_2 = i$ which gives $\sigma_1 = 2i+1$. From the proof of \Cref{gable} we see that the cases $\sigma_2 = \zeta_6$ and $\sigma_2 = i$ correspond to the $(2,0)$- and $(3,0)$-fillings of $c_2$, respectively. The final case $\sigma_2 = \zeta_6^2$ corresponds to the complete structure of $S^3-6_2^2$, where $x = w = \zeta_6$.\end{proof}

\begin{figure}
\includegraphics[scale=.60]{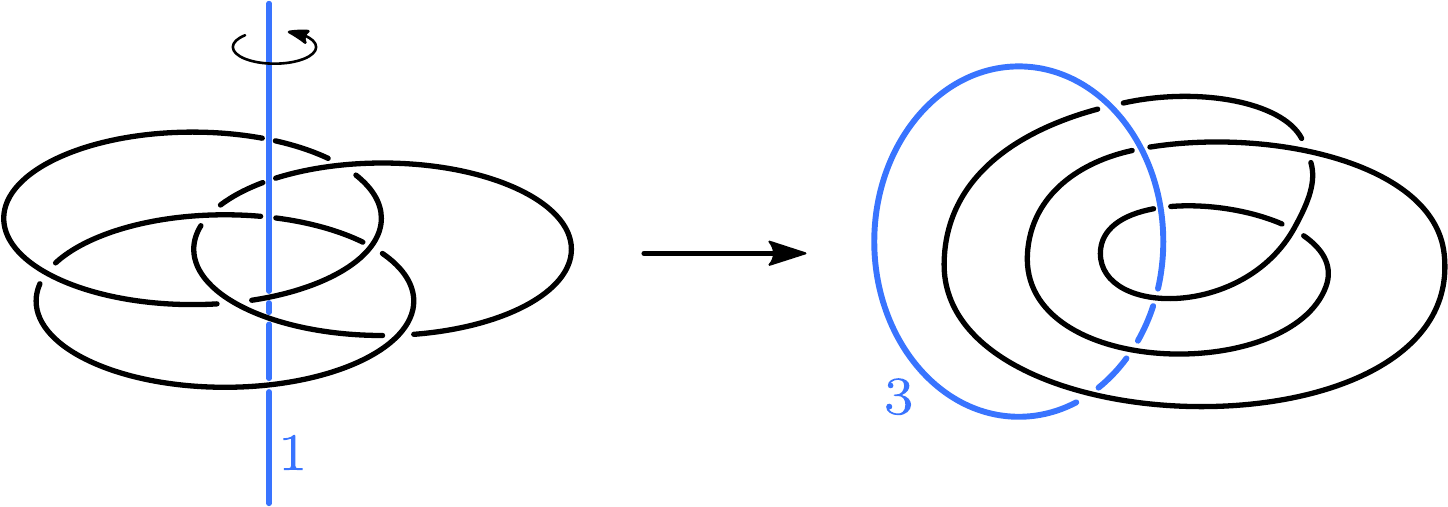}
\caption{\label{fig:borrothreefold} $(3,0)$ filling of $6^2_2$ as $3$-fold quotient of the Borromean rings link complement.}
\end{figure}

\begin{cor}\label{arith}
The $(2,0)$ and $(3,0)$ fillings of $c_2$ are the only arithmetic orbifolds obtained from $6^2_2$ by Dehn filling $c_2$. 
\end{cor}
\begin{proof}
If $(p,q)$ filling of $c_2$ produces an arithmetic orbifold then by Proposition 4.4(a) of \cite{NeumReid}, the invariant trace field of the orbifold is quadratic imaginary.  Since the cusp field can't be $\Q$, Proposition 2.7 of \cite{NeumReid} implies $\mathbb{Q}(\sigma_1)$ is quadratic imaginary as well. So, \Cref{gary} implies that $(2,0)$ and $(3,0)$ are the only possible candidates for arithmeticity. The figure-eight knot complement is a two-fold cyclic cover of the $(2,0)$-filling of $c_2$, and as shown in \Cref{fig:borrothreefold}, the Borromean rings complement is a three-fold cyclic cover of the $(3,0)$-filling.  It follows that these fillings are arithmetic, as the covering knot and link complements are.
\end{proof}
%}

\begin{proof}[Second proof of \Cref{sixtwotwo}] Any manifold covering a filling of one cusp of $\Sth\setminus 6^2_2$ covers a filling of the cusp $c_2$, since $\Sth\setminus 6^2_2$ has a cusp-exchanging involution. We have already observed that the figure-eight knot complement covers the $(2,0)$-filling of $c_2$. The $(3,0)$-filling is not covered by a knot complement since the components of $6^2_2$ have linking number three, see \Cref{prop:linking_number}. 

For any other $(p,q)$, if the $(p,q)$-filling of $c_2$ were covered by a knot complement $M$ with hidden symmetries then the (unique) minimal orbifold in the commensurability class of this (non-arithmetic) filling would have a rigid cusp, since $M$ covers an orbifold with a rigid cusp by \cite[Prop.~9.1]{NeumReid}. It would follow that the cusp field $\mathbb{Q}(\sigma_1)$ was $\mathbb{Q}(i)$ or $\mathbb{Q}(\sqrt{-3})$ at this filling, contradicting \Cref{gary}.
\end{proof}

We now make precise the relationship between $\mathbb{Q}(\sigma_2)$ and the invariant trace field for hyperbolic Dehn fillings of $\Sth\setminus 6^2_2$, and thus in consequence, between the cusp and invariant trace fields of these manifolds. %For this argument, we assume the reader is familiar with degree-one ideal triangulations in the sense of Neumann--Yang, \cite[\S 2.1]{NeumYang}. \neil{The properties relevant here are that every cusped hyperbolic 3-manifold $M$ admits a geometric cell-decomposition that can be refined to an ideal triangulation of $M$ that is semi-geometric, ie. all tetrahedra have dihedral angles which are non-negative and at least one tetrahedron has all positive angles. *** Actually, I am back to being confused. In the argument below, we are using the degree-one triangulation for the (p,q) Dehn filling of the $6^2_2$ complement coming from deforming the complete structure. This is triangulation is does not fit the Epstein-Penner picture right? For that don't we need genuine complete structure? It's possible that this has been extended somewhere, but I am not sure where. *** }\jason{The E-P thing is an aside for us; we don't need to appeal to it. I'll jump on the file later and try to clarify the proof strategy.}

\begin{figure} 
\begin{tikzpicture}
\begin{scope}[xscale=1.875,yscale=1.125]
\draw  (-.3,.5)--(.3,.5);
\draw [dashed](-.65,1)--(.65,1);
\draw (-.3,.5)--(-.65,1);
\draw (.3,.5)--(.65,1);
\draw (-.65,1)--(-.85,1.65);
\draw (.85,1.65)--(.65,1);
\draw (0,3)--(-.65,1);
\draw (0,3)--(-.85,1.65);
\draw (0,3)--(.85,1.65);
\draw (0,3)--(.65,1);
\draw[dashed] (-.85,1.65)--(.85,1.65);
\draw [dashed](-.65,1)--(.3,.5);
\draw [dashed](-.85,1.65)--(.65,1);
\draw (0,3)--(-.3,.5);
\draw (0,3)--(.3,.5);
\node at (-.35,.33){\tiny{$v_0$}};
\node at (.35,.33){\tiny{$v_1$}};
\node at (-.80,1){\tiny{$v_6$}};
\node at (.80,1.05){\tiny{$v_2$}};
\node at (-1,1.65){\tiny{$v_5$}};
\node at (1, 1.65){\tiny{$v_3$}};
\node at (0,3.1){\tiny{$v_4$}};
\node at (-0.25,1.93){\small{$A$}};
\node at (0.25,1.93){\small{$A'$}};
\node at (-0.5,2){\small{$B$}};
\node at (0.5,2){\small{$B'$}};
\node at (0,1.89){\small{$C$}};
\node at (.67, 2.8){\small{$C'$}};
\draw (.55,2.8) arc  (65:116:.45);
\draw[densely dotted,->] (.15,2.79) arc (116:143:.45);
\draw[densely dotted,<-] (-.3,.63) arc (65:85:.45);
\draw[] (-.46,.67) arc (85:115:.45);
\draw[densely dotted,<-] (-.57,1.38) arc (65:95:.45);
\draw[] (-.8,1.42) arc (95:135:.45);
\node at (-1.15, 1.35){$E'$};
\node at (-.8, .65){$D$};
\draw (1.15,1.38) arc  (65:112:.45);
\draw[densely dotted,->] (.78,1.39) arc (112:138:.45);
\node at (1.24, 1.33) {$D'$};
\draw (.82,.82) arc  (65:100:.45);
\draw[densely dotted,->] (.53,.85) arc (100:125:.45);
\node at (.88, .75) {$E$};

\end{scope}
\end{tikzpicture}
\caption{$6^2_2$ as quotient of ideal decahedron}
\label{deca}
\end{figure}
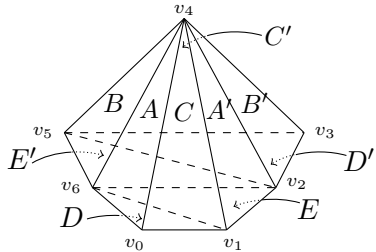

\begin{prop}\label{trfld}
For an orbifold $M$ obtained by $(p,q)$-hyperbolic Dehn filling on $c_2$, the invariant trace field and trace field of $M$ both equal $\Q(\sigma_2)$. 
\end{prop}

Our strategy of proof, which again echoes that of Neumann--Reid in \cite{NeumReid} (see Theorem 6.2 there), is as follows. For an orbifold $M$ as above, corresponding to a point $(w,x)\in V_0$, we use the parameters $x$ and $w$ to arrange isometric copies of the tetrahedra decomposing $\Sth\setminus 6^2_2$ in $\mathbb{H}^3$ so that their union is the polyhedron of \Cref{deca}. Taking the face-pairings of this polyhedron that yield $\Sth\setminus 6^2_2$ as generators of its fundamental group, we thus obtain explicit descriptions (in terms of $x$ and $w$) of their holonomy representations determined by the incomplete hyperbolic structure whose completion is $M$. This representation factors through a surjection to $\pi_1^{\mathit{orb}}(M)$, so the face-pairing generators generate the quotient group, and we can obtain the trace field of $M$ in terms of their traces and those of their products.

\begin{proof}
We may recover the face pairings of the triangulation $\mathcal{T}_4$ of $\Sth\setminus 6^2_2$ from the cusp triangulation pictures in \Cref{fig:cusp_triangs_1}. Assembling its four ideal tetrahedra so that they share an ideal vertex corresponding to the cusp $c_1$, with a horospherical cross-section around this ideal vertex as on the left side of \Cref{fig:cusp_triangs_1}, produces the ideal decahedron of \Cref{deca}. We produce $\Sth-6^2_2$ by identifying faces in pairs: $A$ with $A'$ and so on, so that pairings of faces containing the ideal vertex $v_4$ take it to itself, the pairing of $D$ with $D'$ takes $v_0$ to $v_3$, and that of $E$ with $E'$ takes $v_1$ to $v_6$. Then $v_4$ corresponds to $c_1$ and all other vertices to $c_2$. 

For any $(x,w)\in V_0$, the edge parameters $x$, $y$, $z$ and $w$ labeled in \Cref{fig:cusp_triangs_1} determine the tetrahedra comprising the decagon up to isometry, where $z = 1/(1-w)$ and $y=xw(1-w) = 1/(1-x)$ (see \Cref{defection}). Embedding the resulting decahedron in $\mathbb{H}^3$ with $v_0 = 0$, $v_1=1$ and $v_4 = \infty$ thus yields $v_3=1+x+w$, $v_5=x+w$, and $v_6=x$. This in turn determines isometries $a$, $b$, $c$, $d$, and $e$ realizing the face pairings sending $A$ to $A'$, $B$ to $B'$, $C$ to $C'$, $D$ to $D'$, and $E$ to $E'$, respectively. In particular, we have the matrix representatives below in $\mathrm{SL}(2,\mathbb{C})$. (We have used Mathematica \cite{Mathematica} to assist with some computations here.)
\begin{align}
a=\begin{pmatrix}1 & 1\\
0& 1
\end{pmatrix}, \quad e=\begin{pmatrix} w+wx+x^2-wx^2 & -w-wx-x^3+wx^3\\
w+x-wx & -w-x^2+wx^2\end{pmatrix}.
\end{align}

%One can see the four regular ideal tetrahedra of the triangulation $\mathcal{T}_4$ in the ideal decahedron as $[v_0, v_4,v_1, v_6]$, $[v_6,v_4,v_1,v_2]$, $[v_2,v_4,v_5,v_6]$, $[v_3,v_4,v_5,v_2]$. The edge invariants of edges $[v_0,v_4]$ in $[v_0, v_4,v_1, v_6]$, $[v_6,v_4]$ in  $[v_6,v_4,v_1,v_2]$, $[v_2,v_4]$ in  $[v_2,v_4,v_5,v_6]$, and $[v_3,v_4]$ in  $[v_3,v_4,v_5,v_2]$ are respectively $x$, $y$, $z$, and $w$. So, for the incomplete structure corresponding to $M$, we can assume that $v_0=0$, $v_1=1$, $v_2=1+x$, $v_3=1+x+w$, $v_4=\infty$, $v_5=x+w$, $v_6=x$. Let $a$, $b$, $c$, $d$, and $e$ denote the face pairing isometries sending respectively $A$ to $A'$, $B$ to $B'$, $C$ to $C'$, $D$ to $D'$, and $E$ to $E'$.

For each such $(x,w)\in V_0$, the face-pairing isometries generate the image of $\Gamma \doteq \pi_1(\Sth\setminus 6^2_2)$ under the holonomy representation determined by the hyperbolic structure on $\Sth\setminus 6^2_2$ corresponding to $(x,w)$. At the complete structure $x = w = \frac{1}{2}(1+\sqrt{-3})$, the Poincar\'e polyhedron theorem yields a presentation for $\Gamma$:
\[ \Gamma \cong \langle a, e \,|\,ae^2 a^{-1}e^{-1}a^{-1}e^2=e^2a^{-1}e^{-1}a^{-1}e^2a \rangle. \]
(Specifically, it gives $b=a$, $[a,c]=1$, $d = e^2$, and $c = db^{-1}e^{-1}a^{-1}d$.) 

Now fix $(x,w)\in V_0$ corresponding to a $(p,q)$-hyperbolic Dehn filling of $c_2$, and let $\Gamma_{p,q}$ be the orbifold fundamental group of the resulting hyperbolic orbifold $M$. The corresponding holonomy representation factors through a surjection $\Gamma\to\Gamma_{p,q}$, so as a Kleinian group in $\mathrm{PSL}(2,\mathbb{C})$, $\Gamma_{p,q}$ is generated by the cosets of the matrices $a$ and $e$ described above. The traces of $a$, $e$, and $ae$ thus generate the trace field of $\Gamma_{p,q}$ as an extension of $\mathbb{Q}$ (see equation 3.25 in \cite{MR03}). These are respectively $2$, $wx = \sigma_2$, and $w+x = \sigma_1$, so since $\sigma_1\in\mathbb{Q}(\sigma_2)$ this is the trace field of $\Gamma_{p,q}$.

%Let $\Gamma_{p,q}$ denote the fundamental group of the hyperbolic manifold $M$ obtained by $(p,q)$ Dehn filling on $c_2$. By the Poincar\'e Polyhedron Theorem (see for example \cite{epstein-petronio}), we see that $\Gamma_{p,q}=<a, e:ae^2 a^{-1}e^{-1}a^{-1}e^2=e^2a^{-1}e^{-1}a^{-1}e^2a >$. Using \will{Mathematica} \cite{Mathematica}, we write $a$ and $e$ as \neil{(cosets of)} matrices in $SL(2,\C)$:

The invariant trace field of $\Gamma_{p,q}$ contains the trace of $e^2$ by definition. As this is $\sigma_2^2-2$, the invariant trace field contains $\sigma_2^2$. By Proposition 2.7 of Neumann-Reid \cite{NeumReid} the cusp field of $M$, which here is $\Q(\sigma_1)$, is a subfield of the invariant trace field of $M$. So, $\sigma_2=\frac{\sigma_2^2-1}{\sigma_1-1}$ is contained in the invariant trace field.  Since the invariant trace field is contained in the trace field, it is also $\Q(\sigma_2)$.
\end{proof}

\begin{cor}\label{gareth} For each $(p,q)\in\mathbb{Z}^2$ such that $q\ne 0$, if $M$ is obtained from $\Sth\setminus 6^2_2$ by $(p,q)$-hyperbolic filling of one cusp then the invariant trace field of $M$ equals its cusp field.\end{cor}

\begin{proof}If $M$ is the hyperbolic orbifold obtained by $(p,q)$ Dehn filling on one cusp of $6^2_2$, where $q\ne 0$, then by \Cref{carrie}, $\Q(\sigma_1)=\Q(\sigma_2)$. Now, since the cusp field of $M$ is $\Q(\sigma_1)$, and by \Cref{trfld} above the invariant trace field of $M$ is $\Q(\sigma_2)$, we conclude that the invariant trace field is equal to the cusp field for $M$. \end{proof}

The polyhedra decomposing $\Sth\setminus 6^2_2$ determine a \textit{degree-one triangulation}, in the sense of Neumann--Yang \cite[\S 2.1]{NeumYang}, for the surgered orbifold $M$. For a cusped hyperbolic 3-manifold with a genuine ideal triangulation, the \textit{shape field} generated by the tetrahedral parameters equals the invariant trace field by \cite[Theorem~2.4]{NeumReid}. But we conclude by observing:

\begin{remark}
The $(4,3)$ filling on $c_2$ in $\Sth \setminus 6^2_2$ results in the manifold $m121$ in the SnapPy census, which has a genuine ideal triangulation by five tetrahedra. The shape field of this triangulation, and therefore the invariant trace field of $m121$, is a degree $5$ extension of $\Q$, whereas the shape field $\mathbb{Q}(x,w)$ of the degree-one triangulation it inherits from $\Sth\setminus 6^2_2$ has degree $10$ over $\Q$. Thus the tetrahedral shape field of a degree-one ideal triangulation of a cusped hyperbolic manifold may in general be strictly larger than its invariant trace field.
\end{remark}

%\christian{Do we want a concluding paragraph and/or remark for this section that ties in Corollary 3.17 as capping off another proof of our main theorem for this section? Something along the lines of ``\cite[Corollary 2.2]{ReidWalsh} and the recent work of Hoffman \cite{Hoffman_rigid_cusps} combined show that any hyperbolic knot complement that admits hidden symmetries has cusp field $\mathbb{Q}(\sqrt{-3})$. Thus, \Cref{gary} now implies \Cref{sixtwotwo}.''}

%% file: LimitOrbHiddenSymms_arxiv_v2.bbl
\begin{thebibliography}{10}

\bibitem{Adams_small_vol}
Colin~C Adams.
\newblock Noncompact hyperbolic 3-orbifolds of small volume.
\newblock {\em Topology}, 90:1--15, 1992.

\bibitem{Adams_limit}
Colin~C Adams et~al.
\newblock Limit volumes of hyperbolic three-orbifolds.
\newblock {\em Journal of Differential Geometry}, 34(1):115--141, 1991.

\bibitem{agol2000bounds}
Ian Agol.
\newblock Bounds on exceptional dehn filling.
\newblock {\em Geometry \& Topology}, 4(1):431--449, 2000.

\bibitem{AiRu}
I.~R. Aitchison and J.~H. Rubinstein.
\newblock Combinatorial cubings, cusps, and the dodecahedral knots.
\newblock In {\em Topology '90 ({C}olumbus, {OH}, 1990)}, volume~1 of {\em Ohio
  State Univ. Math. Res. Inst. Publ.}, pages 17--26. de Gruyter, Berlin, 1992.

\bibitem{BBoCWaGT}
Michel Boileau, Steven Boyer, Radu Cebanu, and Genevieve~S. Walsh.
\newblock Knot commensurability and the {B}erge conjecture.
\newblock {\em Geom. Topol.}, 16(2):625--664, 2012.

\bibitem{BoiBoCWa2}
Michel Boileau, Steven Boyer, Radu Cebanu, and Genevieve~S. Walsh.
\newblock Knot complements, hidden symmetries and reflection orbifolds.
\newblock {\em Ann. Fac. Sci. Toulouse Math. (6)}, 24(5):1179--1201, 2015.

\bibitem{BMP05}
Michel Boileau, Bernhard Leeb, and Joan Porti.
\newblock Geometrization of 3-dimensional orbifolds.
\newblock {\em Annals of mathematics}, pages 195--290, 2005.

\bibitem{CheDeMonda}
Eric Chesebro, Jason DeBlois, and Priyadip Mondal.
\newblock Generic hyperbolic knot complements without hidden symmetries.
\newblock Preprint. ar{X}iv:1910.04712.

\bibitem{CLO}
David~A. Cox, John Little, and Donal O'Shea.
\newblock {\em Ideals, varieties, and algorithms}.
\newblock Undergraduate Texts in Mathematics. Springer, Cham, fourth edition,
  2015.
\newblock An introduction to computational algebraic geometry and commutative
  algebra.

\bibitem{SnapPy}
Marc Culler, Nathan~M. Dunfield, Matthias Goerner, and Jeffrey~R. Weeks.
\newblock Snap{P}y, a computer program for studying the geometry and topology
  of $3$-manifolds.
\newblock Available at \url{http://snappy.computop.org} (DD/MM/YYYY).

\bibitem{sagemath}
The~Sage Developers.
\newblock {\em {S}ageMath, the {S}age {M}athematics {S}oftware {S}ystem
  ({V}ersion 7.5.1)}, 2017.
\newblock {\tt http://www.sagemath.org}.

\bibitem{dummitfoote2004}
David~Steven Dummit and Richard~M Foote.
\newblock {\em Abstract algebra}, volume~3.
\newblock Wiley Hoboken, 2004.

\bibitem{DM}
William~D. Dunbar and G.~Robert Meyerhoff.
\newblock Volumes of hyperbolic {$3$}-orbifolds.
\newblock {\em Indiana Univ. Math. J.}, 43(2):611--637, 1994.

\bibitem{Dunfield}
Nathan Dunfield.
\newblock Private communication.

\bibitem{KirbyList}
Rob~Kirby (Ed.).
\newblock Problems in low-dimensional topology.
\newblock In {\em Proceedings of Georgia Topology Conference, Part 2}, pages
  35--473. Press, 1995.

\bibitem{FPS19}
David Futer, Jessica~S. Purcell, and Saul Schleimer.
\newblock Effective bilipschitz bounds on drilling and filling.
\newblock {\em arXiv preprint arXiv:1907.13502}, 2019.

\bibitem{GMM2009minimum}
David Gabai, Robert Meyerhoff, and Peter Milley.
\newblock Minimum volume cusped hyperbolic three-manifolds.
\newblock {\em Journal of the American Mathematical Society}, 22(4):1157--1215,
  2009.

\bibitem{GMM2010mom}
David Gabai, Robert Meyerhoff, and Peter Milley.
\newblock Mom technology and hyperbolic 3-manifolds.
\newblock {\em The Tradition of Ahlfors-Bers, V. Contemp. Math}, 510:81--108,
  2010.

\bibitem{GMacMart}
F.~W. Gehring, C.~Maclachlan, and G.~J. Martin.
\newblock Two-generator arithmetic {K}leinian groups. {II}.
\newblock {\em Bull. London Math. Soc.}, 30(3):258--266, 1998.

\bibitem{haraway2019practical}
Robert Haraway~III.
\newblock Practical bounds for a {D}ehn parental test.
\newblock {\em Proceedings of the American Mathematical Society},
  147(1):427--442, 2019.

\bibitem{heard2005orb}
Damian Heard.
\newblock Orb.
\newblock {\em The computer program for finding hyperbolic structures on
  hyperbolic}, 2005.

\bibitem{HK2008shape}
Craig Hodgson and Steven Kerckhoff.
\newblock The shape of hyperbolic {D}ehn surgery space.
\newblock {\em Geom. Topol.}, 12(2):1033--1090, 2008.

\bibitem{HK2005universal}
Craig~D Hodgson and Steven~P Kerckhoff.
\newblock Universal bounds for hyperbolic {D}ehn surgery.
\newblock {\em Annals of Mathematics}, pages 367--421, 2005.

\bibitem{Hoffman_3comm}
Neil Hoffman.
\newblock Commensurability classes containing three knot complements.
\newblock {\em Algebr. Geom. Topol.}, 10(2):663--677, 2010.

\bibitem{Hoffman_rigid_cusps}
Neil~R. Hoffman.
\newblock Cusp types of quotients of hyperbolic knot complements.
\newblock \url{https://arxiv.org/abs/2001.05066}.

\bibitem{Hoffman_hidden}
Neil~R. Hoffman.
\newblock Small knot complements, exceptional surgeries and hidden symmetries.
\newblock {\em Algebr. Geom. Topol.}, 14(6):3227--3258, 2014.

\bibitem{HMW19}
Neil~R Hoffman, Christian Millichap, and William Worden.
\newblock Symmetries and hidden symmetries of $(\epsilon, d_l)$-twisted knot
  complements.
\newblock 2019.
\newblock \url{https://arxiv.org/abs/1909.10571}.

\bibitem{Mathematica}
Wolfram~Research{,} Inc.
\newblock Mathematica, {V}ersion 11.1.
\newblock Champaign, IL, 2017.

\bibitem{MacMat}
Melissa~L. Macasieb and Thomas~W. Mattman.
\newblock Commensurability classes of {$(-2,3,n)$} pretzel knot complements.
\newblock {\em Algebr. Geom. Topol.}, 8(3):1833--1853, 2008.

\bibitem{MR03}
Colin Maclachlan and Alan~W. Reid.
\newblock {\em The arithmetic of hyperbolic 3-manifolds}, volume 219 of {\em
  Graduate Texts in Mathematics}.
\newblock Springer-Verlag, New York, 2003.

\bibitem{Margulis1991}
G.~A. Margulis.
\newblock {\em Discrete subgroups of semisimple {L}ie groups}, volume~17 of
  {\em Ergebnisse der Mathematik und ihrer Grenzgebiete (3) [Results in
  Mathematics and Related Areas (3)]}.
\newblock Springer-Verlag, Berlin, 1991.

\bibitem{MartPet}
Bruno Martelli and Carlo Petronio.
\newblock {D}ehn filling of the ``magic'' 3-manifold.
\newblock {\em Comm. Anal. Geom.}, 14(5):969--1026, 2006.

\bibitem{Millichap}
Christian Millichap.
\newblock Mutations and short geodesics in hyperbolic 3-manifolds.
\newblock {\em Comm. Anal. Geom.}, 25(3):625--683, 2017.

\bibitem{MilliWord}
Christian Millichap and William Worden.
\newblock Hidden symmetries and commensurability of 2-bridge link complements.
\newblock {\em Pacific J. Math.}, 285(2):453--484, 2016.

\bibitem{NeumReid}
Walter~D. Neumann and Alan~W. Reid.
\newblock Arithmetic of hyperbolic manifolds.
\newblock In {\em Topology '90 ({C}olumbus, {OH}, 1990)}, volume~1 of {\em Ohio
  State Univ. Math. Res. Inst. Publ.}, pages 273--310. de Gruyter, Berlin,
  1992.

\bibitem{NR92b}
Walter~D. Neumann and Alan~W. Reid.
\newblock Notes on {A}dams' small volume orbifolds.
\newblock In {\em Topology '90 ({C}olumbus, {OH}, 1990)}, volume~1 of {\em Ohio
  State Univ. Math. Res. Inst. Publ.}, pages 311--314. de Gruyter, Berlin,
  1992.

\bibitem{NR_rigidity}
Walter~D. Neumann and Alan~W. Reid.
\newblock Rigidity of cusps in deformations of hyperbolic {$3$}-orbifolds.
\newblock {\em Math. Ann.}, 295(2):223--237, 1993.

\bibitem{NeumYang}
Walter~D. Neumann and Jun Yang.
\newblock Bloch invariants of hyperbolic {$3$}-manifolds.
\newblock {\em Duke Math. J.}, 96(1):29--59, 1999.

\bibitem{NZ}
Walter~D. Neumann and Don Zagier.
\newblock Volumes of hyperbolic three-manifolds.
\newblock {\em Topology}, 24(3):307--332, 1985.

\bibitem{ReidFig8}
Alan~W. Reid.
\newblock Arithmeticity of knot complements.
\newblock {\em J. London Math. Soc. (2)}, 43(1):171--184, 1991.

\bibitem{ReidWalsh}
Alan~W. Reid and Genevieve~S. Walsh.
\newblock Commensurability classes of 2-bridge knot complements.
\newblock {\em Algebr. Geom. Topol.}, 8(2):1031--1057, 2008.

\bibitem{Rolfsen}
Dale Rolfsen.
\newblock {\em Knots and links}, volume~7 of {\em Mathematics Lecture Series}.
\newblock Publish or Perish, Inc., Houston, TX, 1990.
\newblock Corrected reprint of the 1976 original.

\bibitem{SW}
Makoto Sakuma and Jeffrey Weeks.
\newblock Examples of canonical decompositions of hyperbolic link complements.
\newblock {\em Japan. J. Math. (N.S.)}, 21(2):393--439, 1995.

\bibitem{Th_notes}
W.~P. Thurston.
\newblock The geometry and topology of 3-manifolds.
\newblock mimeographed lecture notes, 1979.

\end{thebibliography}
